\documentclass[a4paper]{amsart}

\usepackage[utf8]{inputenc}
\usepackage[T1]{fontenc}
\usepackage{amssymb,amsmath,amsthm, enumitem, mathrsfs, tikz-cd, url}
\usetikzlibrary{decorations.pathmorphing}
\usepackage{microtype} 
\usepackage[hidelinks]{hyperref}

\usepackage[top=3.75cm, bottom=3cm, left=3cm, right=3cm]{geometry}

\DeclareRobustCommand{\SkipTocEntry}[5]{} 

\setlist[enumerate]{itemsep=0pt, topsep=5pt, label={\roman*\textnormal{)}}, parsep=0pt,listparindent=\parindent}

\theoremstyle{plain}
\newtheorem{theorem}{Theorem}[section]
\newtheorem*{theorem*}{Theorem}
\newtheorem{lemma}[theorem]{Lemma}
\newtheorem{corollary}[theorem]{Corollary}
\newtheorem{proposition}[theorem]{Proposition}

\theoremstyle{definition}
\newtheorem{definition}[theorem]{Definition}
\newtheorem{conjecture*}{Conjecture}
\newtheorem{example}[theorem]{Example}
\newtheorem{examples}[theorem]{Examples}

\theoremstyle{remark}
\newtheorem{remark}[theorem]{Remark}
\newtheorem*{remark*}{Remark}
\newtheorem*{convention*}{Convention}

\newcommand{\NN}{\mathbb{N}}
\newcommand{\VV}{\mathbb{V}}
\newcommand{\ZZ}{\mathbb{Z}}

\newcommand{\bD}{\mathbf{D}}
\newcommand{\bL}{\mathbf{L}}
\newcommand{\bR}{\mathbf{R}}

\newcommand{\Ascr}{\mathscr{A}}
\newcommand{\Cscr}{\mathscr{C}}
\newcommand{\Bscr}{\mathscr{B}}
\newcommand{\Dscr}{\mathscr{D}}

\newcommand{\Asf}{\mathsf{A}}
\newcommand{\Csf}{\mathsf{C}}
\newcommand{\Isf}{\mathsf{I}}
\newcommand{\Ssf}{\mathsf{S}}
\newcommand{\Tsf}{\mathsf{T}}

\newcommand{\Acal}{\mathcal{A}}
\newcommand{\Fcal}{\mathcal{F}}

\newcommand{\Ical}{\mathcal{I}}
\newcommand{\Jcal}{\mathcal{J}}
\newcommand{\Kcal}{\mathcal{K}}
\newcommand{\Mcal}{\mathcal{M}}
\newcommand{\Ncal}{\mathcal{N}}
\newcommand{\Ocal}{\mathcal{O}}
\newcommand{\Pcal}{\mathcal{P}}

\newcommand{\kk}{\Bbbk}

\newcommand{\tE}{\widetilde{E}}
\newcommand{\tR}{\widetilde{R}}
\newcommand{\tOcal}{\widetilde{\mathcal{O}}}

\newcommand{\mapsfrom}{\mathrel{\reflectbox{\ensuremath{\mapsto}}}}
\newcommand\isoto{\stackrel{\sim}{\smash{\longrightarrow}\rule{0pt}{0.4ex}}}

\renewcommand{\phi}{\varphi}

\DeclareMathOperator{\op}{op}
\DeclareMathOperator{\id}{id}
\DeclareMathOperator{\Hom}{Hom}
\DeclareMathOperator{\RHom}{\mathbf{R}Hom}

\DeclareMathOperator{\Ext}{Ext}
\DeclareMathOperator{\cone}{cone}
\DeclareMathOperator{\Tot}{Tot}
\DeclareMathOperator{\rad}{rad}
\DeclareMathOperator{\Rees}{Rees}
\DeclareMathOperator{\gr}{gr}

\DeclareMathOperator{\Spec}{Spec}
\DeclareMathOperator{\Proj}{Proj}
\DeclareMathOperator{\rProj}{\underline{Proj}}
\DeclareMathOperator{\Bl}{Bl}
\DeclareMathOperator{\sHom}{\mathcal{H}\! \mathit{om}}
\DeclareMathOperator{\sEnd}{\mathcal{E}\! \mathit{nd}}
\DeclareMathOperator{\Sym}{Sym}

\DeclareMathOperator{\Obj}{Obj}
\DeclareMathOperator{\Mor}{\mathsf{Mor}}
\DeclareMathOperator{\Fun}{\mathsf{Fun}}
\DeclareMathOperator{\tw}{\mathsf{tw}}
\DeclareMathOperator{\Cube}{\mathsf{Cube}}
\DeclareMathOperator{\Glue}{\mathsf{Glue}}
\DeclareMathOperator{\Gac}{\mathsf{Gac}}
\DeclareMathOperator{\Filt}{\mathsf{Filt}}
\DeclareMathOperator{\Sch}{\mathsf{Sch}}
\DeclareMathOperator{\fSch}{\mathsf{fSch}}
\DeclareMathOperator{\dgCat}{\mathsf{dgCat}}
\DeclareMathOperator{\dgcat}{\mathsf{dgcat}}
\DeclareMathOperator{\Tri}{\mathsf{Tri}}
\DeclareMathOperator{\dgMod}{\mathsf{dgMod}}
\DeclareMathOperator{\Ac}{\mathsf{Ac}}
\DeclareMathOperator{\Res}{Res}
\DeclareMathOperator{\Ind}{Ind}
\DeclareMathOperator{\Mod}{\mathsf{Mod}}
\DeclareMathOperator{\grMod}{\mathsf{grMod}}
\DeclareMathOperator{\QCoh}{\mathsf{QCoh}}
\DeclareMathOperator{\Coh}{\mathsf{Coh}}
\DeclareMathOperator{\Ab}{\mathsf{Ab}}
\DeclareMathOperator{\Perf}{\mathsf{Perf}}
\DeclareMathOperator{\Thick}{\mathsf{Thick}}
\DeclareMathOperator{\Loc}{\mathsf{Loc}}
\DeclareMathOperator{\hflat}{\mathsf{h-flat}}
\DeclareMathOperator{\hflatperf}{\mathsf{h-flat-perf}}

\title{Categorical resolutions of filtered schemes}
\author{Timothy De Deyn}
\address{Departement Wiskunde, Vrije Universiteit Brussel, Pleinlaan 2, B-1050 Elsene}
\email{timothy.de.deyn@vub.be}
\address{Currently at: School of Mathematics and Statistics, University of Glasgow, University Place, G12
	8QQ Glasgow, United Kingdom}
\email{timothy.dedeyn@glasgow.ac.uk}
\thanks{This work was supported by the Research Foundation - Flanders (FWO) Ph.D.\ Fellowship fundamental research 95249.}

\begin{document}
\begin{abstract}
	We give an alternative proof of the theorem by Kuznetsov and Lunts, stating that any separated scheme of finite type over a field of characteristic zero admits a categorical resolution of singularities.
	Their construction makes use of the fact that every variety (over a field of characteristic zero) can be resolved by a finite sequence of blow-ups along smooth centres.
	We merely require the existence of (projective) resolutions. 
	To accomplish this we put the $\Acal$-spaces of Kuznetsov and Lunts in a different light, viewing them instead as schemes endowed with finite filtrations.
	The categorical resolution is then constructed by gluing together differential graded categories obtained from a hypercube of finite length filtered schemes.
\end{abstract}
\maketitle
{\footnotesize  \tableofcontents}
\section{Introduction}
\addtocontents{toc}{\SkipTocEntry}	
\subsection*{Main result and motivation}
In \cite{KuznetsovLunts}, Kuznetsov and Lunts prove the following theorem.
\begin{theorem*}
	Any separated scheme of finite type over a field of characteristic zero has a categorical resolution by a strongly geometric triangulated category.
	Moreover, if the scheme is proper, so is the resolving category.
\end{theorem*}
Roughly speaking, a strongly geometric categorical resolution of a scheme $X$, as in the theorem, is a `nice' embedding of the derived category $\bD(X):=\bD(\QCoh (X) )$ into a smooth triangulated category admitting a semi-orthogonal decomposition with as components derived categories of smooth algebraic varieties.
The construction in loc.\ cit.\ makes use of a strong version of Hironaka, namely, that every variety (over a field of characteristic zero) can be resolved by a finite sequence of blow-ups along smooth centres.
In this paper, we reprove this result in Theorem \ref{thm: main} without making use of the strong version of Hironaka.
We merely require the existence of (projective) resolutions. 
It is well known that such a projective resolution is a blow-up (at least when the target is quasi-projective over an affine Noetherian scheme\footnote{By some form of Chow's lemma \cite[\href{https://stacks.math.columbia.edu/tag/088U}{Lemma 088U}]{stacks-project} quasi-projectiveness can be achieved by blowing up.	}).

The motivation for weakening the assumptions on the resolutions of singularities used is so that the construction can be applied to orders over schemes (types of coherent sheaves of (non-commutative) algebras over schemes).
Therefore, this paper should be viewed as part one in a bigger picture.
Part two, which is in preparation, will apply the ideas in this paper to construct geometric categorical resolutions of orders over schemes. 	

For simplicity, let $X$ be a variety.
In \cite{KuznetsovLunts}, the authors consider a resolution of singularities of the form

\begin{equation}\label{eq: data KL}
	\begin{tikzcd}
		{Y_n} \\
		{Y_{n-1}} & {Z_{n-1}} \\
		\vdots \\
		{Y_1} & {Z_1} \\
		X & Z_0\rlap{ ,}
		\arrow[hook', from=5-2, to=5-1]
		\arrow["{f_0}"', from=4-1, to=5-1]
		\arrow[from=3-1, to=4-1]
		\arrow[from=2-1, to=3-1]
		\arrow[hook', from=2-2, to=2-1]
		\arrow[hook', from=4-2, to=4-1]
		\arrow["{f_{n-1}}"', from=1-1, to=2-1]
	\end{tikzcd}
\end{equation}
where each of the $f_i$ are blow-ups with smooth centre $Z_i$ and $Y_n$ is smooth. 
They then inductively construct categorical resolutions of $Y_i$ for $i=n,\dots,0$ (with $Y_0:=X$). 
The semi-orthogonal components of the categorical resolution are copies of $\bD(Z_i)$, with suitable multiplicity, and $\bD(Y_n)$.

In this work instead we obtain the categorical resolution from a diagram of the form
\begin{equation}\label{eq: data here}
	\begin{tikzcd}
		X_0 & {X_1} && {X_{n-1}} & {X_n}\\
		X & {S_1} & \dots & {S_{n-1}} & {S_n}\rlap{ ,}
		\arrow[hook', from=2-2, to=2-1]
		\arrow["{g_0}"', from=1-1, to=2-1]
		\arrow["{g_1}"', from=1-2, to=2-2]
		\arrow[hook', from=2-3, to=2-2]
		\arrow[hook', from=2-4, to=2-3]
		\arrow["{g_{n-1}}"', from=1-4, to=2-4]
		\arrow[hook', from=2-5, to=2-4]
		\arrow[equal, from=1-5, to=2-5]
	\end{tikzcd}
\end{equation}
where each $g_i$ is a blow-up with potentially singular centre $S_{i+1}$ and the reduced schemes associated to the $X_i$'s are smooth.
In this case, the semi-orthogonal components of the categorical resolution are copies of $\bD((X_i)_\mathrm{red})$, again with suitable multiplicities.
It is interesting to note that here the dimensions of the components strictly decrease with increasing $i$, and therefore, moreover, the integer $n$ appearing in the diagram is bounded above by the dimension of $X$.

Although the diagrams \eqref{eq: data KL} and \eqref{eq: data here} look similar, there is no analogue in our case for the inductive procedure in \cite{KuznetsovLunts} and we have to proceed differently.
The categorical resolution is constructed by gluing together a number of differential graded (dg) categories in a suitable way.
In the work of Kuznetsov and Lunts it is possible, through an inductive regluing procedure, to restrict to gluing pairs of dg categories in every step of the construction.
In our set-up this regluing procedure no longer works; we lose control over the gluing bimodule (there is no analogue of \cite[Lemma 6.9]{KuznetsovLunts}).
Therefore, we have to glue multiple dg categories simultaneously.
To this end, we consider the gluing of (punctured) hypercubes of dg categories, and construct a quasi fully faithful functor from the dg category situated at the puncture to the glued dg category.
It is this dg functor that yields the categorical resolution, the relevant condition for quasi fully faithfulness is that of the hypercube being `acyclic'.
The nonrational loci defined by Kuznetsov and Lunts yield acyclic squares.

In addition, for our approach, we put some of the results of Kuznetsov and Lunts into a slightly different framework.
Namely, throughout we work in the category of finite length filtered schemes (together with the corresponding pullback and pushforward functors between an adequate notion of quasi-coherent sheaves over them).
Finite length filtered schemes are equivalent to the $\Acal$-spaces considered in \cite{KuznetsovLunts}, but the latter seem less convenient to deal with functoriality issues, especially when we deal with morphisms that change the length of the filtration (the so-called `refinements', see \S\ref{subsubsec: gen mor} or the next subsection).
One should also note that the utility of filtrations could already be seen in the work of Kaledin and Kuznetsov \cite{KaledinKuznetsov}, but the approach here is different.

\addtocontents{toc}{\SkipTocEntry}	
\subsection*{In a nutshell}
Let $X$ be a variety and suppose $f:Y\to X$ is a resolution of singularities. 
When $X$ does not have rational singularities, i.e.\ 
\[
\bR f_* \Ocal_Y\not\cong\Ocal_X,
\]
the derived pullback $\bL f^*:\bD(X)\to\bD(Y)$ is not fully faithful, and so, does not give rise to a categorical resolution, a `nice' embedding of $\bD(X)$ into a smooth triangulated category (see Definition \ref{def: cat res} for a precise definition).
To obtain a fully faithful functor one uses the ingenious insight from Kuznetsov and Lunts to adequately \emph{modify} $\bD(Y)$.
This is done by \emph{gluing} $\bD(S)$ onto it, where $S\subset X$ is a so-called \emph{nonrational locus of $X$ with respect to $f$}, it satisfies 
\[
\bR f_* \Ical_{f^{-1}(S)}\cong \Ical_S.
\]
Here, $\Ical_?$ denotes the ideal of the closed subscheme $?\subseteq X$.
The resulting square
\[
\begin{tikzcd}[sep=1.5em]
	Y &  f^{-1}(S) \\
	X &  S
	\arrow["f"', from=1-1, to=2-1]
	\arrow[from=1-2, to=1-1]
	\arrow[from=1-2, to=2-2]
	\arrow[from=2-2, to=2-1]
\end{tikzcd}
\]
is \emph{acyclic}, in some sense, and this acyclicity yields a fully faithful functor as required for a categorical resolution.

However, the resulting gluing will only be smooth if both $Y$ and
the associated reduced scheme of $S$ are smooth.
When $f$ is a blow-up with centre $S_0$, $S$ is some nilpotent thickening of $S_0$.
In particular it is, generally, not reduced and so there is certainly no reason for it to be smooth.
Thus we have to get rid of the non-reducedness and potentially further resolve $S$.
Removing the non-reducedness can, in some sense, be done by viewing $S$ as a filtered scheme by equipping it with its radical filtration, see also the next paragraph.
Further resolving $S$ (endowed with its radical filtration) and inductively continuing the procedure, using the fact that one can construct acyclic squares as above in a (somewhat) functorial fashion, we inductively construct bigger and bigger acyclic hypercubes.
As $S$ is strictly smaller than $X$, this procedure stops at some point.
For example, with notation from diagram \ref{eq: data here}, in the second step one obtains a cube of the form (we leave out the filtrations)
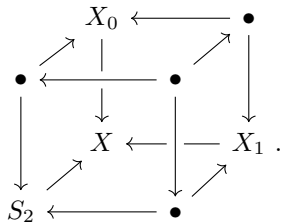
\[
\begin{tikzcd}[sep= 1em]
	& {X_0} && \bullet \\
	\bullet && \bullet \\
	& {X} && {X_1}\rlap{ .} \\
	{S_2} && {\bullet}
	\arrow[from=4-3, to=3-4]
	\arrow[from=4-3, to=4-1]
	\arrow[from=2-1, to=4-1]
	\arrow[from=1-4, to=3-4]
	\arrow[from=2-3, to=1-4]
	\arrow[from=1-2, to=3-2]
	\arrow[from=1-4, to=1-2]
	\arrow[from=3-4, to=3-2]
	\arrow[from=2-1, to=1-2]
	\arrow[from=4-1, to=3-2]
	\arrow[from=2-3, to=4-3, crossing over]
	\arrow[from=2-3, to=2-1, crossing over]
\end{tikzcd}
\]
A sketch of the construction of the hypercube can be found in \S\ref{subsubsec: sketch}.
With the resulting hypercube of (filtered) schemes we associate a hypercube of dg categories.
The punctured hypercube of dg categories  is then glued in order to obtain a categorical resolution of $X$.
The smoothness of the resulting glued dg category is ensured by the fact that all the vertices adjacent to $X$ are smooth by construction (and that all the morphisms in the hypercube are proper).
Placing adequate (finite) filtrations on the schemes in diagram \eqref{eq: data here} we obtain a categorical resolution of the form
\[
\bD(X)\hookrightarrow\langle \bD(X_0,F_0^*),\bD(X_1,F_1^*),\dots, \bD(X_n,F_n^*)  \rangle
\]
with $n\leq \dim X$ and where every component admits a further semi-orthogonal decomposition
\begin{equation}\label{eq: sod filtered scheme}
	\bD(X_i,F_i^*)=
	\langle \underbrace{\bD((X_i)_\mathrm{red}),\dots,\bD((X_i)_\mathrm{red}}_{\text{length }F_i\text{ components}} \rangle.
\end{equation}

The semi-orthogonal decomposition \eqref{eq: sod filtered scheme} is the main reason for considering filtrations. 
Suppose $S$ is some non-reduced scheme whose associated reduced scheme is a smooth variety.
Then, by endowing $S$ with its radical filtration $F_{\text{rad}}^*$ we obtain a (finite length) filtered scheme $(S,F_{\text{rad}}^*)$ that behaves smoothly; of course, $S$ is not smooth.
To formulate the proof, i.e.\ to adequately place filtrations on the schemes in diagram \eqref{eq: data here}, we work in a category with as objects filtered schemes.
The morphisms in this category include those one expects, morphisms of schemes that are compatible with the filtrations, but another type of morphism is also included.
These are the so-called refinements, an example of which is the procedure $S\mapsto(S,F_{\text{rad}}^*)$ of endowing a scheme with a (finite) filtration.
Combining filtered morphisms with refinements we obtain what we call \emph{generalised morphisms}, for lack of inspiration, and generally denote these by a squiggly arrow $\rightsquigarrow$.
Foundational material concerning the category of filtered schemes, including the (derived) category of quasi-coherent modules and the (derived) pushforward and pullback functors for generalised morphism, is developed in \S\ref{sec: filt sch}.
The reader that is willing to take these results for granted, accepting that everything for filtered schemes works as one expects, can skip this section.		
Moreover, we recommend the reader simply read \S\ref{sec: cat res}, looking back at the previous sections when necessary.

\addtocontents{toc}{\SkipTocEntry}	
\subsection*{Overview of the paper}
We start in \S\S\ref{sec: prelim tricat} and \ref{sec: prelim dg cat} by recalling some facts concerning triangulated and dg categories; this can safely be skipped by readers already familiar with those concepts. 
These sections mainly serve to introduce terminology, notation, and conventions

After this, in \S\ref{sec: directed dg cat}, we study `directed dg categories', a natural extension of \cite[\S3.2]{LuntsSchnurer}.
These are the dg categories that appear as enhancements of triangulated categories admitting semi-orthogonal decompositions.
We give a sufficient condition for their smoothness in Proposition \ref{prop: smoothness directed dg cat}.
In Appendix \ref{Asec: more on smoothness}, we briefly discuss necessary conditions.

Then, in \S\ref{sec: gluing dg cat}, we discuss how to glue (punctured) hypercubes of dg categories.
The way in which the categorical resolution is constructed is by gluing a punctured hypercube of dg categories.
We give sufficient conditions for the glued category to be smooth (Corollary \ref{cor: smoothness gluing}), and discuss necessary and sufficient conditions for a natural dg functor from the dg category situated at the puncture to the glued punctured hypercube to be quasi fully faithful (Proposition \ref{prop: acyclic hypercube iff qff}).

Filtered schemes are introduced in \S\ref{sec: filt sch}. 
Most of the chapter consists of making everything that works for schemes also work in the filtered setting.
In \S\ref{subsec: cat fSch}, we define (finite length) filtered schemes and associate appropriate modules over them.
Particularly, we have the usual pullback/pushforward-like adjunction between the module categories for any generalised morphism.
A special instance of a generalised morphism is `taking a refinement'.
Then, in \S\ref{subsec: rel schemes and A-spaces}, we briefly discuss the relation between filtered schemes, schemes and the $\Acal$-spaces of \cite{KuznetsovLunts}.
Next, in \S\ref{subsec: dercat}, we discuss the derived category of a finite length filtered scheme, showing that it has enough admissible complexes to obtain a derived pullback/pushforward-like adjunction.
We define and discuss perfect complexes and the existence of semi-orthogonal decompositions of finite length filtered schemes in, respectively, \S\S\ref{subsec: perf comp} and \ref{subsec: SODs}.
Functorial enhancements of certain types of filtered schemes are constructed in \S\ref{subsec: the enhancement}.
Lastly, in \S\S\ref{subsec: filt blow} and \ref{subsec: filt nrl}, we extend blow-ups and nonrational loci to the filtered setting, and show in \S\ref{subsec: two acyclic squares} how the latter and refinements induce acyclic squares on enhancements.

Finally, in \S\ref{sec: cat res}, we construct a categorical resolution of any finite length filtered scheme whose underlying scheme is separated and of finite type over a field of characteristic zero.

\addtocontents{toc}{\SkipTocEntry}	
\subsection*{Acknowledgements}
This work is part of the author's Ph.D.\ thesis at the Vrije Universiteit Brussel under the supervision of Michel Van den Bergh.
The author would like to thank him for generously sharing his knowledge and his help throughout.
Moreover, the author thanks Greg Stevenson for some useful comments and pointing out some typos.

\addtocontents{toc}{\SkipTocEntry}	
\subsection*{Some conventions}
Throughout $\kk$ denotes an arbitrary field of characteristic zero.
As everything will be considered over this field, we will neglect to write the adjective `$\kk$-' in places.

By `module' we mean `right module', unless otherwise specified. 
All dg categories considered are small unless explicitly stated otherwise.
For dg bimodules, by convention, the contravariant variable comes first just as for hom-complexes, i.e.\ for dg $(\Ascr,\Bscr)$-bimodules we consider dg functors $\Bscr^{\op}\otimes_\kk\Ascr\to\Cscr(\kk)$. 
Perfect dg modules are the compact objects of the derived category (they are not assumed to be in addition h-projective).

The closed subscheme corresponding to a quasi-coherent ideal $\Ical$ of a scheme $X$ is denoted by $\VV_X(\Ical)$.
For finite length filtered schemes we often adopt a naming convention for the filtration that reflects its length.
For example, $(X,F^*)$ and $(Y,F^*)$ are implicitly understood to have the same length, whilst $(Z,G^*)$ could have an a priori different length.	

\section{Preliminaries on triangulated categories}\label{sec: prelim tricat}
	We advise the reader knowledgable about triangulated categories to skip this section.
	It collects some definitions and results concerning triangulated categories that we will use throughout this paper. 
	As most definitions are standard, this mainly serves to introduce notation.
\subsection{Compact objects}
	Let $\Tsf$ be a triangulated category with arbitrary direct sums (some authors call this cocomplete, but we refrain from using this terminology).
	An object $A$ of $\Tsf$ is called \emph{compact} if $\Hom_\Tsf(A,-):\Tsf\to\Ab$, the covariant hom-functor, commutes with arbitrary direct sums.		
	The full subcategory of $\Tsf$ consisting of compact objects is a thick (i.e.\ closed under direct sums) triangulated subcategory and denoted $\Tsf^c$.
	
	A triangulated category $\Tsf$ with direct sums is \emph{compactly generated} if there exists a set $S\subset\Tsf^c$ of compact objects that \emph{generate} $\Tsf$, i.e.\ such that its \emph{right orthogonal}
	\[
	S^\perp:=\{ B\in\Tsf \mid \Hom_\Tsf(A,B[i])=0,\text{ for }A\in S, i\in\ZZ \} = 0.
	\]
	In this case $S$ determines both $\Tsf$ and $\Tsf^c$, namely
	\begin{align*}
		\Tsf^c&=\Thick(S), \\
		\Tsf&=\Loc(S),
	\end{align*}
	see e.g.\ \cite[Theorem 2.1]{NeemanGro}.
	Here, $\Thick(S)$, respectively, $\Loc(S)$, denotes the smallest strictly full triangulated subcategory containing $S$ that is thick, respectively, localising (i.e.\ closed under direct sums).
	
	Let $F:\Tsf_1\to\Tsf_2$ be a triangulated functor between triangulated categories with direct sums.
	We say $F$ \emph{commutes with direct sums} if, for any set of objects $A_i$ in $\Tsf_1$, the natural morphism
	\[
	\oplus_i F(A_i)\to F(\oplus_i A_i)
	\] 
	is an isomorphism. 
	Moreover, we say $F$ \emph{preserves compactness} when $F(\Tsf_1^c)\subseteq \Tsf_2^c$.
	
	To finish this subsection we recall two lemmas, for which we refer to \cite{KuznetsovLunts} for proofs.
	Some original (or at least older) references for some of these statements are \cite[Theorem 5.1]{NeemanGro} and \cite[Lemma 1]{Beilinson}.
	
	\begin{lemma}[{\cite[Lemma 2.10]{KuznetsovLunts}}]\label{lem: nice functor preserving or reflecting compactness}
		Let $F:\Tsf_1\to\Tsf_2$ be a triangulated functor between triangulated categories with direct sums.
		\begin{enumerate}
			\item Assume $F$ is fully faithful and commutes with direct sums. 
			If $F(A)$ is compact, then $A$ is compact.
			\item\label{item: nicefunctor2} Assume $F$ has a right adjoint $G$ and $\Tsf_1$ is compactly generated.
			Then, $F$ preserves compactness if and only if $G$ commutes with direct sums.
		\end{enumerate}
	\end{lemma}
	\begin{remark}
		The requirement that $\Tsf_1$ is compactly generated in $\ref{item: nicefunctor2}$ is only needed for the only if direction.
	\end{remark}
	
	\begin{lemma}[{\cite[Lemma 2.12]{KuznetsovLunts}}]\label{lem: fully faithful if so on compact generators etcetc}
		Let $F:\Tsf_1\to\Tsf_2$ be a triangulated functor between triangulated categories with direct sums that commutes with arbitrary direct sums.
		Let $S\subset\Tsf_1^c$ be a set of compact objects that generates $\Tsf_1$. 
		If $F$ preserves compactness and, for all $A,A'\in S$ and $i\in\ZZ$,
		\[
		F:\Hom_{\Tsf_1}(A,A'[i])\to\Hom_{\Tsf_2}(FA,FA'[i])
		\]
		is bijective, then $F$ is fully faithful.
		If, moreover, $F(S)$ generates $\Tsf_2$, then $F$ is an equivalence.
	\end{lemma}
	\begin{remark}
		It is enough to assume $F(S)\subseteq \Tsf_2^c$ since this implies $F$ preserves compactness (as $\Tsf_1^c=\Thick(S)$).
	\end{remark}

\subsection{Semi-orthogonal decompositions}
	Let $\Tsf$ be a triangulated category and let $\Tsf_1,\dots,\Tsf_n$ be triangulated subcategories of $\Tsf$.
	The sequence $(\Tsf_1,\dots,\Tsf_n)$ is called a \emph{semi-orthogonal collection} of triangulated subcategories if
	\[
	\Hom_{\Tsf}(\Tsf_i,\Tsf_j)=0\quad\text{for }1\leq j < i \leq n.
	\]
	In words, there are no morphisms `going backwards'.
	If, in addition, the smallest triangulated subcategory containing $\Tsf_1,\dots,\Tsf_n$ is the whole of $\Tsf$, the sequence is called a \emph{semi-orthogonal decomposition}.
	The subcategories $\Tsf_i$ are called the \emph{components} of the decomposition.
	We use the notation
	\[
	\Tsf=\langle \Tsf_1,\dots,\Tsf_n\rangle
	\] 
	to express that $\Tsf$ has a semi-orthogonal decomposition with components $\Tsf_1,\dots,\Tsf_n$.
	
	We mention some properties in the special case $n=2$.
	So, assume $\Tsf=\langle \Tsf_1,\ \Tsf_2\rangle$ has a semi-orthogonal decomposition with two components.
	In this case, $\Tsf_1$ is \emph{left admissible} and $\Tsf_2$ is \emph{right admissible}, by definition, this means that the inclusion $j:\Tsf_1\hookrightarrow\Tsf$ admits a left adjoint $j^*$ and the inclusion $i:\Tsf_2\hookrightarrow\Tsf$ admits a right adjoint $i^!$.
	Furthermore, for any object $A$ in $\Tsf$ there exists a distinguished triangle
	\[
	i i^! A\to A\to j j^* A\to
	\]
	and we have
	\begin{align*}
		\Tsf_1&=\Tsf_2^\perp:=\{A\in\Tsf\mid \Hom_{\Tsf}(\Tsf_2,A)=0\}, \\
		\Tsf_2&={}^\perp\Tsf_1:=\{A\in\Tsf\mid \Hom_{\Tsf}(A,\Tsf_1)=0\}.
	\end{align*}
	Conversely, any right admissible triangulated subcategory $\Ssf\subseteq\Tsf$ induces a semi-orthogonal decomposition
	\[
	\Tsf=\langle \Ssf^\perp,\Ssf\rangle.
	\]
	Of course, a `mirrored' statement holds for left admissible triangulated subcategories.
	
	Let us finish this section with a lemma for future reference. 
	
	\begin{lemma}\label{lem: compact sod}
		Let $\Tsf$ be a triangulated category with direct sums admitting a semi-orthogonal decomposition $\Tsf=\langle \Tsf_1, \Tsf_2\rangle$.
		Assume that $\Tsf_1=\mathop\mathsf{Loc}(S_1)$ with $S_1\subseteq \Tsf^c$.
		Then, for any distinguished triangle
		\[
		A_2\to A\to A_1\to 
		\]
		with $A_i\in\Tsf_i$ we have that $A$ is compact if and only if $A_1$ and $A_2$ are compact (in $\Tsf$).
	\end{lemma}
	\begin{proof}
		One direction is obvious, for the other assume $A$ is compact. 
		Let $j:\Tsf_1\to \Tsf$ denote the inclusion and let $j^*$ be its left adjoint.
		We have $A_1=j j^* A$ (this follows for example form \cite[Section IV.1 Corollary 5]{GelfandManin}). 
		By assumption, $\Tsf_1$ is closed under direct sums, so as a result $j$ commutes with direct sums.
		Therefore, its left adjoint $j^*$ preserves compactness by Lemma \ref{lem: nice functor preserving or reflecting compactness}.
		Moreover, as $\Tsf_1$ is compactly generated by $S_1$ by assumption, $\Tsf_1^c=\Thick(S_1)$, so $j(\Tsf_1^c)\subseteq \Tsf^c$ as $S_1\subseteq \Tsf^c$.
		It follows that $A_1$ is compact in $\Tsf$, and thus also $A_2$.
	\end{proof}
	\begin{remark}
		It follows that
		\[
		\Tsf^c=\langle \Tsf_1\cap \Tsf^c, \Tsf_2\cap\Tsf^c\rangle
		\]
		and, moreover, $\Tsf_1^c=\Tsf_1\cap \Tsf^c$.
		However, in general, $\Tsf_2\cap \Tsf^c \varsubsetneq\Tsf_2^c$ (e.g.\ look at $\bD(X) = \langle \bD(U), \bD_Z(X) \rangle$ with $U\subset X$ a quasi-compact open of a quasi-compact quasi-separated scheme $X$ and $Z:=X-U$).
		In addition the statement is not true when one replaces the assumption on $\Tsf_1$ by a similar one on $\Tsf_2$ (e.g. look at $\bD(\left(\begin{smallmatrix}	\Ascr & 0 \\ \phi & \Bscr \end{smallmatrix}\right)) = \langle b^!\bD(\Bscr), a^*\bD(\Ascr) \rangle$, where we used the notation of Lemma \ref{lem: functors A,B,C}).
	\end{remark}

\section{Preliminaries on differential graded categories}\label{sec: prelim dg cat}
	We advise the reader knowledgeable about differential graded (dg) categories to skip this section.
	It gathers some recollections concerning dg categories. 
	As most definitions are standard, this mainly serves to introduce the conventions and notation we use.
	This is far from a comprehensive introduction to the subject.
	A good introduction to dg categories is \cite{Keller}, see also \cite[\S3]{KuznetsovLunts} and \cite[\S3]{Orlov} on which we based much of this section.
	We fix a ground field\footnote{
		Of course, most of what follows works more generally over a general commutative base ring. But, as our main application will be in the setting where $\kk$ is a field we restrict to this here already for ease. (Otherwise we would have to take derived tensor products at appropriate places.)
	}
	$\kk$.
	As everything will be considered over this field, we will neglect to write the adjective `$\kk$-' in places.

	\subsection{The category of differential graded categories}
		A \emph{dg category} $\Ascr$ over a field $\kk$ is a category enriched over $\Csf(\kk)$, the closed symmetric monoidal category\footnote{\label{foot: monoidal structure C(k)}
			The monoidal structure on $\Csf(\kk)$ is given by the usual tensor product of complexes with symmetry given by
			\[
			a\otimes b\mapsto (-1)^{|a| |b|} b\otimes a,
			\]
			where $|-|$ denotes the degree of an element.
			This is the so-called `Koszul sign rule', i.e.\ add signs when exchanging elements, and is why there are so many minus signs abound when working with dg categories.
			Note moreover that by the definition of the monoidal structure we have
			\begin{gather*}
				(f\otimes g)(a\otimes b) := (-1)^{|a| |g|} f(a)\otimes g(b),\\
				(f\otimes g)\circ(f'\otimes g') :=  (-1)^{|f'| |g|} (f\circ f')\otimes (g\circ g'),
			\end{gather*}
			where $f, f', g$ and $g'$ are morphisms of complexes and $a$ and $b$ are elements in the respective domains of the morphisms.
		}  of cochain complexes over $\kk$.
		That is, it is a usual category such that for all objects $A, B\in \Ascr$ the morphism set $\Hom_\Ascr(A,B)$, which we will often simply denote by $\Ascr(A,B)$, has the structure of a $\kk$-complex and the composition is a morphism of $\kk$-complexes.
		Thus,
		\[
		\Ascr(A,B) = \oplus_{i\in\ZZ} \Ascr(A,B)^i
		\]
		is a graded $\kk$-module endowed with degree one morphisms $d^i:\Ascr(A,B)^i\to \Ascr(A,B)^{i+1}$ squaring to zero.
		When $f\in \Ascr(A_1,A_2)^i$, we say $f$ is \emph{homogeneous of degree} $i$ and we denote this by $|f|=i$.
		We say $f$ is \emph{closed} when $df=0$.
		
		For any dg category $\Ascr$ its \emph{homotopy category} $[\Ascr]$ is an ordinary category that has the same objects as $\Ascr$ and has $H^0\Hom_\Ascr(-,-)$ as morphisms.
		Similarly, its \emph{underlying category} $Z^0\Ascr$ is the category obtained by taking $Z^0\Hom_\Ascr(-,-)$ instead of $H^0\Hom_\Ascr(-,-)$.
		We say a degree zero morphism $f$ is a \emph{homotopy equivalence} if it is invertible in $[\Ascr]$, and is a \emph{dg isomorphism} if it is invertible in $Z^0\Ascr$. 
		Two objects of $\Ascr$ are \emph{homotopy equivalent} if there exists a homotopy equivalence between them, similarly they are \emph{dg isomorphic} if instead there exists a dg isomorphism.
		
		For any two dg categories $\Ascr$ and $\Bscr$ a \emph{dg functor} $F:\Ascr\to\Bscr$ is a $\Csf(\kk)$-enriched functor, i.e.\ for all $A,B\in\Ascr$ the map $\Ascr(A,B)\to\Bscr(FA,FB)$ is a morphism of $\kk$-complexes.
		A `good' notion of equivalences between dg categories is given by \emph{quasi-equivalences}.
		These are dg functors that induce quasi-isomorphisms on the morphism complexes, we say it is \emph{quasi fully faithful}, and in addition induce essentially surjective functors on the homotopy categories.
		
		The \emph{category of small dg categories} is denoted $\dgcat$ and has the small dg categories, i.e.\ whose objects form a set, as objects and the dg functors as morphisms.
		In this text all dg categories considered are small unless explicitly stated otherwise.
		The category of small dg categories has the structure of closed monoidal category.
		Namely, the \emph{tensor product} of two dg categories $\Ascr$ and $\Bscr$ is the dg category  $\Ascr\otimes_\kk\Bscr$ with objects $\Obj\Ascr\times\Obj\Bscr$ and morphisms
		\[
		(\Ascr\otimes_\kk\Bscr)( (A,B), (A',B') ):= \Ascr(A, A')\otimes_\kk\Bscr(B,B').
		\]
		There are minus signs in the composition due to the Koszul sign rule.
		Moreover, for any two dg categories $\Ascr$ and $\Bscr$ there is a dg category $\Fun_{dg}(\Ascr,\Bscr)$ whose objects are the dg functors from $\Ascr$ to $\Bscr$.
		Given dg functors $F, G:\Ascr\to\Bscr$ the complex $\Hom(F,G)$ has as $i$th component families of morphisms $\eta_A\in\Bscr(FA,GA)^i$ making the usual diagram for a natural transformation commute up to a sign given by the Koszul sign rule.	
		
		Lastly, for any dg category $\Ascr$ there is an \emph{opposite dg category} $\Ascr^{\op}$. 
		It has as objects $\Obj\Ascr$ and as morphisms $\Ascr^{\op}(A,B):=\Ascr(B,A)$.
		Again, there are some minus signs in the composition given by the Koszul sign rule.

	\subsection{Differential graded modules}
		To any dg category $\Ascr$ one can associate the (big) \emph{dg category of (right) dg $\Ascr$-modules} $\dgMod\Ascr$.
		It is defined as the dg category $\Fun_{dg}(\Ascr^{\op}, \Cscr(\kk))$, where $\Cscr(\kk)$ is the (big) dg category of $\kk$-complexes (as $\Csf(\kk)$ is closed monoidal it is naturally enriched over itself).
		With this definition $\kk$-complexes are simply dg $\kk$-modules, where we view $\kk$ as a dg algebra (a dg category with one object) concentrated in degree zero.
		There is a natural fully faithful dg functor
		\[
		h^\bullet:\Ascr\to\dgMod\Ascr, \quad A\mapsto \Ascr(-,A),
		\]
		defined on morphisms by composition, called the \emph{Yoneda embedding}.
		The image of an object $A\in\Ascr$ is denoted $h^A$ and is called a \emph{representable} dg $\Ascr$-module (represented by $A$). 
		The category of left dg $\Ascr$-modules is $\dgMod\Ascr^{\op}$.
		Moreover, for any two dg categories $\Ascr$ and $\Bscr$ we can define dg $(\Ascr,\Bscr)$-bimodules by considering $\dgMod(\Bscr\otimes_\kk \Ascr^{\op})$.
		By convention the contravariant variable comes first just as for hom-complexes, i.e.\ we consider dg functors $\Bscr^{\op}\otimes_\kk\Ascr\to\Cscr(\kk)$. 
		This is for convenience as most bimodules we consider will end up being hom-complexes in dg categories.
		For any dg category $\Ascr$ there is a special dg $(\Ascr,\Ascr)$-bimodule worth mentioning, namely the \emph{diagonal dg bimodule} which we simply denote by $\Ascr$.
		It is given by mapping $(A,B)\mapsto\Ascr(A,B)$.
		(Note that by our convention on bimodules there is no confusion possible when we write $\Ascr(A,B)$, i.e.\ whether we mean hom-complex or diagonal bimodule.)
		Furthermore, for any dg $(\Ascr,\Bscr)$-bimodule $\phi$ and dg functors $F:\Ascr'\to\Ascr$ and $G:\Bscr'\to\Bscr$ we define the \emph{restricted} dg $(\Ascr',\Bscr')$-bimodule ${}_F\phi_G$, also denoted ${}_{\Ascr'}\phi_{\Bscr'}$, as
		\[
		{}_F\phi_G(B',A'):=\phi(G(B'),F(A'))\quad\text{for }A\in\Ascr',B'\in\Bscr.
		\]
		
		As for ordinary modules over rings, one can tensor and hom dg modules to obtain new dg modules and there is the tensor-hom adjunction. 
		We refer the reader to the references given above for details.
		
		Let $F:\Ascr\to\Bscr$ be a dg functor between dg categories.
		Precomposition with $F$ induces the \emph{restriction dg functor} $F_*:\dgMod\Bscr\to\dgMod\Ascr$, sometimes also denoted $\Res_F$, it maps $N\mapsto N\circ F$.
		It has a left adjoint $F^*:\dgMod\Ascr\to\dgMod\Bscr$, the \emph{induction dg functor}, sometimes also denoted $\Ind_F$, given by mapping $M\mapsto M\otimes_{\Ascr} {}_F\Bscr$, where ${}_F\Bscr$ is the dg $(\Ascr,\Bscr)$-bimodule obtained from the diagonal bimodule by restricting along the left $\Bscr$-action.
		It is useful to note that $F^*(h^A)\cong h^{F(A)}$ functorially, so $F^*$ extends $F$ to dg modules.
		Lastly, $F_*$ also has a right adjoint $F^!$.
		It is defined by $F^!M(B):=(\dgMod\Ascr)(F_*h^B,M)$.
		Succinctly, we have the following diagram
		\begin{equation}\label{eq: functors on dgmod induced by dg functor}
			\begin{tikzcd}[sep=3.5em]
				\dgMod\Ascr \\ \dgMod\Bscr\rlap{ .}
				\arrow[""{name=0, anchor=center, inner sep=0}, "{F_*}"{description}, from=2-1, to=1-1]
				\arrow[""{name=1, anchor=center, inner sep=0}, "{F^!}"{description}, bend left=55, from=1-1, to=2-1]
				\arrow[""{name=2, anchor=center, inner sep=0}, "{F^*}"{description}, bend right=55, from=1-1, to=2-1]
				\arrow["\dashv"{anchor=center}, draw=none, from=2, to=0]
				\arrow["\dashv"{anchor=center}, draw=none, from=0, to=1]
			\end{tikzcd}
		\end{equation}
		
		Let $\Ascr$ be a dg category.
		We say that a dg $\Ascr$-module $M$ is \emph{acyclic} if the dg $\kk$-module $M(A)$ is acyclic for all $A\in\Ascr$.
		The dg subcategory of acyclic dg modules is denoted $\Ac\Ascr$.
		The homotopy category $[\dgMod\Ascr]$ carries a natural triangulated structure, which we discuss in the next section, and $[\Ac\Ascr]$ is a localising triangulated subcategory.
		We define the \emph{derived category} of $\Ascr$ as the Verdier quotient
		\[
		\bD(\Ascr):= [\dgMod\Ascr]/[\Ac\Ascr].
		\] 
		It is a compactly generated triangulated category, so in particular it has (arbitrary small) direct sums.		
		The Yoneda embedding induces a fully faithful functor $[\Ascr]\to\bD(\Ascr)$ and its image (the representable dg modules) forms a set of compact generators.
		
		For the definition of free, (finitely generated) semi-free, h-projective, h-injective and h-flat dg modules we refer to the references given above.
		We simply mention that these allow us to define derived functors in the usual fashion.
		In particular, we obtain derived versions of the functors in the diagram \eqref{eq: functors on dgmod induced by dg functor} and a straightforward application of Lemma \ref{lem: fully faithful if so on compact generators etcetc} yields.
		\begin{lemma}\label{lem: prop F imlies prop IndF}
			Let $F:\Ascr\to\Bscr$ be a dg functor. 
			Both the derived induction functor $\bL F^*$ and the restriction functor $F_*$ commute with arbitrary direct sums.
			Moreover,
			\[
			\bL F^*(h_\Ascr^A)\cong h_\Bscr^{F(A)}.
			\]
			Thus, if $F$ is quasi fully faithful (respectively a quasi-equivalence), then $\bL F^*$ is fully faithful (respectively an equivalence).
		\end{lemma}
		A dg functor whose derived induction functor (or equivalently restriction functor) induces an equivalence at the derived level, as in the above lemma, is called a \emph{Morita equivalence}. 
		Hence, the lemma shows that any quasi-equivalence is a Morita equivalence.
		
		A dg module $M$ is called \emph{perfect} if its image in $\bD(\Ascr)$ is compact, i.e.\ ${\bD(\Ascr)}(M,-)$ commutes with direct sums. 
		Sometimes perfect dg modules are assumed to be additionally semi-free or h-projective, in which case the perfect dg modules are the homotopy direct summands  of finitely generated semi-free dg modules.
		We do not include this assumption, they are merely direct summands in the derived category instead of homotopy direct summands; but as a consequence, the perfect dg modules are closed under quasi-isomorphisms.
		
		We finish this subsection with a definition that will be needed later on.
		A dg $(\Ascr,\Bscr)$-bimodule $\phi$ is said to be \emph{right perfect}, or $\Bscr$-\emph{perfect}, if $\phi(-,A)$ is a perfect dg $\Bscr$-module for every $A\in\Ascr$; equivalently if $\bL\phi:=-\otimes_{\Ascr}^\bL\phi:\bD(\Ascr)\to \bD(\Bscr)$ preserves perfectness (i.e.\ compactness).
		More generally, if $P$ is a property of dg modules we will say that $\phi$ has $P$ as right dg module, or is right $P$, if $\phi(-,A)$ has $P$ for every $A\in\Ascr$.
		Of course, we can also consider the left version of the above.

	\subsection{Pretriangulated differential graded categories and enhancements}
		Let $\Ascr$ be a dg category.
		For any dg $\Ascr$-module $M$ and integer $n\in\ZZ$ we can define the shifted dg module $M[n]$ of $M$ by shifting the objectwise complexes 
		\[
		M[n](A) := M(A)[n]\quad\text{for }A\in\Ascr.
		\]
		Similarly, for any closed degree zero morphism $f:M\to N$ between dg $\Ascr$-modules we can define its cone dg module $\cone(f)$ by taking the cone objectwise, i.e. for any object $A\in\Ascr$ we have a morphism of complexes $f_A:M(A)\to N(A)$ so we can put
		\[
		\cone(f)(A) := \cone(f_A:M(A)\to N(A)).
		\]
		Both of these objectwise constructions are readily seen to extend to dg functors.
		In fact, we can characterise them more abstractly. 
		The shift $M[n]$ is the unique (up to unique dg isomorphism) dg module equipped with a closed degree $n$ isomorphism $M[n]\to M$.
		Similarly, $\cone(f)$ is the unique (up to unique dg isomorphisms) dg module equipped with degree zero morphisms
		\[
		M[1]\xrightarrow{i}\cone(f)\xrightarrow{p}M[1],\quad N\xrightarrow{j}\cone(f)\xrightarrow{s}N,
		\]
		satisfying
		\begin{gather*}
			pi=\id_{M[1]},\quad sj=\id_{N},\quad pj=0,\quad si=0,\quad ip+js=\id_{\cone(f)},\\	d(j)=d(p)=0,\quad d(i)=jf\epsilon,\quad d(s)=-f\epsilon p,
		\end{gather*}
		where $\epsilon:M[1]\to M$ is a closed degree one isomorphism.
		(Not all of the above relations are needed to uniquely determine the cone, e.g.\ the latter two imply each other given the former.)
		Using the shift and cone we obtain, as for complexes over $\kk$, a canonical triangulated structure on the homotopy category of dg modules $[\dgMod\Ascr]$.
		
		A dg category $\Ascr$ is called \emph{pretriangulated} if the essential image of $[\Ascr]$ in $[\dgMod\Ascr]$ under the Yoneda embedding is closed under shifts and cones of degree zero morphisms, i.e.\ the essential image is a triangulated subcategory of  $[\dgMod\Ascr]$.
		Worded differently, $\Ascr$ is pretriangulated if for every $A\in\Ascr$ and $k\in\ZZ$ the dg module $h^A[k]$ is homotopy equivalent to a representable dg module and for every closed degree zero morphism $f:M\to N$ the dg module $\cone(h^f:h^M\to h^N)$ is homotopy equivalent to a representable dg module.
		If instead of `homotopy equivalent' we require `dg isomorphic', we say that $\Ascr$ is $\emph{strongly pretriangulated}$.
		By the dg Yoneda lemma, being strongly pretriangulated is simply requiring the existence of certain objects in $\Ascr$ together with appropriate morphisms.
		
		Every dg category can be embedded in a `smallest' strongly pretriangulated dg category, called its \emph{pretriangulated hull}.
		For example, one can define this as the dg category consisting of the finitely generated semi-free dg modules, these are exactly the dg modules generated by the representable ones under shifts and cones.
		Alternatively, it can be described more concretely using twisted complexes.
		See the next section for this.
		
		The importance of pretriangulated dg categories lies in the fact that they can be used to enhance triangulated categories.
		An \emph{enhancement} of a triangulated category $\Tsf$ is a pretriangulated dg category $\Ascr$ together with an equivalence $\Tsf\cong[\Ascr]$ of triangulated categories.
		We often do not mention the specific equivalence explicitly, although it is part of the data, and simply say that $\Ascr$ is a dg enhancement of $\Tsf$.
		Enhancements do not always exist and, if they do, need not be unique up to quasi-equivalence.
		However, in most algebraic and geometric settings they exist and are unique.
		See \cite{CanonacoStellari} for an overview, and \cite{CanonacoNeemanStellari} for more recent results.
		
		Let $\Ascr$ be a dg category and $\Bscr\subseteq\Ascr$ be a dg subcategory. 
		One can consider its \emph{dg quotient} $\Ascr/\Bscr$ as constructed in \cite{Drinfeld}.
		As we are working over a field the quotient is constructed by adding a contracting homotopy for every $B\in\Bscr$, i.e.\ a morphism $\epsilon_B$ of degree minus one with $d(\epsilon_B)=\id_B$. 
		When $\Ascr$ and $\Bscr$ are pretriangulated this construction gives a natural enhancement of the Verdier quotient $[\Ascr]/[\Bscr]$ as $\Ascr/\Bscr$ is then pretriangulated and
		\[
		[\Ascr/\Bscr]\cong [\Ascr]/[\Bscr],
		\]
		see \cite[Theorem 3.4]{Drinfeld} and \cite[Lemma 1.5]{LuntsOrlov}.

	\subsection{Twisted complexes}\label{subsec: twist}
		We will define our gluing of dg categories as the pretriangulated hull of a `directed' dg category. 
		One way of working with the pretriangulated hull in an explicit fashion is by considering twisted complexes.
		This subsection serves to make our conventions concerning twisted complexes explicit. 
		We use those of \cite{BondalLarsenLunts} and \cite{Drinfeld} (though the latter is not explicit in the signs of the differential in the `shift closed dg category'). 
		For the remainder of this subsection, $\Ascr$ is a fixed dg category. 

		\subsubsection{Adding shifts}
			The \emph{shift closed dg category} $\ZZ\Ascr$ is defined as follows.
			\begin{itemize}
				\item The objects are formally tuples $(A,n)$, where $A\in\Ascr$ and $n\in \ZZ$. However, we mostly write $(A,n)$ as $A[n]$.
				\item The morphism complex $\ZZ\Ascr(A[k],B[l])$ is equal to $\Ascr(A,B)[l-k]$ as graded complex but with differential given by $d_{\text{new}} = (-1)^l d_{\text{old}}$. Composition is naturally induced from $\Ascr$.
			\end{itemize}
			Clearly, $\ZZ\Ascr$ contains $\Ascr$ as a full dg subcategory. 
			The following motivates our choice in signs.
			
			\begin{lemma}
				The category $\ZZ\Ascr$ embeds into the category of dg $\Ascr$-modules $\dgMod\Ascr$ by on objects mapping $A[n]\mapsto h^A[n]$ and on morphisms sending $f\in\ZZ\Ascr(A[k],B[l])^n$ to the natural transformation $f_*:h^A[k]\to h^B[l]$ of degree $n$ with 
				\[
				f_{*,C}: \Ascr(C,A)[k]\to \Ascr(C,B)[l],\,g\mapsto f g,\quad C\in\Ascr
				\]
				on morphisms.
			\end{lemma}

		\subsubsection{Adding direct sums}
			The \emph{shift and (direct) sum closed dg category } $\Sigma \Ascr$ is defined as follows.
			\begin{itemize}
				\item The objects are formally finite sequences $(A_i)_{i=1}^n=(A_1, A_2, \dots, A_n)$ of objects in $\ZZ\Ascr$. 
				We usually write $(A_i)_{i=1}^n$ as $\oplus_{i=1}^n A_i$ or simply $\oplus_i A_i$.
				\item The morphism complex $\Sigma \Ascr(\oplus_i A_i, \oplus_j B_j)$ is equal to $\oplus_{i,j}\ZZ \Ascr(A_i, B_j)$ with composition given by matrix multiplication.
			\end{itemize}
			Note that the empty sequence is allowed above, in which case the morphism complex is zero. 
			So the empty sequence yields a zero object.
			Moreover, in order to get matrix multiplication to work nicely, one must index morphisms as follows, for $(f_{ji}) \in \Sigma \Ascr(\oplus_i A_i, \oplus_j B_j)$ we have $f_{ji}:A_i\to B_j$.

		\subsubsection{Adding cones}
			The \emph{dg category of twisted complexes} $\tw\Ascr$ is defined as follows.
			\begin{itemize}
				\item The objects are tuples $( A, \delta )$, called \emph{twisted complexes}, where $A$ is an object of $\Sigma\Ascr$ and $\delta$ is a degree one endomorphism of $A$ in $\Sigma\Ascr$ given by a strictly upper triangular matrix, i.e.\ $\delta_{ij}=0$ for $i\geq j$ (`$\delta$ decreases the $A_i$'), satisfying $d_{\Sigma\Ascr}\delta+\delta^2=0$.
				\item The morphism complex $(\tw\Ascr)((A,\delta), (B,\delta'))$ equals $\Sigma \Ascr(\oplus_i A_i, \oplus_i B_i)$ as graded complex but with differential given by $d+[\delta,-]$, where $d$ is the differential of $\Sigma \Ascr$. Explicitly, for $f:(A,\delta)\to (B,\delta')$, we have $d_{\tw \Ascr}f= d_{\Sigma \Ascr}f + \delta' f - (-1)^{|f|} f\delta$.	
			\end{itemize}
			In particular, the following shows that $[\tw \Ascr]$ identifies with the triangulated subcategory of $[\dgMod\Ascr]$ generated by the representable dg $\Ascr$-modules (under the functor mapping $( \oplus A_i[k_i], \delta )$ to the dg $\Ascr$-module $\oplus_i \Ascr(-,A_i)[k_i]$ as a graded module endowed with the differential $d_{\Ascr}+\oplus_{i,j}\delta_{ij}$).
			
			\begin{proposition}[{\cite[Proposition 3.10]{BondalLarsenLunts}}]
				Let $\mathscr A$ be a dg category. Then, the following statements hold:
				\begin{enumerate}
					\item the dg category $\tw\Ascr$ is closed under taking cones of closed degree zero morphisms, 
					\item every object in $\tw\Ascr$ can be obtained from objects in $\Ascr$ by taking successive cones of closed degree zero morphisms.		
				\end{enumerate}
			\end{proposition}
			\begin{remark}\label{rem: cone in tw}
				With these conventions, the cone of a closed degree zero morphism $f:A\to B$ in $\Ascr$ is given by the twisted complex 
				$(B\oplus A[1],\left( \begin{smallmatrix}
					0 & f
					\\
					0 & 0 
				\end{smallmatrix}\right) )$.
			\end{remark}

%

	\subsection{Smooth and proper differential graded categories}
		A dg category $\Ascr$ is ($\kk$-)\emph{smooth} if the diagonal bimodule is perfect, i.e.\ $\Ascr$ is perfect over $\Ascr^{\op}\otimes_\kk\Ascr$.
		Spelled out, this means that $\Ascr$ as bimodule is a direct summand, in the derived category, of a bimodule obtainable from representable bimodules by a finite number of shifts and the taking of cones of closed degree zero morphisms. 
		We say that $\Ascr$ is ($\kk$-)\emph{proper} if for all objects $A_1,A_2\in \Ascr$ the $\kk$-complex $\Ascr(A_1,A_2)$ is perfect; equivalently, it has finite dimensional cohomology.
		
		The following lemma of To\"{e}n and Vaqui\'{e} will be crucial for showing the perfectness of bimodules later on. 
		We include a proof for the convenience of the reader.
		\begin{lemma}[{\cite[Lemma 2.8.2]{ToenVaquie}}]\label{lem: right perfect sometimes gives perfect}
			Let $\Ascr, \Bscr$ be dg categories and let $\phi$ be a dg $(\Ascr,\Bscr)$-bimodule.
			Assume $\Ascr$ is smooth and $\phi$ is right perfect, then $\phi$ is perfect.
		\end{lemma}
		\begin{proof}
			As $\phi$ is right perfect, tensoring with $\phi$ induces a compactness preserving functor $\bD(\Ascr^{\op}\otimes_\kk\Ascr)\to\bD(\Ascr^{\op}\otimes_\kk\Bscr)$ (it maps the compact generators $h^{(A,A')}$ to $h^A\otimes_\kk\phi(-,A')$). 
			Moreover, as it maps $\Ascr\mapsto\phi$ and $\Ascr$ is perfect by assumption, we conclude that $\phi$ is perfect.
		\end{proof}
		
		\section{Directed differential graded categories}\label{sec: directed dg cat}
		In this section, we collect some results on directed dg categories, which are a natural extension of \cite[\S3.2]{LuntsSchnurer}.
		These appear naturally as dg enhancements of triangulated categories equipped with semi-orthogonal decompositions.
		We give some sufficient conditions for their smoothness and make some remarks concerning necessary conditions.
		
		\subsection{Terminology}
		Let $\Cscr$ be a dg category and assume it has full dg subcategories\footnote{The indexing $0,\dots,n-1$ may seem strange to the reader, it is chosen to be uniform with the hypercube indexing later on.} $\Ascr_0$, \dots, $\Ascr_{n-1}$ such that 
		$
		\Obj \Cscr = \bigsqcup_{i=0,\dots,n-1}\Obj \Ascr_i
		$ 
		and\footnote{It actually suffices to merely require these hom-complexes to be acyclic, as the dg category is then quasi-equivalent to a dg category where they are zero.} $\Hom_\Ascr(\Ascr_i,\Ascr_j)=0$ for $i>j$.
		We call such a dg category \emph{directed}, or say that it comes with a \emph{directed decomposition}.
		One can think of a directed dg category as a lower\footnote{This depends a bit on one's point of view, as
			\begin{center}
				$\left(\begin{smallmatrix}	
					\Ascr & 0 \\ 
					\phi & \Bscr 
				\end{smallmatrix}\right)
				=\left(\begin{smallmatrix}	
					\Bscr & \phi \\ 
					0 & \Ascr 
				\end{smallmatrix}\right)
				\quad\text{and}\quad
				\left(\begin{smallmatrix}	
					\Ascr & 0 \\ 
					\phi & \Bscr 
				\end{smallmatrix}\right)^{\op}
				=\left(\begin{smallmatrix}	
					\Ascr^{\op} & \phi^{\op} \\ 
					0 & \Bscr^{\op} 
				\end{smallmatrix}\right)$.
			\end{center}
		} 
		triangular matrix.
		Indeed, giving such a dg category is equivalent to giving the data of dg categories $\Ascr_0$, \dots, $\Ascr_{n-1}$ together with dg $(\Ascr_i,\Ascr_j)$-bimodules $\phi_{ij}$ for $i> j$ and $(\Ascr_i,\Ascr_j)$-bimodule morphisms $\phi_{ik}\otimes_{\Ascr_k} \phi_{kj}\to \phi_{ij}$ for $0\leq j<k<i\leq n-1$ satisfying the natural associativity and unitality conditions.
		One can graphically depict this as a lower triangular matrix
		\begin{equation}\label{eq: lower matrix dg cat}
			\Cscr = 
			\begin{pmatrix}
				\Ascr_0		& 0		& 	0 	& \dots & 0\\
				\phi_{10}	& \Ascr_1 & 0 	& \dots & 0\\
				\phi_{20}	& \phi_{21}	& \Ascr_2 & \dots & 0\\
				\vdots & \vdots & \vdots & \ddots & \vdots \\
				\phi_{n-1,0}	& \phi_{n-1,1}	& \phi_{n-1,2} & \dots & \Ascr_{n-1} \\
			\end{pmatrix},
		\end{equation}
		but note that for $n>2$ the notation does not include all the data necessary for defining the composition.
		
		These types of dg categories naturally induce semi-orthogonal decompositions, as the following result shows.
		(This is similar to \cite[Corollary 4.5]{KuznetsovLunts} or \cite[Proposition 3.7]{Orlov} but with a slightly different set-up, as they assume the $\Ascr_i$'s are pretriangulated.)
		\begin{proposition}\label{prop: SOD directed}
			Let $\Cscr$ be as in Equation \eqref{eq: lower matrix dg cat}.
			The inclusions $\Ascr_i\hookrightarrow\Cscr$ induce a semi-orthogonal decomposition
			\[
			[\tw\Cscr]=\langle [\tw\Ascr_0],\dots,[\tw\Ascr_{n-1}]\rangle.
			\]
		\end{proposition}
		\begin{proof}
			The inclusions induce fully faithful triangulated functors $[\tw\Ascr_i]\to[\tw\Cscr]$.
			We thus shamelessly identify the domains with their essential image.
			
			To observe the claim, note that the $\tw\Ascr_i$'s are suitably orthogonal (as the $\Ascr_i$'s are) and that they generate $[\tw\Cscr]$ as triangulated category, i.e.\ $[\tw\Cscr]$ is the smallest triangulated subcategory of itself containing them, since they contain the objects of $\Cscr$ (and $\tw\Cscr$ is generated by those under cones and shifts).
		\end{proof}
		
		Conversely, these categories appear naturally as enhancements of triangulated categories with semi-orthogonal decompositions, as the following proposition shows.
		Its proof is exactly the same as \cite[Proposition 3.8]{Orlov} which is the case $n=2$.
		
		\begin{proposition}
			Let $\Tsf=[\Ascr]$ be a dg enhanced triangulated category and assume it admits a semi-orthogonal decomposition of the form $\Tsf=\langle \Tsf_0,\dots,\Tsf_{n-1}\rangle$.
			Define $\Ascr_i$ as the full dg subcategory of $\Ascr$ with objects from $\Tsf_i$ and put $\phi_{ij}:= {}_{\Ascr_i}\Ascr_{\Ascr_j}$ the restriction of the diagonal bimodule equipped with bimodule maps induced from the composition in $\Ascr$.
			Then, $\Tsf \cong [\tw(\Cscr)]$ with $\Cscr$ as in Equation \eqref{eq: lower matrix dg cat}.
		\end{proposition}
		
		To finish this subsection, we note that there also exists a semi-orthogonal decomposition on the derived level.
		
		\begin{proposition}\label{prop: SOD D(directed)}
			Let $\Cscr$ be as in Equation \eqref{eq: lower matrix dg cat}.
			The derived induction functors of the inclusions $\Ascr_i\hookrightarrow\Cscr$ induce a semi-orthogonal decomposition
			\[
			\bD(\Cscr)=\langle \bD(\Ascr_0),\dots, \bD(\Ascr_{n-1})\rangle.
			\]
		\end{proposition}
		\begin{proof}
			By induction, we can reduce to the case $n=2$ (for the induction step see e.g.\ the beginning of the proof of Proposition \ref{prop: smoothness directed dg cat} below).
			This is then essentially \cite[Proposition 4.6]{KuznetsovLunts}, but, as the set-up is slightly different, we repeat the argument for convenience of the reader.
			
			So, let $\Cscr=\left(\begin{smallmatrix}	\Ascr & 0 \\ \phi & \Bscr \end{smallmatrix}\right)$ and denote by $a:\Ascr\hookrightarrow\Cscr$ and $b:\Bscr\hookrightarrow\Cscr$ the inclusion functors.
			We have to show $\bD(\Cscr)=\langle\bD(\Ascr), \bD(\Bscr) \rangle$.
			
			First note that $a^*$ is exact (see Lemma \ref{lem: functors A,B,C} below) and, by Lemma \ref{lem: prop F imlies prop IndF}, $\bL a^*=a^*$ and $\bL b^*$ are fully faithful as $a$ and $b$ are, so we identify their domains with their essential images.
			As $\bL b^*$ admits a right adjoint, we have a semi-orthogonal decomposition
			\[
			\bD(\Cscr) = \langle\bD(\Bscr)^\perp, \bD(\Bscr) \rangle.
			\]
			Moreover, $b_* a^*=-\otimes^{\bL}_{\Ascr} {}_a\Cscr_b = 0$ as ${}_a\Cscr_b=0$ due to the directedness. 
			So, $\bD(\Ascr)\subseteq\bD(\Bscr)^\perp$.
			Thus, as $ a^*$ also admits a right adjoint, we obtain a further semi-orthogonal decomposition  
			\[
			\bD(\Bscr)^\perp = \langle\bD(\Ascr)^\perp\cap\bD(\Bscr)^\perp, \bD(\Ascr) \rangle.
			\]
			But  $\bD(\Ascr)^\perp\cap\bD(\Bscr)^\perp=0$ as $\bD(\Cscr)$ is generated by the representable dg modules $h^A_\Cscr= a^* h^A_\Ascr$ and $h^B_\Cscr=\bL b^* h^B_\Bscr$ with $A\in\Ascr$ and $B\in\Bscr$.
			
			Alternatively, one can also argue that $\bD(\Ascr)$ and $\bD(\Bscr)$ generate $\bD(\Cscr)$ using the distinguished triangle \eqref{eq: dist trian in dgmod C} below.
		\end{proof}

	\subsection{Smoothness and properness}
		Recall the following Theorem of Lunts and Schnürer.
		\begin{theorem}[{\cite[Theorem 3.24]{LuntsSchnurer}}]\label{thm: LS}
			Let $
			\Cscr = \left(\begin{smallmatrix}	\Ascr & 0 \\ \phi & \Bscr \end{smallmatrix}\right)
			$.
			The following are equivalent
			\begin{enumerate}
				\item $\Ascr$ and $\Bscr$ are smooth and $\phi$ is perfect,
				\item $\Cscr$ is smooth.
			\end{enumerate}
		\end{theorem}
		In this subsection, we give some sufficient conditions for the smoothness and properness of directed dg category when $n>2$.
		The conditions we give are far from being necessary, but suffice for our purposes.
		We elaborate a bit on necessary conditions in Appendix \ref{Asec: more on smoothness}.
		
		Our goal is to prove the following proposition.
		
		\begin{proposition}\label{prop: smoothness directed dg cat}
			Let $\Cscr$ be a directed dg category as in Equation \eqref{eq: lower matrix dg cat}.
			Assume the dg categories $\Ascr_i$ are smooth and the dg bimodules $\phi_{ij}$ are right perfect.
			Then, the directed dg category $\Cscr$ is smooth.
			Moreover, if the $\Ascr_i$ are proper, then so is $\Cscr$.
		\end{proposition}
		
		We will prove this by induction on $n$, reducing to Theorem \ref{thm: LS}. 
		The main ingredient in proving the above proposition is the following result, which characterises perfect modules over directed dg categories for $n=2$.
		
		\begin{proposition}\label{prop: perf iff}
			Let
			$
			\Cscr = \left(\begin{smallmatrix}	\Ascr & 0 \\ \phi & \Bscr \end{smallmatrix}\right)$ and denote by $a:\Ascr\hookrightarrow\Cscr$ and $b:\Bscr\hookrightarrow\Cscr$ the inclusion functors. 
			Then, a dg $\Cscr$-module $M$ is perfect if and only if the dg modules $b_*M$ and $\cone(b_*M\otimes_{\Bscr}^{\bL}\phi \to a_*M)$ are perfect, where the morphism is induced by the $\Cscr$-action\footnote{With notation of Lemma \ref{lem: functors A,B,C} below, it is obtained by precomposing $\mu$ with the natural morphism $M_\Bscr\otimes^{\bL}_\Bscr\phi\to M_\Bscr\otimes_\Bscr\phi$.}.
			
			In particular, whenever $\phi$ is right perfect, $M$ is perfect if and only if its restrictions $a_*M$ and $b_*M$ are perfect.
		\end{proposition}
		
		In the remainder of this subsection, we state some preliminary lemmas after which we prove both propositions.
		
		\begin{lemma}[{See also \cite[\S3.2]{LuntsSchnurer}}]\label{lem: functors A,B,C}
			With notation as in Proposition \ref{prop: perf iff}, there is an equivalence of dg categories between $\dgMod\Cscr$ and the dg category consisting of triples $(M_\Ascr,M_\Bscr, \mu)$ where $M_\Ascr$ and $M_\Bscr$ are, respectively, a dg $\Ascr$-and $\Bscr$-module and $\mu:M_\Bscr\otimes_\Bscr\phi\to M_\Ascr$ is a morphism of dg $\Ascr$-modules.
			
			Moreover, $a:\Ascr\hookrightarrow\Cscr$ and $b:\Bscr\hookrightarrow\Cscr$ induce the following dg functors:
			\begingroup
			\allowdisplaybreaks
			\begin{align*}
				&\begin{aligned}
					a_*:=\Res_a:\dgMod\Cscr&\to\dgMod\Ascr,\\
					(M_\Ascr,M_\Bscr, \mu)&\mapsto M_\Ascr,						
				\end{aligned}\\
				&\begin{aligned}
					b_*:=\Res_b:\dgMod\Cscr&\to\dgMod\Bscr,\\
					(M_\Ascr,M_\Bscr, \mu)&\mapsto M_\Bscr,						
				\end{aligned}\\
				&\begin{aligned}
					a^*:=\Ind_a:\dgMod\Ascr&\to\dgMod\Cscr,\\
					M&\mapsto (M,0,0),
				\end{aligned}\\
				&\begin{aligned}
					b^*:=\Ind_b:\dgMod\Bscr&\to\dgMod\Cscr,\\
					M&\mapsto (M\otimes_\Bscr\phi,M,\id),
				\end{aligned}\\
				&\begin{aligned}
					a^!:\dgMod\Ascr&\to\dgMod\Cscr,\\										
					M&\mapsto (M, B\mapsto\Hom(\phi(-,B),M), \textrm{evaluation}),
				\end{aligned}\\
				&\begin{aligned}
					b^!:\dgMod\Bscr&\to\dgMod\Cscr,\\	
					M&\mapsto (0, M, 0)\  (= C\mapsto\Hom(b_*h^C, M)),
				\end{aligned}
			\end{align*}
			\endgroup
			and 
			\[
			a^*\dashv a_*\dashv a^!,\quad b^*\dashv b_* \dashv b^!.
			\]
		\end{lemma}
		
		\begin{lemma}\label{lem: restr h-proj}
			With notation as in Proposition \ref{prop: perf iff}, the dg functor $b_*$ preserves h-projective dg modules and when $\phi$ is right h-projective so does $a_*$.
		\end{lemma}
		\begin{proof}
			Considering Lemma \ref{lem: functors A,B,C}, we observe that $b^!$ preserves acyclics and $a^!$ does so when $\phi$ is right h-projective.
			The claim follows from this by abstract nonsense. 
		\end{proof}
		
		
		\begin{proof}[Proof of Proposition \ref{prop: perf iff}]
			We may assume that $\phi$ is h-projective as bimodule.
			Indeed, let $\widetilde{\phi}\to\phi$ be an h-projective resolution.
			As the vertical arrow in the commutative diagram
			\[
			\begin{tikzcd}[sep=tiny]
				&& \left(\begin{smallmatrix}	\Ascr & 0 \\ \widetilde{\phi} & \Bscr \end{smallmatrix}\right) \\
				\left(\begin{smallmatrix}	\Ascr & 0 \\ 0 & \Bscr \end{smallmatrix}\right) \\
				&& \left(\begin{smallmatrix}	\Ascr & 0 \\ \phi & \Bscr \end{smallmatrix}\right)
				\arrow[from=2-1, to=3-3]
				\arrow[from=2-1, to=1-3]
				\arrow[from=1-3, to=3-3]
			\end{tikzcd}
			\]
			is a quasi-equivalence, we may replace $\phi$ by $\widetilde{\phi}$. (We can `push down' the result from $\left(\begin{smallmatrix}	\Ascr & 0 \\ \widetilde{\phi} & \Bscr \end{smallmatrix}\right)$ to $\left(\begin{smallmatrix}	\Ascr & 0 \\ \phi & \Bscr \end{smallmatrix}\right)$).
			Moreover, as we work over a field, $\phi$ is also h-projective as right (and left) dg module.
			
			To prove the claim, we may also assume $M$ to be h-projective, replacing it by an h-projective resolution if necessary.
			Taking the cone of the counit $b^*b_*\to1$ in $\dgMod\Cscr$ gives the triangle
			\begin{equation}\label{eq: triangle induced by counit}
				\left(\begin{smallmatrix}	M_\Bscr\otimes_{\Bscr}\phi \\ M_\Bscr	\end{smallmatrix}\right)\to 
				\left(\begin{smallmatrix}	M_\Ascr \\ M_\Bscr	\end{smallmatrix}\right)\to 
				\left(\begin{smallmatrix}	\cone(M_\Bscr\otimes_{\Bscr}\phi\to M_\Ascr) \\ \cone(\id_{M_\Bscr})	\end{smallmatrix}\right) \to
			\end{equation}		
			inducing a distinguished triangle in $\bD(\Cscr)$.
			As
			\[
			\left(\begin{smallmatrix}	\cone(M_\Bscr\otimes_{\Bscr}\phi\to M_\Ascr) \\ 0	\end{smallmatrix}\right) \to
			\left(\begin{smallmatrix}	\cone(M_\Bscr\otimes_{\Bscr}\phi\to M_\Ascr) \\ \cone(\id_{M_\Bscr})	\end{smallmatrix}\right)
			\]
			is a quasi-isomorphism (exactness can be checked componentwise), $M$ is h-projective and $\phi$ is right h-projective, we have by Lemmas \ref{lem: functors A,B,C} and \ref{lem: restr h-proj} that the distinguished triangle induced by \eqref{eq: triangle induced by counit} has
			\begin{align*}
				\text{leftmost term} &= \bL b^*(b_*M), \\
				\text{middle term} &= M, \\ 
				\text{rightmost term} &= a^*(\cone(b_*M\otimes_{\Bscr}^{\bL}\phi \to a_*M)).
			\end{align*}
			The claim now follows from Lemmas \ref{lem: nice functor preserving or reflecting compactness}, \ref{lem: compact sod} and Proposition \ref{prop: SOD D(directed)}.
			(Both  $ a^*$ and $\bL b^*$ reflect and preserve perfectness, as they are fully faithful, commute with direct sums and have right adjoints that also commute with direct sums.)
		\end{proof}
		\begin{remark}
			The last claim of the proposition (i.e.\ the case when $\phi$ is right perfect) can also be proven as follows.
			Both $a_*$ and $b_*$ preserve perfectness as they map representable dg modules to perfect ones.
			The requirement that $\phi$ is right perfect is necessary here as 
			\[
			a_*h^B_\Cscr=\phi(-,B)\quad\text{for }B\in\Bscr.
			\]
			The fact that they jointly reflect perfectness can be seen making use of the following distinguished triangle
			\begin{equation}\label{eq: dist trian in dgmod C}
				a^*(b_*M\otimes^{\bL}_\Bscr\phi)\to a^*a_*M\oplus \bL b^*b_*M\to M\to
			\end{equation}
			in $\bD(\Cscr)$. 
			Its existence is shown by reducing to the case $M$ and $\phi$ h-projective, as above, and inferring it from the short exact sequence
			\[
			0\to a^*(b_*M\otimes_\Bscr\phi)\to a^*a_*M\oplus b^*b_*M\to M\to 0
			\]
			in $\dgMod\Cscr$.
		\end{remark}
		
		\begin{proof}[Proof of Proposition \ref{prop: smoothness directed dg cat}]
			We prove this by induction on $n$.
			The case $n=2$ follows from Lemma \ref{lem: right perfect sometimes gives perfect} and Theorem \ref{thm: LS}.
			
			So, let $n>2$. 
			Define 
			\begin{align*}
				&\Ascr := \textrm{the full subcategory on the objects }\Obj\Ascr_0\sqcup\Obj\Ascr_1, \\
				&\phi_k := {}_{\Ascr_k}\Cscr_\Ascr\textrm{ the restriction of the diagonal bimodule}, \\
				&\phi_{ik}\otimes_{\Ascr_k}\phi_k\to \phi_i \textrm{ induced by the composition of }\Cscr.
			\end{align*}
			Then, 
			\[
			\Ascr = 
			\begin{pmatrix}
				\Ascr_0 & 0 \\
				\phi_{10} & \Ascr_1
			\end{pmatrix}
			\]
			and 
			\[
			\Cscr = 
			\begin{pmatrix}
				\Ascr		& 0		& 	0 	& \dots & 0\\
				\phi_{2}	& \Ascr_2 & 0 	& \dots & 0\\
				\phi_{3}	& \phi_{32}	& \Ascr_3 & \dots & 0\\
				\vdots & \vdots & \vdots & \ddots & \vdots \\
				\phi_{n-1}	& \phi_{n-1,2}	& \phi_{n-1,3} & \dots & \Ascr_{n-1} \\
			\end{pmatrix}.
			\]
			By induction, $\Ascr$ is smooth and we will be done, again by induction, if we can show that the $\phi_i$ are right perfect. 
			But this follows immediately from Proposition \ref{prop: perf iff} as, denoting the inclusions $\Ascr_i\hookrightarrow\Ascr$ by $a_i$, 
			\begin{align*}
				&a_{0,*}\phi_k(-, A_k) = \phi_{k0}(-, A_k),\\
				&a_{1,*}\phi_k(-, A_k) = \phi_{k1}(-, A_k),\qquad\text{for }A_k\in\Ascr_k.
			\end{align*}
			
			Lastly, assume in addition that the $\Ascr_i$ are proper.
			As the $\Ascr_i$ are smooth, the bimodules $\phi_{ij}$ are perfect by Lemma \ref{lem: right perfect sometimes gives perfect}. 
			Therefore, $\phi_{ij}(A_j, A_i)$ is $\kk$-perfect for every $A_i\in\Ascr_i$ and $A_j\in\Ascr_j$ (this holds for the representable dg modules hence also for the perfect dg modules).
			Thus, all the hom-complexes in $\Cscr$ are $\kk$-perfect, so by definition $\Cscr$ is proper.
		\end{proof}

\section{Gluing hypercubes of dg categories}\label{sec: gluing dg cat}
		Let $[n]:=\{0,\dots, n-1\}$ (and $[0]:=\varnothing$) and consider a \emph{hypercube} (or \emph{$n$-cube} when we want to emphasise the dimension) of dg categories
		\[
		\underline{\Ascr}:=\underline{\Ascr}_n:=\{\Ascr_I,V_{I,j}\}_{I\subset[n],j\notin I},
		\]
		i.e.\ a collection 
		\[
		\{\Ascr_I \}_{I\subset [n]} 
		\] 
		of dg categories together with dg functors\footnote{
			We use the following notation. If $I\subset[n]$ consists of elements $i_1<\dots<i_r$, we write $\Ascr_{i_1,\dots,i_r}$ for $\Ascr_I$. 
			Moreover, $\widehat{i}_k$\ indicates that the element $i_k$ is omitted from the set.
		}
		\[
		V_{{\{i_1,\dots,\widehat{i}_k,\dots,i_r\}},i_k}: \Ascr_{i_1,\dots,\widehat{i}_k,\dots,i_r}\to \Ascr_{i_1,\dots,i_r}
		\] 
		that \emph{strictly}\footnote{It is more natural to weaken this, only allowing the squares to commute up to some higher data. One need then take this data into consideration when totalising. However, in our setting we can always reduce to strictly commuting squares, so we do this here for simplicity.} commute.
		In short, it is a functor $2^{[n]}\to\dgcat$, where we view $2^{[n]}$ (:= the power set of $[n]$) as a category obtained from the poset ordered by inclusion.
		We will usually write $V_{i_k}$ for $V_{{\{i_1,\dots,\widehat{i}_k,\dots,i_r\}},i_k}$.
		Moreover, for $I=\{i_1<\dots<i_r\}$, we usually write $V_I$ or $V_{i_1\dots i_r}$ for any composition (they are all equal by assumption) of $V_i$'s starting at a $\Ascr_{I'}$, with $I'\subset[n]$ disjoint from $I$, and ending at $\Ascr_{I'\sqcup I}$.
		Giving the same name to different functors is sloppy, but it makes the notation easier.
		As everything strictly commutes, not much confusion can arise in practice, so we hope the reader forgives us.
		
		In addition, we also consider the \emph{punctured hypercube} obtained by removing the dg category $\Ascr_\varnothing$.
		We denote the resulting data by $\underline{\Ascr}_n^\circ$.
		
		In this section, we do the following:
		\begin{itemize}
			\item we construct a dg category $\Glue(\underline{\Ascr}_n^\circ)$ glued from the punctured $n$-cube,
			\item give necessary and sufficient conditions for a natural dg functor $\Ascr_\varnothing\to \Glue(\underline{\Ascr}_n^\circ)$ to be quasi fully faithful,
			\item find suitable conditions to ensure smoothness/properness of $\Glue(\underline{\Ascr}_n^\circ)$. 
		\end{itemize}
		Lastly, we discuss two constructions, stacking and extending, of hypercubes.
		
		For the remainder of this section, let $\underline{\Ascr}_n$ denote a hypercube of dg categories.
		
	\subsection{Totalisation}\label{subsec: Tot}
		We give a procedure to associate a complex to the data of a hypercube with complexes as its vertices.
		This will then easily extend to hypercubes with values in $Z^0\dgMod\Ascr$ for $\Ascr$ a dg category.
		The procedure is a slight modification of the one given by Khovanov in \cite[\S3.3]{KhovanovJones}. 
		We modify it by simultaneously taking the total complex, so that we obtain a complex as opposed to a double complex.
		Moreover, we also `flip' the signs in some sense.
		
		Let $W$ be the graded algebra generated by $X_i$ of degree $-1$ and $\partial_i:=\partial/\partial X_i$ of degree $1$, for $i\in \{0,\dots, n-1\}$, satisfying the relations
		\begin{align*}
			&X_iX_j + X_jX_i = 0, \\
			&\partial_i\partial_j + \partial_j\partial_i = 0, \\
			&\partial_i X_j + X_j\partial_i = \begin{cases}
				1\text{ if }i=j, \\
				0\text{ if }i\neq j.
			\end{cases}
		\end{align*}
		We will make use of the $X_i$'s and $\partial_i$'s to take care of the minus signs in the differential of our complex.
		For this we view below products of $X_i$'s as elements in $W/(\sum_i W\partial_i)$ and consider the endomorphism induced by multiplication by $\partial_i$ on $W/(\sum_i W\partial_i)$. 
		
		Let $\underline{A}=\{A_I^\bullet, \alpha_{I,l}\}_{I\subset [n],l\notin I}$ be a hypercube of $\kk$-complexes.
		We associate a single complex $t(\underline{A})$ to $\underline{A}$, called its \emph{totalisation}, as follows:
		\[
		t(\underline{A}) := \bigoplus_{I\subset [n]} \left(\prod_{l\in [n]\backslash I} X_l \right) \otimes_\kk A_I^\bullet
		\]
		as graded $\kk$-module\footnote{In principle we should write $ \left(\kk\prod_{l\in [n]\backslash I} X_l \right) \otimes_\kk A_I$, but we do not.} with differential given by\footnote{Recall from footnote \ref{foot: monoidal structure C(k)} that signs arise when evaluating this on elements.}
		\[
		d := \bigoplus_{I\subset [n]} \left(\left(\bigoplus_{l\in [n]\backslash I} (\partial_l\cdot -)\otimes_\kk\alpha_{I,l}  \right) + 1\otimes_\kk d_{A_I^\bullet} \right).
		\]
		For notational ease we will often simply write $X_{[n]\backslash I}$ instead of $\left(\prod_{l\in [n]\backslash I} X_l \right)$.
		The following lemma is obligatory.
		\begin{lemma}
			We have $d\circ d=0$.
		\end{lemma}
		\begin{proof}
			This follows as $\partial_l\partial_{l'}=-\partial_{l'}\partial_l$, for $l\neq l'$.
		\end{proof}
		\begin{remark}
			In the construction we did not have to choose any explicit signs in the differential.
			The signs come about when choosing an explicit ordering of $[n]$ and hence of the factors in $\prod_l X_l$.
			However, it is worthwhile to note that any explicit choice, such that every square in the hypercube anti-commutes, gives an isomorphic complex (the algorithm given in \cite{Rickard-stackexchange} also works in this setting). 
			Alternatively, one can argue by viewing the hypercube as a CW-complex and a choice of signs as a cellular 1-cochain whose differential has value -1 on every face.
			Different sign choices then differ by a coboundary as the hypercube is contractible, which induces an isomorphism of the corresponding chain complexes, see e.g. \cite[Lemma 2.2]{OzsvathRasmussenSzabo}. 
		\end{remark}
		We give some examples for small $n$ (omitting the $X_i$'s and $\partial_i$'s from the notation).
		\begin{examples}\label{exs: t} 
			We have the following.
			\begin{enumerate}
				\item A 0-cube of complexes is the same as a complex $A^\bullet$. 
				Clearly, $t(A^\bullet)=A^\bullet$.
				\item A 1-cube of complexes is the same as a morphism of complexes $f:A^\bullet\to B^\bullet$.
				We have\footnote{This is why we put the $\prod_l X_l$'s as first factor in the tensor product.}
				\[
				t(A^\bullet\xrightarrow{f}B^\bullet) = \cone(f).
				\]
				\item For a 2-cube, i.e.\ a square,
				\[
				\Box=
				\begin{tikzcd}
					B^\bullet\ar[r, "f'"] & D^\bullet \\
					A^\bullet\ar[u, "g"]\ar[r, "f"] & C^\bullet\ar[u, "g'"']
				\end{tikzcd}
				\]
				of complexes we have
				\[
				t(\Box)=\Tot(\ldots \to 0 \to A^\bullet \xrightarrow{\left(\begin{smallmatrix} \mp g \\ \pm f \end{smallmatrix}\right)} B^\bullet\oplus C^\bullet \xrightarrow{\left(\begin{smallmatrix} f' &  g' \end{smallmatrix}\right)} D^\bullet \to 0 \to \ldots),
				\]
				where $D^\bullet$ is in the zeroth position and the signs depend on the chosen explicit ordering of $[2]$.
			\end{enumerate}
		\end{examples}
		
		For a morphism $\{f_I\}_{I\subset[n]}:\underline{A}\to \underline{B}$ (which is just a natural transformation between the corresponding functors) we obtain a morphism of complexes
		\begin{equation}\label{eq: functoriality t(-)}
			\oplus_{I\subset [n]} \left( 1\otimes_\kk f_I \right): t(\underline{A})\to t(\underline{B}),
		\end{equation}
		as
		\[
		( 1\otimes_\kk f_I )( 1\otimes_\kk d_{A_I} ) = ( 1\otimes_\kk d_{B_I} )( 1\otimes_\kk f_I )
		\]
		and
		\[
		( 1\otimes_\kk f_I )( (\partial_l\cdot -)\otimes_\kk\alpha_{I,l} ) = ( (\partial_l\cdot -)\otimes_\kk\beta_{I,l} )( 1\otimes_\kk f_I ),
		\]
		since the $\{f_I\}$ have degree zero.
		Hereby, we obtain a functor
		\[
		t:\Fun(2^{[n]},\Csf(\kk))\to \Csf(\kk).
		\]
		Moreover, we can extend \eqref{eq: functoriality t(-)} to the case when the $\{f_I\}$ are not necessarily closed or of degree zero (but of course we do not obtain morphisms of complexes in this way). 
		This allows us, by simply taking the above construction objectwise, to extend the $t$-construction to hypercubes of dg modules over a dg category $\Ascr$ with closed degree zero morphisms as edges.
		We obtain a functor 
		\begin{equation}\label{eq: extending t to dg modules}
			t:\Fun(2^{[n]},Z^0\dgMod\Ascr)\to Z^0\dgMod\Ascr.
		\end{equation}
		In more detail, for a hypercube of dg modules $\underline{M}:2^{[n]}\to Z^0\dgMod\Ascr$, we define
		\[
		t(\underline{M})(A):=t(\underline{M}(A))\quad\text{for }A\in\Ascr
		\]
		($\underline{M}(A)$ is the composition of $\underline{M}$ with `evaluation at $A$')
		and define the action of $\Ascr$ via\footnote{By definition $m\cdot a:=(-1)^{|a||m|} M(a)(m)$, so $a$ acts via $\oplus_{I\subset [n]} \left( 1\otimes_\kk M_I(a) \right)$, the various minus signs cancel each other.} 
		\[
		\left( X_{[n]\backslash I} \otimes_\kk m \right)\cdot a := \left( X_{[n]\backslash I} \otimes_\kk (m\cdot a) \right).
		\]
		
		\begin{remark}\label{rem: modification t()}
			Further along, we will also use the above procedure for hypercubes indexed by index sets different from $2^{[n]}$. 
			For example we will use indexing by 
			\[
			J_{ij}:= \{ I\subset[n] \mid \{i,j\}\subset I\subset\{i,i+1,\dots,j-1,j\} \}.
			\]
			Then, we change $[n]\backslash I$ to $\{i,i+1,\dots,j-1,j\}\backslash I$ in the construction.
			
			More generally, the above procedure works for any set $K$ and subset $J\subset 2^K$ of `hypercube shape', if we change $[n]-I$ by $J_{max}-I$, where $J_{max}\in J$ is the subset of maximal cardinality. 
		\end{remark}
		
	\subsection{The gluing}
		We associate a dg category $\Glue(\underline{\Ascr}_n^\circ)$ to the punctured hypercube $\underline{\Ascr}_n^\circ$, which we will think of as the dg category obtained by gluing the punctured hypercube.
		This is constructed as the dg category of twisted complexes, i.e.\ the pretriangulated hull, over a directed dg category that we construct first.
		An alternative, more explicit, definition, which only considers a subcategory of the twisted complexes, is discussed in Appendix \ref{Asec: alternative gluing}.
		In the case $n=2$, i.e.\ a square, we obtain the usual gluing of two dg categories along a bimodule (of a specific type).
		In \S\ref{subsec: gluing squares}, we compare our conventions for the gluing with others found in the literature.
		
		\subsubsection{Generalised arrow dg category}
		We start by associating a directed dg category to the punctured hypercube.
		In order to define the morphism complexes, let us consider the following subsets of the power set of $[n]$ for any $0\leq i\leq j\leq n-1$:
		\[
		J_{ij}:= \{ I\subset[n] \mid \{i,j\}\subset I\subset\{i,i+1,\dots,j-1,j\} \}.
		\]
		(These form a partition of $2^{[n]}\backslash\varnothing$.)
		
		Define the \emph{generalised arrow dg category}\footnote{
			The name comes from viewing a directed dg category $ \Cscr = \begin{pmatrix} \Ascr & 0 \\ \phi & \Bscr \end{pmatrix}$ as an `arrow dg category', as it effectively makes elements of the dg bimodule $\phi$ into arrows. 
			Directed dg categories can then be thought of as `generalised arrow dg categories'.	
		} 
		$\Gac(\underline{\Ascr}_n^\circ)$ as the dg category having objects 
		\[
		\bigsqcup_{i=0}^{n-1}\Obj\Ascr_i.
		\]
		For $i\leq j$ denote by $\underline{\Ascr}_{J_{ij}}$ the subgraph $\{\Ascr_I\}_{I\subset J_{ij}}$ of $\underline{\Ascr}_n$ viewed as a $((j-i-1)\vee 0)$-cube (we will identify $J_{ij}$ with $[(j-i-1)\vee 0]$ for indexing the hypercube\footnote{$\{(I-\{i,j\})-i-1 \mid I\in J_{ij}\}=[(j-i-1)\vee 0]$, essentially this identification just results in the change we mentioned in Remark \ref{rem: modification t()}.}).
		We can view $\underline{\Ascr}_{J_{ij}}$ as a hypercube of dg $(\Ascr_j,\Ascr_i)$-bimodules, by first considering the vertices as diagonal bimodules and then suitably restricting them along the edges of the hypercube.
		Applying the $t$-construction then yields a dg $(\Ascr_j,\Ascr_i)$-bimodule that can be evaluated on objects $A_i\in\Ascr_i$ and $A_j\in\Ascr_j$.
		We define the morphism complex
		\begin{equation}\label{eq: Hom in Gac}
			\Hom_{\Gac}(A_i, A_j) :=
			\begin{cases}
				t(\underline{\Ascr}_{J_{ij}})(A_i,A_j) & \text{ if } i\leq j, \\
				0 & \text{ if } i>j.
			\end{cases}
		\end{equation}
		Concretely, for $i\leq j$,
		\[
		\Hom_{\Gac}(A_i, A_j) := \bigoplus_{I\subset J_{ij}} \left(\prod_{l\in \{i,\dots j\}\backslash I} X_l \right) \otimes_\kk \Ascr_I(V_{I\backslash\{i\}}A_i,V_{I\backslash\{j\}}A_j)
		\]
		with differential
		\[
		d := \bigoplus_{I\subset J_{ij}} \left(\left( \bigoplus_{l\in \{i,\dots j\}\backslash I} (\partial_l\cdot -)\otimes_\kk V_{I,l} \right) + 1\otimes_\kk d_{\Ascr_I} \right).
		\]
		Sometimes we suppress $\Pi_{l} X_l$ from the notation, denoting a morphism $\left( \Pi_{l\in \{i,\dots j\}\backslash I}X_l \otimes_\kk f_I \right)$ living in the ${I\subset J_{ij}}$ summand of $\Hom(X, Y)$ by $f_I$.
		Moreover, we will use the subscript to keep track of the degree of $f_I$ as a morphism in the generalised arrow category, it has degree $|f_I|-|\{i,\dots j\}\backslash I|$ (here, as well as in the sequel, $|f_I|$ denotes the degree considered as a morphism of $\Ascr_I$).
		
		We define the composition as follows.
		For $i\leq j\leq k$ (it is the zero map otherwise), $A_i\in\Ascr_i$, $A_j\in\Ascr_j$ and $A_k\in\Ascr_k$, the composition
		\[
		\Hom_{\Gac}(A_j,A_k)\otimes_\kk\Hom_{\Gac}(A_i,A_j) \to \Hom_{\Gac}(A_i,A_k)
		\]
		sends\footnote{\label{foot: concat}
			On the $\Pi_{l}X_l$ part, which we did not write, we simply do concatenation \[(\Pi_{l'\in \{j,\dots k\}\backslash I'}X_{l'}) (\Pi_{l\in \{i,\dots j\}\backslash I}X_{l}) = \pm\Pi_{l\in \{i,\dots k\}\backslash (I\cup I')}X_l.\] (As $\{j,\dots k\}\backslash I' \sqcup \{i,\dots j\}\backslash I = \{i,\dots k\}\backslash (I\cup I')$.)
			Note that when making signs explicit, it is best to order the product of the $X_i$'s with the order opposite to the usual one of $\NN$. 
			This ensures that there is no minus sign when concatenating.
		}
		\begin{multline}\label{eq: composition}
			g_{I'} \otimes_\kk f_{I} \mapsto (-1)^{|g_{I'}| |\{i,\dots j\}\backslash I|} V_{I\backslash I'}(g_{I'}) V_{I'\backslash I}(f_{I})\\\quad (\text{this element lives in the } {I\cup I'\subset J_{ik}} \text{ summand of }\Hom_{\Gac}(A_i,A_k)).
		\end{multline}
		The unit morphism of an object $A_i\in \Ascr_i$ is given by 
		\[
		1\otimes1_{A_i}\in 1\otimes_\kk \Ascr_i(A_i,A_i)= \Hom_{\Gac}(A_i,A_i).
		\]
		
		\begin{lemma}\label{lem: ass/un cube level}
			The composition is associative and unital, and it is a morphism of chain complexes.
		\end{lemma}
		\begin{proof}
			This is a graded morphism as $\{j,\dots k\}\backslash I' \sqcup \{i,\dots j\}\backslash I = \{i,\dots k\}\backslash (I\cup I')$, see also footnote \ref{foot: concat}.
			Showing that it commutes with the differential is tedious but straightforward.
			Moreover, checking associativity and unitality boils down to a routine verification that the signs match up. 
			(For associativity note that $|I\cup I'| = |I|+|I'|-1$.)
		\end{proof}
		
		\subsubsection{Gluing dg category}
		We define the \emph{gluing dg category} $\Glue(\underline{\Ascr}_n^\circ)$ of the punctured hypercube $\underline{\Ascr}_n^\circ$ as the dg category of twisted complexes over the generalised arrow dg category
		\[
		\Glue(\underline{\Ascr}_n^\circ):= \tw(\Gac(\underline{\Ascr}_n^\circ)).
		\]
		An alternative, more concrete, description is given in Appendix \ref{Asec: alternative gluing}.
		
		\begin{remark}
			The gluing of hypercubes also appeared recently in the setting of stable $\infty$-categories in \cite{ChristDyckerhoffWalde}, as part of their `totalizations of categorical multicomplexes'.
			The gluing dg category here appears as one of the terms in the complex obtained by `totalizing' the hypercube diagram as in loc.~cit.~
			It is not immediately clear how to naturally obtain the other terms, and thereby the complex, in this framework at the moment.
		\end{remark}
		\begin{remark}
			The ordering of the vertices matters, see Remark \ref{rem: order matters}.
		\end{remark}
		
		As an immediate corollary of Propositions \ref{prop: SOD directed} and \ref{prop: SOD D(directed)} we obtain (note $\bD(\tw(-))=\bD(-)$).
		
		\begin{corollary}\label{cor: SODs Glue}
			We have semi-orthogonal decompositions
			\begin{align*}
				[\Glue(\underline{\Ascr}_n^\circ)]&=\langle [\tw\Ascr_0], \dots, [\tw\Ascr_{n-1}]\rangle, \\
				\bD(\Glue(\underline{\Ascr}_n^\circ))&=\langle \bD(\Ascr_0), \dots, \bD(\Ascr_{n-1})\rangle.
			\end{align*}
		\end{corollary}
		
	\subsection{An example: gluing squares}\label{subsec: gluing squares}
		Consider a punctured square
		\begin{equation}\label{eq: .->.<-. of dg cats}
			\begin{tikzcd}
				\Ascr_0 \arrow[r, "V_1"]  & \Ascr_{01} \\
				& \Ascr_1 \arrow[u, "V_0"]
			\end{tikzcd}
		\end{equation}
		and let $\phi:={}_{V_0}{\Ascr_{01}}_{V_1}$ be the dg $(\Ascr_1,\Ascr_0)$-bimodule obtained by restricting the diagonal.
		The following notions of a gluing of $\Ascr_0$ and $\Ascr_1$ along $\phi$ can be found in the literature.
		\begin{itemize}
			\item In Tabuada \cite{Tabuada2007} the directed dg category 
			\[
			\Bscr := \begin{pmatrix}	
				\Ascr_0 & 0 \\
				\phi & \Ascr_1
			\end{pmatrix}
			\]
			is considered (as upper triangular matrix $\left( \begin{smallmatrix}
				\Ascr_1 & \phi \\ 
				0 & \Ascr_0 
			\end{smallmatrix}\right)$).
			\item In Orlov \cite{Orlov} the gluing is defined as the pretriangulated hull of $\Bscr$. 
			This is precisely $\Glue(\underline{\Ascr}^{\circ}_2)$.
			\item Lastly, we could define the gluing as the full dg subcategory of $\tw\Bscr$ consisting of objects $\cone(M_0\xrightarrow{\mu} M_1)$, i.e\ the twisted complex
			$(M_1\oplus M_0[1],\left( \begin{smallmatrix}
				0 & \mu
				\\
				0 & 0 
			\end{smallmatrix}\right) )$,
			where $M_0\in\Ascr_0$, $M_1\in\Ascr_1$ and $\mu\in \phi(M_0,M_1)$ is a closed morphisms of degree zero. 
			This is the alternative gluing discussed in Appendix \ref{Asec: alternative gluing}.
			
			Moreover, this is essentially the definition of Kuznetsov and Lunts in \cite{KuznetsovLunts}.
			However, their gluing $\Ascr_0\times_\phi\Ascr_1$ is defined as the full dg subcategory of $\tw \left(\begin{smallmatrix}	\Ascr_0 & 0 \\ \phi[-1] & \Ascr_1 \end{smallmatrix}\right)$ (note the shift in the bimodule) consisting of twisted complexes of the form $(M_1\oplus M_0,\left( \begin{smallmatrix}	0 & \mu \\ 0 & 0 \end{smallmatrix}\right) )$, where $M_0\in\Ascr_0$, $M_1\in\Ascr_1$ and $\mu\in\phi(M_0,M_1)$ is a closed morphism of degree zero. 
			This explains the `strange' minus signs in their composition and differential; they are artefacts of working with the shifted bimodule $\phi[-1]$ (and hence the differential and left module structure of $\phi$ obtains extra minus signs).
			On the other hand, they have no signs coming from shifts in the twisted complexes.
		\end{itemize}
		
		\begin{remark}\label{rem: order matters}
			Let us denote $\Ascr_0\sqcup_\phi\Ascr_1:=\Glue(\underline{\Ascr}^{\circ}_2)$ in this remark to make the ordering of the vertices clear.
			In general, $\Ascr_0\sqcup_{\Ascr_{01}}\Ascr_1$ and $\Ascr_1\sqcup_{\Ascr_{01}}\Ascr_0$ are different.
			Take for example $\Ascr_{01}$ to be an enhancement of the perfect complexes over $\mathbb{P}^{1}_{\kk}$ and $\Ascr_{i}$ the dg subcategory `triangularly generated' by $\mathcal{O}_{\mathbb{P}^{1}}(i)$.
			Then, $\Ascr_0\sqcup_{\Ascr_{01}}\Ascr_1$ gives back the perfect complexes over $\mathbb{P}^{1}$ whilst $\Ascr_1\sqcup_{\Ascr_{01}}\Ascr_0$ gives the perfect complexes over $\kk^{2}$.
			These are not quasi-equivalent (the homotopy category former is indecomposable whilst that of the latter is not).
		\end{remark}	
		
		\begin{remark}
			One can show that the homotopy pullback $\Ascr_0\times^h_\phi\Ascr_1$ of the diagram \eqref{eq: .->.<-. of dg cats} (for the model structure on $\dgcat$ for which the weak equivalences are the quasi-equivalences and the fibrations are `componentwise surjections giving isofibrations on $H^0$') can be identified with the full subcategory of $\Glue(\underline{\Ascr}^{\circ}_2)$ consisting of those twisted complexes $\cone(M_0\xrightarrow{\mu} M_1)$ for which $\mu$ is a homotopy equivalence when viewed as a morphism in $\Ascr_{01}$, see e.g.\ \cite[\S3.3]{CanonacoNeemanStellari}.
			In general the inclusion $\Ascr_0\times^h_{\Ascr_{01}}\Ascr_1 \hookrightarrow \Glue(\underline{\Ascr}^{\circ}_2)$ is not a quasi-equivalence.
			Moreover, it can happen that the latter is smooth whilst the former is not.			
		\end{remark}
		
	\subsection{Acyclic hypercube and quasi fully faithful dg functor}\label{subsec: a and qff}
		We give sufficient and necessary conditions for a natural dg functor $\pi:\Ascr_\varnothing\to \Glue(\underline{\Ascr}^{\circ}_n)$ (which we define below) to be quasi fully faithful.
		The relevant condition is the following.
		\begin{definition}
			A hypercube $\underline{\Ascr}_n$ of dg categories is called \emph{acyclic} if, when viewed as a hypercube of dg $(\Ascr_\varnothing,\Ascr_\varnothing)$-bimodules\footnote{
				Recall that we do this by first considering the vertices as diagonal bimodules and by then suitably restricting them along the edges of the hypercube.	
			}, it yields an acyclic dg $(\Ascr_\varnothing,\Ascr_\varnothing)$-bimodule after applying $t$, i.e.\
			\[
			t(\underline{\Ascr}_n)\in\dgMod(\Ascr_\varnothing^{\op}\otimes_\kk\Ascr_\varnothing)
			\]
			is acyclic.
		\end{definition}
		\begin{remark}
			This is independent of the ordering of the vertices in the hypercube.
		\end{remark}
		
		Let us define a natural dg functor
		\[
		\pi:\Ascr_\varnothing\to \Glue(\underline{\Ascr}_n^\circ).
		\]
		We use notation of \S\ref{subsec: twist}.
		Moreover, for $0\leq i<n$, let temporarily $\underline{i}$ denote $n-i-1$.
		The dg functor is given on objects by mapping $A\in\Ascr_\varnothing$ to the twisted complex
		\[
		\left( \oplus_{i=0}^{n-1}V_{\underline{i}}A[i], \alpha \right),
		\]
		with
		\begin{align*}
			\alpha_{ji}:=&(-1)^{j} 1_{V_{\underline{i}\underline{j}}A}\in \Ascr^0_{\underline{i}\underline{j}}(V_{\underline{i}\underline{j}}A, V_{\underline{i}\underline{j}}B) \subseteq\Hom_{\Gac}^{1+j-i}(V_{\underline{i}}A, V_{\underline{j}}A)\\ =& \Hom_{\ZZ\!\Gac}^{1}(V_{\underline{i}}A[i], V_{\underline{j}}A[j])
		\end{align*}
		for $ j<i$ and zero otherwise, and on morphisms by mapping $f:A\to B$ to
		\begin{align*}
			\pi f :=& \oplus_ i (-1)^{(n-1-i)|f|}V_if \in \oplus_i \Ascr_i^{|f|}(V_iA, V_iB) \\ =& \oplus_i \Hom^{|f|}_{\Gac}(V_iA,V_iB) \subseteq \Hom_{\Glue}^{|f|}(\pi A, \pi B).
		\end{align*}	
		The minus signs compensate for the fact that $\Ascr_{\underline{i}}=\Ascr_{n-1-i}$ sits in the image of $\pi$ with a twist $[i]$.
		To verify that $\pi$ is indeed a dg functor one has to check the following:
		\begin{itemize}
			\item the $\alpha$'s satisfy $d\alpha+\alpha^2=0$,
			\item $\pi$ commutes with the differential,
			\item $\pi$ respects units and composition.
		\end{itemize}
		This is straightforward, we leave it to the motivated reader.
		
		The following shows that the acyclic hypercube condition is `natural'.
		
		\begin{proposition}\label{prop: acyclic hypercube iff qff}
			The hypercube $\underline{\Ascr}_n$ is acyclic if and only if the natural dg functor $\pi:\Ascr_\varnothing\to\Glue(\underline{\Ascr}_n^\circ)$ is quasi fully faithful.
		\end{proposition}
		\begin{proof}
			The inclusion $\underline{\Ascr}_n^\circ\hookrightarrow \underline{\Ascr}_n$ induces, for every $A,$ $B\in\Ascr_\varnothing$, a distinguished triangle
			\[
			t(\underline{\Ascr}_n^\circ)(A,B)	\to t(\underline{\Ascr}_n)(A,B) \to \Ascr_\varnothing(A,B)[n]\xrightarrow{\delta} t(\underline{\Ascr}_n^\circ)(A,B)[1]
			\]
			with boundary morphism
			\[
			\delta:\Ascr_\varnothing(A,B)[n] \to t(\underline{\Ascr}_n^\circ)(A,B)[1],~ f\mapsto \sum_ i (-1)^{(n-1-i)}V_if
			\]
			(see e.g.\ \cite[\href{https://stacks.math.columbia.edu/tag/014I}{Definition 014I}]{stacks-project}).
			
			We claim that there exists an isomorphism 
			\begin{equation}\label{eq: ident t(-) with Hom}
				t(\underline{\Ascr}_n^\circ)(A,B)[1]\cong\Hom_{\Glue(\underline{\Ascr}_n^\circ)}(\pi A,\pi B)[n],
			\end{equation}
			making the diagram
			\[
			\begin{tikzcd}
				\Ascr_\varnothing(A,B)[n]\arrow[dr, "\pi"]\arrow[r, "\delta"]	&	t(\underline{\Ascr}_n^\circ)(A,B)[1]\arrow[d,"\text{Eq.~}\eqref{eq: ident t(-) with Hom}"]	\\
				&	\Hom_{\Glue(\underline{\Ascr}_n^\circ)}(\pi A,\pi B)[n]
			\end{tikzcd}
			\]
			commute. 
			It follows that $t(\underline{\Ascr}_n)$ is acyclic if and only if $\pi$ is quasi fully faithful.
			
			Thus, it remains to show the existence of the isomorphism \eqref{eq: ident t(-) with Hom}.
			For this observe that by definition (as graded modules)
			\[
			t(\underline{\Ascr}_n^\circ)(A,B)[1] = \bigoplus_{i\leq j}\bigoplus_{I\subset J_{ij}}
			\left( X_{[n]\backslash I} \otimes_\kk	\Ascr_I(V_I A, V_I B) \right)[1]
			\]
			(as $2^{[n]}\backslash\varnothing=\sqcup_{i\leq j} J_{ij}$) and
			\begin{align*}
				\Hom_{\Glue(\underline{\Ascr}_n^\circ)}(\pi A,\pi B)[n] &= \bigoplus_{i, j} \Hom_{\Gac(\underline{\Ascr}_n^\circ)}(V_iA,V_jA)[i-j+n] \\
				&= \bigoplus_{i\leq j} t(\underline{\Ascr}_{J_{ij}})(V_iA,V_jA)[i-j+n] \\
				&= \bigoplus_{i\leq j}\bigoplus_{I\subset J_{ij}}
				\left( X_{\{i,\dots,j\}\backslash I}  \otimes_\kk \Ascr_I(V_I A, V_I B) \right)[i-j+n].
			\end{align*}
			The required isomorphism is given on the summand corresponding to $I\subset J_{ij}$ by mapping 
			\[
			X_{n-1}\dots X_{j+1}X_{\{i,\dots,j\}\backslash I}X_{i-1}\dots X_0  \otimes_\kk f \mapsto (-1)^{(n-i-1)|f|+(n-1)|I|-i} X_{\{i,\dots,j\}\backslash I}  \otimes_\kk	f,
			\]
			i.e.\ it is graded (the difference in shifts is nicely compensated by the difference in $X_l$'s), compatible with the differentials and makes the diagram commute.
			Checking this is straightforward but uninteresting, so we leave it to the motivated reader.
		\end{proof}
		\begin{remark}
			Let $I\in J_{ij}$ for some $i\leq j$.
			In $t(\underline{\Ascr}_n^\circ)(A,B)$ all the edges $V_{I,l}$ for $l\in [n]\backslash I$ contribute to the differential, whilst in $\Hom_{\Glue(\underline{\Ascr}_n^\circ)}(A,B)$ only the edges with $i\leq l \leq j$ appear in $\oplus_{i\leq j} t(\underline{\Ascr}_{J_{ij}})(V_iA,V_jA)$. 
			The other $V_{I,l}$'s, with $l<i$ or $j<l$, emerge in the commutator term $[\delta,-]$ of the differential of the twisted complexes.
		\end{remark}
		
		It will be crucial later, when showing that we obtain a categorical resolution, to have some control over the image of the restriction functor associated to $\pi:\Ascr_\varnothing\to\Glue(\underline{\Ascr}_n^\circ)$.
		We finish the subsection with this.
		
		Let us first explain and introduce some notation.
		The inclusion $\Ascr_i\hookrightarrow \Gac(\underline{\Ascr}_n^\circ)\hookrightarrow \Glue(\underline{\Ascr}_n^\circ)$ along with the dg Yoneda embedding allows us to identify $[\Ascr_i]\subseteq\bD(\Glue(\underline{\Ascr}_n^\circ))$.
		This is compatible with the semi-orthogonal decompositions of Corollary \ref{cor: SODs Glue}.
		Moreover, denote by $\phi_{ji}$ the hom-complex defined in Equation \eqref{eq: Hom in Gac} viewed as dg $(\Ascr_j,\Ascr_i)$-bimodule.
		
		\begin{lemma}\label{lem: restriction functor restricted to component of sod}
			Let $\underline{\Ascr}_n$ be a hypercube of dg categories and consider the restriction functor 
			\[
			\Res_\pi:\bD(\Glue(\underline{\Ascr}_n^\circ))\to \bD(\Ascr_\varnothing).
			\]
			The image of $A_i\in[\Ascr_i]$ under $\Res_\pi$ is an iterated cone of its images under $\Res_{V_j}(-\otimes_{\Ascr_i}^{\bL}\phi_{ij})$'s where $V_j$ denotes the edge $\Ascr_\varnothing\to \Ascr_j$. 
		\end{lemma}
		\begin{proof}	
			Consider the restriction dg functor on the level of dg modules, i.e.\ \[\Res_\pi:\dgMod\Glue(\underline{\Ascr}_n^\circ)\to \dgMod\Ascr_\varnothing.\]
			
			For any $A\in\Ascr_\varnothing$ the twisted complex $\pi A$ is an iterated cone of $V_jA$'s.
			This is functorial in $A$.
			Therefore, 
			\[
			(\Res_\pi h^{A_i})(A) = \Hom_{\Glue}(\pi A, A_i)
			\]
			is an iterated extension of $\Hom_{\Glue}(V_jA, A_i)$'s.
			Now observe that
			\begin{align*}
				\Hom_{\Glue}(V_jA, A_i) &= \phi_{ij}(V_jA, A_i) \\
				&= (h^{A_i}\otimes_{\Ascr_i}\phi_{ij} ) (V_jA) \\
				&= (\Res_{V_j}(h^{A_i}\otimes_{\Ascr_i}\phi_{ij} )) (A)
			\end{align*}
			from which the claim follows as everything is functorial in $A$.
		\end{proof}
		\begin{remark}
			This can be done in a more sophisticated way, but the above suffices for our purposes.
			By making `iterated cone' of functors more precise, one can describe the functor $\Res_\pi$ restricted to $\bD(\Ascr_i)$ itself as an iterated cone (making use of \cite[Theorem 7.2]{Toen} essentially reduces this to the above).
		\end{remark}
		
	\subsection{A sufficient condition for smoothness and properness}
		We have 
		\[
		\Gac(\underline{\Ascr}_n^\circ) = 
		\begin{pmatrix}
			\Ascr_0		& 0		& 	0 	& \dots & 0\\
			\phi_{10}	& \Ascr_1 & 0 	& \dots & 0\\
			\phi_{20}	& \phi_{21}	& \Ascr_2 & \dots & 0\\
			\vdots & \vdots & \vdots & \ddots & \vdots \\
			\phi_{n-1,0}	& \phi_{n-1,1}	& \phi_{n-1,2} & \dots & \Ascr_{n-1} \\
		\end{pmatrix}.
		\]
		where the $\phi_{ji}$'s are the hom-complexes defined in Equation \eqref{eq: Hom in Gac} viewed as dg $(\Ascr_j,\Ascr_i)$-bimodules.
		
		\begin{lemma}\label{lem: bimodules right perfect}
			Suppose that the restriction functors 
			\[
			\Res_{V_{I\backslash i}}:\bD(\Ascr_{I})\to \bD(\Ascr_i)\quad\text{for }i\in I\subseteq [n],
			\]
			induced from composing the edges of the hypercube, preserve compactness. 
			Then, the bimodules $\phi_{ji}$ are right perfect.
		\end{lemma}
		\begin{proof}
			For a dg $(\Ascr_j,\Ascr_i)$-bimodule of the form $\phi={}_{V_{I\backslash\{j\}}}(\Ascr_I)_{V_{I\backslash\{i\}}}$, for $i,j\in I\subseteq[n]$, we have $\bL\phi = \Res_{V_{I\backslash\{i\}}} \bL\!\Ind_{V_{I\backslash\{j\}}}$ as, for $M\in\dgMod\Ascr_j$ and $A_i\in\Ascr_i$,
			\begin{align*}
				(M\otimes_{\Ascr_j}\phi)(A_i) &= M\otimes_{\Ascr_j} \Ascr_I(V_{I\backslash\{i\}}(A_i),V_{I\backslash\{j\}}(-)) \\
				&= (\Ind_{V_{I\backslash\{j\}}}M)(V_{I\backslash\{i\}}A_i) \\
				&= (\dgMod \Ascr_I) (h^{V_{I\backslash\{i\}}A_i}, \Ind_{V_{I\backslash\{j\}}}M) \\ 
				&= (\dgMod \Ascr_I) (\Ind_{V_{I\backslash\{i\}}}h^{A_i}, \Ind_{V_{I\backslash\{j\}}}M) \\
				&= (\dgMod \Ascr_i) (h^{A_i}, \Res_{V_{I\backslash\{i\}}}\Ind_{V_{I\backslash\{j\}}}M) \\
				&= (\Res_{V_{I\backslash\{i\}}}\Ind_{V_{I\backslash\{j\}}}M)(A_i).										
			\end{align*}
			Since $\bL\!\Ind_{V_{I\backslash\{j\}}}$ always preserves compactness (e.g.\ by Lemma \ref{lem: nice functor preserving or reflecting compactness} as its right adjoint commutes with coproducts, or simply because it maps representables to representables) and $\Res_{V_{I\backslash\{i\}}}$ preserves compactness by assumption, it is clear that $\bL\phi$ preserves compactness.
			
			In general, we can write $\phi_{ji}$ as an iterated cone of bimodules of the above form.
			Indeed, for $i < j$, $A_i\in\Ascr_i$ and $A_j\in\Ascr_j$ we have (as graded module)
			\[
			\phi_{ji}(A_i, A_j) = 	t(\underline{\Ascr}_{J_{ij}})(A_i,A_j) = \bigoplus_{I\subset J_{ij}}\left( X_{\{i,\dots j\}\backslash I} \otimes_\kk \Ascr_I(V_{I\backslash\{i\}}A_i,V_{I\backslash\{j\}}A_j)\right).
			\]
			Define
			\[
			J_{ij}^k := \{ I\subset J_{ij}\mid |I|\geq k \},\quad \textrm{for } 2\leq k \leq j-i+1.
			\]
			This gives us a filtration
			\[
			\left\{ \{i,\dots,j\} \right\}=J^{j-i+1}_{ij}\subseteq J^{j-i}_{ij}\subseteq \dots \subseteq J^{2}_{ij}=J_{ij}.
			\]
			Let  $\underline{\Ascr}_{J^k_{ij}}$ be the subgraph of $\underline{\Ascr}_{J_{ij}}$ obtained by inserting zeroes whenever $I\in J_{ij}\backslash J^k_{ij}$, still viewed as a $(j-i-1)$-cube.
			
			Finally, the inclusion $\underline{\Ascr}_{J^{k+1}_{ij}}\hookrightarrow\underline{\Ascr}_{J^{k}_{ij}}$ induces a distinguished triangle of dg bimodules
			\[
			t(\underline{\Ascr}_{J^{k+1}_{ij}}) \to t(\underline{\Ascr}_{J^{k}_{ij}})\to \bigoplus_{\{I\in J_{ij} \mid |I| = k\}} {}_{V_{I\backslash\{j\}}}(\Ascr_I)_{V_{I\backslash\{i\}}}[j-i+1-k]  \to ,
			\]
			which shows the claim.
		\end{proof}
		
		Combining the previous lemma with Proposition \ref{prop: smoothness directed dg cat}, noting that smoothness is invariant under Morita equivalences \cite[Theorem 3.17]{LuntsSchnurer} and that the dg category of twisted complexes over a proper dg category remains proper; we immediately obtain the following.
		
		\begin{corollary}\label{cor: smoothness gluing}
			Let $\underline{\Ascr}_n$ be a hypercube of dg categories.
			Suppose the dg categories $\Ascr_i$ are smooth and restriction along $\Ascr_i\to\Ascr_I$ for $i\in I\subseteq [n]$, induced by composing the edges of the hypercube, preserves compactness.
			Then, the glued dg category $\Glue(\underline{\Ascr}_n^{\circ})$ is smooth.
			Moreover, if the $\Ascr_i$ are proper, then so is $\Glue(\underline{\Ascr}_n^{\circ})$.
		\end{corollary}
		
	\subsection{Two constructions}
		We end this section by giving two constructions of hypercubes in the category $\Csf:=Z^0\dgMod(\Ascr)$, with $\Ascr$ a dg category.
		These constructions work more generally for hypercubes in an arbitrary category $\Csf$, but we are mostly interested in how these constructions behave with respect to acyclic hypercubes.
		It will be convenient at times to name our $n$-cubes with binary labels, i.e.\ thinking of them as functors $\{0,1\}^n\to\Csf$ (where we view  $\{0,1\}^n$ as a category obtained from the poset ordered by lexicographical ordering) instead of functors $2^{[n-1]}\to\Csf$. 
		(Binary labels are convenient for the operations below, whilst for the $t$-construction power set labels are convenient)
		
		The following lemma is key to all we do in this section. 
		From a high-brow point of view it follows from noting that the `category' of categories is cartesian closed, i.e.\ we have the following equivalence of functor categories
		\[
		\Fun(\{0,1\},\Fun(\{0,1\}^{n-1},\Csf))\cong \Fun(\{0,1\}\times\{0,1\}^{n-1},\Csf). 
		\]
		This allows us to view $n$-cubes as morphisms of $(n-1)$-cubes.
		We will use this to define operations on hypercubes.
		Concretely, with $\Cube_n:= \Fun(\{0,1\}^{n},\Csf)$ the equivalence gives.
		
		\begin{lemma}\label{lem: n-cube can be viewed as morphism of (n-1)-cubes}
			We have 
			\[
			\Cube_n\cong\Mor(\Cube_{n-1}).
			\]
			More precisely, let $\underline{A}$ be an $n$-cube.
			We can view $\underline{A}$ as a morphism of $(n-1)$-cubes as follows.
			Define $(n-1)$-cubes $\underline{A}_0$ and $\underline{A}_1$ via
			\begin{align*}
				\underline{A}_0(i_0,\dots,i_{n-2})&:=\underline{A}(i_0,\dots,i_{n-2},0),\\ \underline{A}_1(i_0,\dots,i_{n-2})&:=\underline{A}(i_0,\dots,i_{n-2},1)\
			\end{align*}
			and a morphism $\alpha:\underline{A}_0\to \underline{A}_1$ via 
			\[
			\alpha_{i_0,\dots,i_{n-2}}=\underline{A}(i_0,\dots,i_{n-2},0\leq 1).
			\]
			Then, the above equivalence maps $\underline{A}$ to $\alpha$.
		\end{lemma}
		\begin{remark}
			Technically, we applied, in addition, another equivalence $$\Fun(\{0,1\}\times\{0,1\}^{n-1},\Csf) \cong \Fun(\{0,1\}^{n-1}\times \{0,1\},\Csf)$$ to get the specific indexing as in the lemma.
			
			Similarly, there are $n-1$ other ways we could have viewed $\underline{A}$ as a morphism, corresponding to a splitting $\{0,1\}^n\cong\{0,1\}\times \{0,1\}^{n-1}$.
		\end{remark}
		
		\begin{lemma}\label{lem: t(A)=t(t(alpha))}
			With notation as in the previous lemma
			\[
			t(\underline{A})= t(t(\underline{A}_0)\xrightarrow{t(\alpha)}t(\underline{A}_1))\quad\text{(up to possible signs that are not of importance)}.
			\]
			Consequently, a hypercube $\underline{A}$ is acyclic if and only if the corresponding morphism $\alpha$ induces a quasi-isomorphism $t(\alpha)$.
		\end{lemma}
		\begin{proof}
			In fact, with a specific choice for the signs in the complexes, we get an honest equality.
			(Different choices would then lead to compensating minus signs.) 
			Relabelling to power set labels we have $\underline{A}_0 = \{A_I\}_{I\subset[n-1]}$ and $\underline{A}_1 = \{A_{I\cup\{n-1\}}\}_{I\subset[n-1]}$.
			We can think of $t(t(\underline{A}_0)\xrightarrow{t(\alpha)}t(\underline{A}_1))$ as $X_{n-1}t(\underline{A}_0)\oplus t(\underline{A}_1)$ with differential equal to that of $t(\underline{A})$.
			When picking an ordering of the products of the $X_i$'s so that $n-1$ is the smallest, we have equality of the complexes.
			(If one does not want to pick an ordering, the isomorphism is given by applying $X_{n-1}\partial_{n-1}$ to the summands indexed by $I\subseteq[n-1]$ and the identity to the others.)
		\end{proof}
		
		The following lemma follows immediately from the previous lemma.
		
		\begin{lemma}\label{lem: seeing acyclic on faces}
			A hypercube having two opposing acyclic faces is itself acyclic.
			Consequently, a hypercube $\underline{A}$ such that all squares `in one direction' are acyclic, i.e.\ the square $\underline{A}(\bullet,\bullet, i_3,\dots,i_n)$ (or some permutation) is acyclic for all $\{i_3,\dots,i_n\}\subseteq \{0,1\}$, is an acyclic hypercube.
		\end{lemma}
		
		\begin{definition}
			Let $\underline{A}$ and $\underline{B}$ be $n$-cubes sharing a face, we define the \emph{stacking of $\underline{A}$ and $\underline{B}$} to be the $n$-cube obtained by composing the corresponding morphisms of $(n-1)$-cubes from Lemma \ref{lem: n-cube can be viewed as morphism of (n-1)-cubes} along their common face.
		\end{definition}
		\begin{example}
			This does exactly what you would expect.
			The stacking of 
			\[
			\begin{tikzcd}[sep=0.75em]
				& A && E \\
				B && F \\
				& C && G \\
				D && H
				\arrow[from=2-1, to=4-1]
				\arrow[from=1-2, to=3-2]
				\arrow[from=1-4, to=3-4]
				\arrow[from=3-4, to=4-3]
				\arrow[from=3-2, to=4-1]
				\arrow[from=1-4, to=2-3]
				\arrow[from=1-2, to=2-1]
				\arrow["{g_2}"'{pos=0.3}, from=3-4, to=3-2]
				\arrow["{g_1}"', from=1-4, to=1-2]
				\arrow["{g_4}"', from=4-3, to=4-1]
				\arrow["{g_3}"'{pos=0.3}, from=2-3, to=2-1, crossing over]
				\arrow[from=2-3, to=4-3, crossing over]
			\end{tikzcd}\quad\text{and}\quad
			\begin{tikzcd}[sep=0.75em]
				& E && I \\
				F && J \\
				& G && K \\
				H && L
				\arrow[from=2-1, to=4-1]
				\arrow[from=1-2, to=3-2]
				\arrow[from=1-4, to=3-4]
				\arrow[from=3-4, to=4-3]
				\arrow[from=3-2, to=4-1]
				\arrow[from=1-4, to=2-3]
				\arrow[from=1-2, to=2-1]
				\arrow["{f_2}"'{pos=0.3}, from=3-4, to=3-2]
				\arrow["{f_1}"', from=1-4, to=1-2]
				\arrow["{f_4}"', from=4-3, to=4-1]
				\arrow["{f_3}"'{pos=0.3}, from=2-3, to=2-1, crossing over]
				\arrow[from=2-3, to=4-3, crossing over]
			\end{tikzcd}
			\]
			is 
			\[
			\begin{tikzcd}[sep=1em]
				& A && I \\
				B && J \\
				& C && K\rlap{ .} \\
				D && L
				\arrow[from=2-1, to=4-1]
				\arrow[from=1-2, to=3-2]
				\arrow[from=1-4, to=3-4]
				\arrow[from=3-4, to=4-3]
				\arrow[from=3-2, to=4-1]
				\arrow[from=1-4, to=2-3]
				\arrow[from=1-2, to=2-1]
				\arrow["{g_2 f_2}"'{pos=0.2}, from=3-4, to=3-2]
				\arrow["{g_1 f_1}"', from=1-4, to=1-2]
				\arrow["{g_4 f_4}"', from=4-3, to=4-1]
				\arrow["{g_3 f_3}"'{pos=0.2}, from=2-3, to=2-1, crossing over]
				\arrow[from=2-3, to=4-3, crossing over]
			\end{tikzcd}
			\]
		\end{example}
		
		\begin{lemma}\label{lem: stacking acyclic is acyclic}
			The stacking of acyclic hypercubes is acyclic.
		\end{lemma}
		\begin{proof}
			This follows immediately from the definition and Lemma \ref{lem: t(A)=t(t(alpha))}.
		\end{proof}
		
		\begin{definition}
			Let $\alpha:\underline{A}_0\to \underline{A}_1$ and $\beta:\underline{B}_0\to \underline{B}_1$ be $n$-cubes sharing the face $\underline{A}_1=\underline{B}_0$, we define the \emph{extension of $\underline{A}$ by $\underline{B}$} (the order matters) to be the $(n+1)$-cube obtained as follows.
			Denote by $\underline{C}$ the $n$-cube corresponding to $\id:\underline{B}_1\to \underline{B}_1$.
			Then, $\beta\alpha$ and $\beta$ define a morphism $\gamma:\underline{A}\to \underline{C}$ by Lemma \ref{lem: n-cube can be viewed as morphism of (n-1)-cubes}:
			\[
			\begin{tikzcd}[sep=1em]
				& {\underline{A}_0} &&& {\underline{A}_1=\underline{B}_0} \\
				{} &&&& {} & {} \\
				{} &&&& {} & {} \\
				& {\underline{B}_1} &&& {\underline{B}_1}\rlap{ .}
				\arrow["\id", from=4-2, to=4-5]
				\arrow["\beta\alpha"', from=1-2, to=4-2]
				\arrow["\beta", from=1-5, to=4-5]
				\arrow["\alpha", from=1-2, to=1-5]
				\arrow["\gamma="', no head, dashed, from=2-1, to=3-1, start anchor= real center, end anchor=real center]
				\arrow[dashed, no head, from=3-1, to=3-5, start anchor= real center, end anchor={[xshift=3ex]real center}]
				\arrow[dashed, no head, from=3-5, to=2-5, start anchor= real center, end anchor=real center, shift right=5]
				\arrow[dashed, no head, from=2-1, to=2-5, start anchor= real center, end anchor={[xshift=3ex]real center}]
			\end{tikzcd}
			\]
			We define the extension to be the $(n+1)$-cube corresponding, again by Lemma \ref{lem: n-cube can be viewed as morphism of (n-1)-cubes}, to $\gamma$.
		\end{definition}
		\begin{example}
			It is perhaps less clear what this construction does. 
			The extension of 
			\[
			\begin{tikzcd}[sep=1.5em]
				B && D \\
				\\
				A && C
				\arrow[from=3-1, to=1-1]
				\arrow[from=1-1, to=1-3, "b"]
				\arrow[from=3-1, to=3-3, "a"]
				\arrow[from=3-3, to=1-3]
			\end{tikzcd}\quad\text{by}\quad
			\begin{tikzcd}[sep=1.5em]
				D && F \\
				\\
				C && E
				\arrow[from=3-1, to=1-1]
				\arrow[from=1-1, to=1-3, "d"]
				\arrow[from=3-1, to=3-3, "c"]
				\arrow[from=3-3, to=1-3]
			\end{tikzcd}
			\]
			is
			\[
			\begin{tikzcd}[sep= 0.75em]
				& B && D \\
				F && F \\
				& A && C\rlap{ .} \\
				E && E
				\arrow[from=3-2, to=1-2]
				\arrow["b", from=1-2, to=1-4]
				\arrow["a"{pos=0.3}, from=3-2, to=3-4]
				\arrow[from=3-4, to=1-4]
				\arrow[equal, from=2-1, to=2-3, crossing over]
				\arrow[equal, from=4-1, to=4-3]
				\arrow["d"', from=1-4, to=2-3]
				\arrow["c"', from=3-4, to=4-3]
				\arrow["ca"', from=3-2, to=4-1]
				\arrow["db"', from=1-2, to=2-1]
				\arrow[from=4-1, to=2-1]
				\arrow[from=4-3, to=2-3, crossing over]
			\end{tikzcd}
			\]
		\end{example}
		
		\begin{lemma}\label{lem: extension of acyclic is acyclic}
			The extension of an acyclic hypercube, by a (not necessarily acyclic) hypercube, is acyclic.
		\end{lemma}
		\begin{proof}
			With notation as in the definition we have that $\underline{A}$ and $\underline{C}$ are acyclic.
			Hence, the result follows immediately from Lemma \ref{lem: seeing acyclic on faces}.
		\end{proof}

\section{Filtered schemes}\label{sec: filt sch}
		In this section, we set up the general theory of (finite length) filtered schemes.
		As we will see, these are an alternative incarnation of the $\Acal$-spaces from \cite{KuznetsovLunts}.
		For this reason many of the results proved in this section for filtered schemes will directly parallel or follow in a quite straightforward manner from results in loc.\ cit.
		
	\subsection{The category of filtered schemes}\label{subsec: cat fSch}
		\subsubsection{Generalities}
		We start with some basic definitions.	
		\begin{definition}
			A \emph{filtered scheme} is a scheme $(X,\Ocal_X)$ equipped with an ascending filtration $(F^i\Ocal_X)_{i\in \mathbb Z}$ of quasi-coherent sheaves of ideals satisfying the following:\footnote{In the left-hand side of \ref{eq: in def filt scheme}, $F^i\mathcal O_XF^j\mathcal O_X$ is the image of the natural morphism $ F^i\mathcal O_X\otimes_{\Ocal_X}F^j\mathcal O_X\to \mathcal O_X$. More precisely this is the subsheaf of $\mathcal O_X$ whose sections can locally be written as sums of products of sections of $F^i\mathcal O_X$ and $F^j\mathcal O_X$.}
			\begin{enumerate}
				\item $F^0\Ocal_X=\Ocal_X$,
				\item\label{eq: in def filt scheme} $F^i\mathcal O_XF^j\mathcal O_X\subseteq F^{i+j}\mathcal O_X$ for all $i$, $j\in \mathbb Z$.
			\end{enumerate}
			If we want to make the filtration explicit in the notation we write $F^*:=F^*\mathcal O_X:=(F^i\mathcal O_X)_{i\in\mathbb Z}$ and denote the filtered scheme by $(X, F^*)$.
		\end{definition}
		We say that a filtered scheme \emph{has finite length $n$}, or is an \emph{$n$-filtered scheme}, if $F^{-n}\Ocal_X=0$.
		Explicitly, this means that the filtration has the following form 
		\[
		0=F^{-n}\Ocal_X\subseteq F^{-n+1}\Ocal_X\subseteq\dots\subseteq F^{-1}\Ocal_X\subseteq F^{0}\Ocal_X =\Ocal_X.
		\]
		The $n$ is part of the data, we will therefore sometimes denote an $n$-filtered scheme by $(X,{}_{n}F^*)$.
		Moreover, we often adopt a naming convention for the filtration of a filtered scheme of finite length that reflects its length.
		For example, $(X,F^*)$ and $(Y,F^*)$ are implicitly understood to have the same length, whilst $(Z,G^*)$ could have an a priori different length.
		Note that finite filtrations are solely a non-reduced phenomenon, every finite filtration on a reduced scheme is the trivial one.
		
		Usually, the filtration considered will be finite. 
		Therefore, whenever we say `filtered scheme' we really have `finite length filtered scheme' in mind.
		
		\begin{definition}
			A \emph{morphism of filtered schemes} $(X,F^*)\to(Y,G^*)$ consists of a morphism of schemes $(f,f^\sharp):X\to Y$ such that the morphism on the structure sheaves\footnote{The direct and inverse image of a filtered sheaf of rings have a natural induced filtration (as they are left exact).}
			\[
			f^\sharp : (\mathcal{O}_Y, G^*)\to (f_*\mathcal{O}_X, F^*) 
			\]
			or equivalently
			\[
			f^\sharp : (f^{-1}\mathcal{O}_Y, G^*)\to (\mathcal{O}_X, F^*) 
			\]
			is a morphism of filtered sheaves of rings.
			Explicitly, we require that
			\[
			f^\sharp(G^i\mathcal{O}_Y) \subseteq f_* F^i\mathcal{O}_X \Leftrightarrow f^\sharp(f^{-1} G^i\mathcal{O}_Y) \subseteq F^i\mathcal{O}_X
			\]
			for all $i\in \mathbb Z$.
		\end{definition}
		
		In this way, we obtain a category.
		We denote the \emph{category of filtered schemes} by $\fSch$ and the full subcategory of $n$-filtered schemes by $n$-$\fSch$.
		
		Let $P$ be a property of schemes or of morphisms of schemes.
		We say that a filtered scheme $(X, F^*)$ has property $P$ if the underlying scheme $X$ has property $P$.
		Similarly, a morphism of filtered schemes $(X,F^*)\to(Y,G^*)$ has property $P$ if the underlying morphism of schemes has property $P$.
		
		\subsubsection{Rees algebra}
		To define appropriate `modules' over filtered schemes, we make use of sheaves of graded modules over the Rees algebra associated to the filtered structure sheaf.
		Therefore, we start by motivating this choice and collecting some facts concerning these. 
		
		Let $(X,F^*)$ be a filtered scheme. 
		We can consider the category $\Filt(X,F^*)$ of sheaves of filtered $\Ocal_X$-modules, or simply \emph{filtered modules}.
		Its objects are $\Ocal_X$-modules $\Mcal$ equipped with an exhaustive\footnote{That is, $\Mcal=\cup_{i\in\ZZ}F^i\Mcal$, where the right-hand side includes a sheafification.} ascending filtration $(F^i\Mcal)_{i\in\ZZ}$ compatible with that of $(\Ocal_X,F^*)$ under the $\Ocal_X$-action on $\Mcal$.
		`Unfortunately', however, this category is not abelian, but merely \emph{quasi-abelian}\footnote{A quasi-abelian category is a pre-abelian category in which the collection of kernel-cokernel pairs forms an exact structure.}.
		As homological algebra is our bread and butter, and works easiest in abelian categories, we instead would like to consider an abelian category that best encapsulates $\Filt(X,F^*)$.
		This is done by considering graded modules over the associated Rees algebra.
		
		Define the \emph{Rees algebra} associated to $(\Ocal_X,F^*)$ to be
		\[
		\tOcal_{(X,F^*)}:= \Rees(\Ocal_X,F^*):=\bigoplus_{i\in\mathbb Z} F^{i}\Ocal_X t^i\subseteq \Ocal_X[t,t^{-1}],
		\]
		which obtains its ring structure and grading by viewing it as a subalgebra of $\mathcal{O}_X[t,t^{-1}]$.
		The indeterminate $t$ of degree one is helpful to keep track of the degree of elements, but we often omit it.
		This is a sheaf of graded $\Ocal_X$-algebras.
		As the filtration is usually clear from context, we often use the slightly abusive notation $\tOcal_X$ instead of $\tOcal_{(X,F^*)}$.
		
		Next, consider the category $\grMod(\tOcal_X)$ consisting of sheaves of graded modules over $\tOcal_X$, or simply \emph{graded modules}.
		Its objects are families\footnote{		
			This is the `correct' way of thinking of a sheaf of graded modules, as opposed to as a single sheaf of modules $\oplus_i \Mcal^i$ with a direct sum decomposition.
			Although, as the category has countable direct sums, both descriptions are equivalent.
			However, the benefit of the former, for example, is that it makes it clearer how to define the pushforward (as the pushforward of sheaves need not commute with direct sums, and so, will not commute with forgetting the grading).		
		} $( \Mcal^i )_{i\in\ZZ}$ of $\Ocal_X$-modules together with morphisms
		\[
		\Mcal^i\times F^j\Ocal_X\to \Mcal^{i+j},\quad i,j\in \ZZ,
		\]
		satisfying the usual associativity and unitality conditions.
		We will usually simply denote our graded modules as 
		\[
		\oplus_{i} \Mcal^{i},
		\]
		but one should really think of them as a collection of $\Ocal_X$-modules $( \Mcal^i )_{i\in\ZZ}$.
		Graded modules have a natural $\ZZ$-action given by shifting, this is defined by \[\Mcal(i)^j:=\Mcal^{i+j}.\]
		(We use round brackets to distinguish this from the shift when viewed as a complex concentrated in degree zero.)
		
		We have the following.
		\begin{proposition}\label{prop: Rees gives abelian hull}
			The category $\grMod(\tOcal_X)$ is the `(right) abelian hull' of $\Filt(X,F^*)$.
			This means that there exists a fully faithful functor 
			\begin{align*}
				\iota: \Filt(X,F^*)&\to \grMod(\tOcal_X),\\
				(\Mcal,F^*)&\mapsto \oplus_i F^i\Mcal
			\end{align*}
			that preserves and reflects exactness and is universal in some precise sense (see \cite[Proposition 1.2.34]{Schneiders}).
			Its essential image consists of the $t$-torsion free graded modules, where $t$ is the distinguished degree one element of $\tOcal_X$.
			
			Moreover, $\iota$ induces an equivalence $\bD(\Filt(X,F^*))\cong\bD(\grMod(\tOcal_X))$ of derived categories.
		\end{proposition}
		\begin{proof}[Sketch of proof]
			This follows, for example, from \cite[Propositions 1.2.32 and 1.2.36]{Schneiders} adapting \cite[Proposition 3.14]{SchapiraSchneiders} to show that any graded module is a quotient of a filtered module. 
			See also \cite[Appendix B]{BondalVandenBergh}.
		\end{proof}
		We henceforth identify $\Filt(X,F^*)$ with the full subcategory of $\grMod(\tOcal_X)$ consisting of $t$-torsion free modules.
		Moreover, as we are ultimately only interested in the derived category associated to a filtered scheme, we see that there is no harm in simply considering graded modules over the Rees algebra. 
		
		When $(X,F^*)$ has finite length $n$ our main interest will be in the full abelian subcategory $\grMod^n(\tOcal_X)$ of $\grMod(\tOcal_X)$ consisting of \emph{length $n$ graded modules} $\Mcal=\oplus_i\Mcal^i$ for which $\Mcal^i=0$ for $i\leq-n$ and multiplication by $t$ induces an isomorphism $\Mcal^i\isoto \Mcal^{i+1}$ for $i\geq 0$. 
		We also allow $n=\infty$ in which case we only require the latter condition.
		\begin{remark}\label{rem: Rees gives abelian hull}
			There are obvious analogous categories on the filtered side, $\Filt^n(X,F^*)$ consists of \emph{length $n$ filtered modules}, i.e.\ filtered modules $(\Mcal,F^*)$ with $F^0\Mcal=\Mcal$ and $F^{-n}\Mcal=0$.
			Proposition \ref{prop: Rees gives abelian hull} restricts nicely to this setting.
		\end{remark}
		
		We end this subsection with a result that will be important when defining pullback functors below.
		\begin{lemma}\label{lem: adjoints Gr^N in Gr^infty}
			Let $(X,F^*)$ be a filtered scheme (not necessarily of length $n$). 
			The inclusion 
			\[
			\grMod^n(\tOcal_X)\hookrightarrow \grMod^\infty(\tOcal_X)
			\]
			has a left adjoint $l^n$ which, on objects, is given by
			\[
			\Mcal\mapsto\Mcal/\Mcal_*,				
			\]
			where $\Mcal_*$ is the graded $\tOcal_X$-submodule of $\Mcal$ generated by $\oplus_{i\leq -n}\Mcal^i$. 
		\end{lemma}
		\begin{proof}
			More explicitly, we have
			\[
			\Mcal_*^i=  \begin{cases} 
				\Mcal^{-n}t^{n+i}       & \text{if } i > -n ,\\
				\Mcal^i & \text{if } i \leq -n ,
			\end{cases}	
			\]
			and
			\[
			l^n(\Mcal)^i=  \begin{cases} 
				\Mcal^i/(\Mcal^{-n}t^{n+i})       & \text{if } i > -n ,\\
				0 & \text{if } i \leq -n .
			\end{cases}	
			\]	
			As any morphism from $\Mcal$ to an object of $\grMod^n(\tOcal_X)$ sends $\Mcal_*$ to zero, and hence factors uniquely through $l^n(\Mcal)$, the adjunction follows.
		\end{proof}
		\begin{remark}
			The inclusion $\grMod^n(\tOcal_X)\hookrightarrow \grMod^\infty(\tOcal_X)$ also has a right adjoint $r^n$, the sections of $r^n(\Mcal)^i$ are those sections of $\Mcal^i$ that (locally) get annihilated by $\oplus_{j\leq-n-i} F^j\Ocal_X$.
			We will not use this adjoint, so we do not go into further detail.
		\end{remark}
		We refer to the left, respectively, right adjoint of the inclusion 
		\[
		\grMod^n(\tOcal_X)\hookrightarrow \grMod^\infty(\tOcal_X)
		\] 
		as the \emph{left, respectively, right $n$th truncation functor}.
		In the next lemma, we collect some relations that will be used below.
		They follow from the uniqueness of adjoints.
		\begin{lemma}\label{lem: rel ln's}
			The left truncation functors $l^n:\grMod^\infty(\tOcal_X) \to \grMod^n(\tOcal_X)$  satisfy the following:
			\begin{enumerate}
				\item if $m\geq n$, then $l^nl^m=l^n$,
				\item for $i\geq 0$, $l^n(-)(i)=l^{n+i}(-(i))$.
			\end{enumerate}
		\end{lemma}
		
		\subsubsection{Module categories}
		To any filtered scheme $(X,F^*)$ we associate an abelian \emph{category of (quasi-coherent) sheaves of modules}, or simply \emph{(quasi-coherent) modules}, by considering (certain) graded modules over the Rees algebra associated to the filtered structure sheaf.
		Define
		\begin{align*}
			\Mod(X,F^*)&:=\grMod(\tOcal_X),\\
			\QCoh(X,F^*)&:=\grMod_{\QCoh}(\tOcal_X),
		\end{align*}
		where the latter notation means graded $\tOcal_X$-modules that are quasi-coherent as underlying $\Ocal_X$-modules (which is equivalent to every graded component being a quasi-coherent $\Ocal_X$-module)\footnote{
			As $F^0\Ocal_X=\Ocal_X$, we have a morphism of graded rings $\Ocal_X\to \tOcal_X$, where we view $\mathcal{O}_X$ as concentrated in degree zero. 
			Thus any graded $\tOcal_X$-module is a graded $\Ocal_X$-module (and a direct sum of modules is quasi-coherent if and only if every summand is quasi-coherent).}. 
		
		When $(X,F^*)$ has finite length $n$ we consider moreover the following subcategories of \emph{length $n$ (quasi-coherent) modules}
		\begin{align*}
			\Mod^n(X,F^*)&:=\grMod^n(\tOcal_X),\\
			\QCoh^n(X,F^*)&:=\grMod_{\QCoh}^n(\tOcal_X).
		\end{align*}
		Furthermore, \emph{length $n$ coherent modules} also make sense in this setting. 
		We define
		\[
		\Coh^n(X,F^*):=\grMod_{\Coh}^n(\tOcal_X)
		\]
		as the length $n$ graded $\tOcal_X$-modules $\Mcal=\oplus_{i} \Mcal^{i}$ such that $\oplus_{i\leq 0} \Mcal^{i}$ is coherent as $\Ocal_X$-module (which is equivalent to $\Mcal^{-n+1}, \dots,\ \Mcal^{-1}$ and $\Mcal^{0}$ being coherent).
		
		There are other ways of characterising these modules.
		\begin{lemma}\label{lem: char (Q)Coh}
			Let $\Mcal$ be an object of $\Mod(X,F^*)$.
			The following are equivalent:\:\!\footnote{This uses the fact that the filtration consists of quasi-coherent ideals, hence that $\tOcal_X$ is a quasi-coherent $\Ocal_X$-module.}
			\begin{enumerate}
				\item\label{item: char Qcoh1} $\Mcal\in\QCoh(X,F^*)$,
				\item\label{item: char Qcoh2} 
				for every point $x\in X$ there exists an open neighbourhood $x\in U\subseteq X$ such that $\Mcal|_U$ is isomorphic to the cokernel of a map					\[
				\oplus_{i\in I} \tOcal_U(k_i)\to \oplus_{j\in J} \tOcal_U(l_j),
				\]
				for some sets $I$ and $J$ and integers $k_i$ and $l_j$.
			\end{enumerate}
			Moreover, we can change both \ref{item: char Qcoh1} and \ref{item: char Qcoh2} above by
			\begin{itemize}
				\item \ref{item: char Qcoh1} $\leftrightarrow$ $\Mcal\in\QCoh^\infty(X,F^*)$, 
				\\ \ref{item: char Qcoh2} $\leftrightarrow$ the integers $k_i,l_i$ are positive,
			\end{itemize}		
			
			Furthermore, suppose that $(X,F^*)$ is of finite length $n$.
			Then, we can change \ref{item: char Qcoh1} and \ref{item: char Qcoh2} above by
			\begin{itemize}
				\item \ref{item: char Qcoh1} $\leftrightarrow$ $\Mcal\in\QCoh^n(X,F^*)$, \\ \ref{item: char Qcoh2} $\leftrightarrow$ $\tOcal_U(k_i)$ and $\tOcal_U(l_j)$ replaced by $l^n(\tOcal_U(k_i))$ and $l^n(\tOcal_U(l_j))$, and the integers $k_i,l_i\in \{0,\dots,n-1\}$,
			\end{itemize}
			and if, in addition, every $F^i\Ocal_X$ is coherent as $\Ocal_X$-module\:\!\footnote{
				Meaning that $\tOcal_X\in\Coh^n(X,F^*)$ (which is clearly necessary for the lemma to hold).
				This holds automatically when $(X,F^*)$ is locally Noetherian as coherent modules are better behaved then. 
				For example, coherent = finite presentation = quasi-coherent finite type, and quasi-coherent submodules/quotients of coherent modules are again coherent.
				In general the structure sheaf need not be coherent, but when it is coherent = finite presentation. 
				This is essentially the content of this part of the lemma. }
			by
			\begin{itemize}
				\item \ref{item: char Qcoh1} $\leftrightarrow$ $\Mcal\in\Coh^n(X,F^*)$, \\ \ref{item: char Qcoh2} $\leftrightarrow$ $\tOcal_U(k_i)$ and $\tOcal_U(l_j)$ replaced by $l^n(\tOcal_U(k_i))$ and $l^n(\tOcal_U(l_j))$, the sets $I$ and $J$ taken finite and the integers $k_i,l_i\in \{0,\dots,n-1\}$.
			\end{itemize}
		\end{lemma}
		\begin{proof}
			One way to prove the non-obvious direction is noting that affine locally on $\Spec(R)$ a (quasi-)coherent graded module $\Mcal$ corresponds to an $R$-module $M$ which is naturally a graded $\Rees(R,F^*)$-module. 
			The claim follows from this.
			
			An alternative way, is noting that $\QCoh$, $\QCoh^\infty$, $\QCoh^n$ and $\Coh^n$ are closed under the appropriate operations, kernels and (possibly infinite) direct sums, and contain the $\tOcal_X(i)$'s or $l^n(\tOcal_X(i))$'s.
			Hence, it suffices to locally construct a surjection from a (truncated) free graded module.
			This can be done by constructing a surjection from a free $\Ocal_X$-module in every graded component, tensoring with $\tOcal_X$ over $\Ocal_X$, postcomposing with multiplication and then truncating.
		\end{proof}
		We sometimes refer to $\Mod^n$, $\QCoh^n$ and $\Coh^n$ as \emph{truncated module categories}.
		When working with filtered schemes of finite length, these are the categories that are of importance.
		Furthermore, the natural embedding of Proposition \ref{prop: Rees gives abelian hull} clearly restricts well to (the truncated) (quasi-)coherent filtered $(\Ocal_X,F^*)$-modules.
		So, we likewise identify these types of filtered modules with their essential image. 
		
		\subsubsection{Generalised morphisms, pullback and pushforward}\label{subsubsec: gen mor}
		In general, the usual pullback and pushforward of schemes, suitably reinterpreted for the graded context, will not preserve the truncated module categories.
		Therefore, in order to obtain a pullback/pushforward-like adjunction between them, we have to post-compose with the left/right truncation functors.
		Unfortunately, doing so can break the functoriality of the composition of the pullback and pushforward functors when the filtration length decreases from source to target\footnote{
			When the lengths increase from source to target, everything works out fine. For morphisms $(X,F^*)\xrightarrow{f}(Y,G^*)\xrightarrow{g}(Z,H^*)$ with $n_X\leq n_Y\leq n_Z$ we have a natural isomorphism and equalities $l^{n_X}(g\circ f)^*\cong l^{n_X}f^*\circ l^{n_Y}g^*$, $r^{n_Y}f_*=f_*$, $r^{n_Z}g_*=g_*$ and $r^{n_Z}(g\circ f)_*=(g\circ f)_*$, when interpreted as functors on the `correct' domains. (The first isomorphism can be seen using the latter equalities and uniqueness of adjoints.) However, as morphisms of this form do not naturally appear in this work, we do not consider them.
		}, as the truncation functors do not always compose well themselves.
		To overcome this we define a slightly more general type of morphism between filtered schemes which allows us to decrease the length.
		For this, we include a Veronese into the definition, where the \emph{$d$-Veronese} of a graded object $M=(M^i)_{i\in\ZZ}$ is the graded object 
		\[
		M^{(d)}:=(M^{di})_{i\in\ZZ}
		\]
		obtained by keeping only the multiples of $d$ in the grading.
		
		\begin{definition}
			A \emph{generalised morphism, or $d$-morphism,} from a $dn$-filtered scheme to an $n$-filtered scheme $(f,d):(X,{}_{dn}F^*)\rightsquigarrow(Y,{}_{n}G^*)$ consists of a morphism of schemes $(f,f^\sharp):X\to Y$ such that
			\[
			f^\sharp : (\mathcal{O}_Y, G^*)\to (f_*\mathcal{O}_X, F^*)^{(d)},
			\]
			or equivalently
			\[
			f^\sharp : (f^{-1}\mathcal{O}_Y, G^*)\to (\mathcal{O}_X, F^*)^{(d)},
			\]
			is a morphism of filtered sheaves of rings.
			Explicitly, we require that
			\begin{equation}\label{eq: generalised morphism}
				f^\sharp(G^i\mathcal{O}_Y) \subseteq f_* F^{di}\mathcal{O}_X \Leftrightarrow f^\sharp(f^{-1} F^i\mathcal{O}_Y) \subseteq F^{di}\mathcal{O}_X
			\end{equation}
			for all $i\in \mathbb Z$.
		\end{definition}
		
		Note that generalised morphisms compose as they should, i.e.\ $(g,e)\circ(f,d) = (g\circ f,de)$.
		We sometimes leave the `$(\ ,d)$ part' out of the notation and denote a $d$-morphism, for $d>1$, by a squiggly arrow $\rightsquigarrow$.
		Two special cases of $d$-morphisms are worth highlighting:
		\begin{itemize}
			\item when $d=1$, we obtain the usual notion of a morphism between filtered schemes of the same length,
			\item when $f=\id$ and we have equality instead of inclusion in Equation \eqref{eq: generalised morphism}, we obtain a \emph{$d$-refinement}.
			This is a one-sided inverse to taking the $d$-Veronese, i.e.\ a $d$-refinement of an $n$-filtered scheme $(X,F^*)$ is a $dn$-filtered scheme $(X,G^*)$ such that $G^{di}\Ocal_X=F^i\Ocal_X$ for all $i\in\ZZ$.
			Refinements are certainly not unique. 
			We always view a $d$-refinement of an $n$-filtered scheme as a $dn$-filtered scheme.		
		\end{itemize}
		In fact, any generalised morphism can be written as the composition of an ordinary morphism and a refinement. 
		Indeed, let $(f,d):(X,{}_{dn}F^*)\rightsquigarrow(Y,{}_{n}G^*)$ be a $d$-morphism. 
		Define a $dn$-filtered scheme $(Y,G'^{*})$ via\footnote{This is, in some sense,  the left adjoint to taking the $d$-Veronese, see also the proof of Lemma \ref{lem: pull/push}.}
		\[
		G'^{i}\Ocal_Y:=G^{\lfloor i/d\rfloor}\Ocal_Y,
		\]
		then $(f,d)$ decomposes as 
		\[
		\begin{tikzcd}
			(X,{}_{dn}F^*)\arrow[r, "{(f,1)}"] & (Y,{}_{dn}G'^{^*})\arrow[r, "{(\id,d)}", squiggly] & (Y,{}_{n}G^*),
		\end{tikzcd}
		\]
		where the first morphism is an ordinary morphism and the second is a refinement.
		As refinements give well-behaved functors at the level of the truncated morphism categories we obtain the following.
		
		\begin{lemma}\label{lem: pull/push}
			Any generalised morphism of filtered schemes $$(f,d):(X,{}_{dn}F^*)\rightsquigarrow(Y,{}_{n}G^*)$$ induces a pullback/pushforward-like adjunction
			\[
			\begin{tikzcd}[sep=2.5em]
				{\Mod^{dn}(X,F^*)} \\ {\Mod^{n}(Y,G^*)}\rlap{ ,}
				\arrow[""{name=0, anchor=center, inner sep=0}, "{(f,d)_{*}}", bend left=45, from=1-1, to=2-1]
				\arrow[""{name=1, anchor=center, inner sep=0}, "{(f,d)^{*}}", bend left=45, from=2-1, to=1-1]
				\arrow["\dashv"{anchor=center}, draw=none, from=1, to=0]
			\end{tikzcd}
			\]
			which, when $f$ is quasi-compact and quasi-separated, restricts to quasi-coherent modules and, when in addition $f$ is proper and the filtered schemes are locally Noetherian, restricts to coherent modules.
			
			Moreover, these compose well: if $n=em$ and $(g,e):(Y,{}_{em}G^*)\rightsquigarrow(Z,{}_{m}H^*)$ is an $e$-morphism we have
			$(g,e)_{*}\circ(f,d)_{*} = (g\circ f,de)_{*}$ (honest equality) and consequently canonically $(f,d)^{*}\circ(g,e)^{*} \cong (g\circ f,de)^{*}$.
		\end{lemma}
		\begin{proof}
			As above define a $dn$-filtered scheme $(Y,G'^{*})$ via	$G'^{i}\Ocal_Y:=G^{\lfloor i/d\rfloor}\Ocal_Y$.
			Taking the $d$-Veronese induces a functor 
			\begin{align*}
				(-)^{d}:\Mod^{dn}(Y,G'^*)&\to\Mod^n(Y,G^*), \\
				\oplus_{i}\Mcal^{i}&\mapsto \oplus_{i}\Mcal^{di}
			\end{align*}
			admitting a left adjoint
			\begin{align*}
				\epsilon:\Mod^n(Y,G^*)&\to\Mod^{dn}(Y,G'^*),	\\
				\oplus_{i}\Mcal^{i}&\mapsto \oplus_{i}\Mcal^{\lfloor i/d \rfloor}\rlap{ .}		
			\end{align*}
			We define $(f,d)^{*}$ and $(f,d)_{*}$ through the following diagram:
			\begin{equation}\label{eq: def pullpush gen morph}
				\begin{tikzcd}
					{\Mod^{dn}(X,F^*)} & {\Mod^{dn}(Y,G'^*)} & {\Mod^{n}(Y,G^*)}\rlap{ .}
					\arrow[""{name=0, anchor=center, inner sep=0}, "{f_{*}}"{description}, bend left=20, from=1-1, to=1-2]
					\arrow[""{name=1, anchor=center, inner sep=0}, "{(-)^{(d)}}"{description}, bend left=20, from=1-2, to=1-3]
					\arrow[""{name=2, anchor=center, inner sep=0}, "\epsilon"{description}, bend left=20, from=1-3, to=1-2]
					\arrow[""{name=3, anchor=center, inner sep=0}, "{l^{dn}f^{*}}"{description}, bend left=20, from=1-2, to=1-1]
					\arrow["{(f,d)_{*}}", bend left=35, from=1-1, to=1-3]
					\arrow["{(f,d)^{*}}", bend left=35, from=1-3, to=1-1]
					\arrow["\dashv"{anchor=center, rotate=90}, draw=none, from=2, to=1]
					\arrow["\dashv"{anchor=center, rotate=90}, draw=none, from=3, to=0]
				\end{tikzcd}
			\end{equation}
			Here $f_*$ applies the usual pushforward component-wise: 
			\[
			f_*\left(\oplus_{i}\Mcal^{i}\right):=\oplus_{i}f_*\Mcal^{i},
			\] 
			whilst $f^*$ is defined as one expects, using the graded tensor product:
			\[
			f^*(\oplus_{i}\Mcal^{i}):=f^{-1}\left(\oplus_{i}\Mcal^{i}\right)\otimes_{f^{-1}\left(\tOcal_{(Y,G'^*)}\right)}\tOcal_{(X,F^*)}.
			\]
			The latter will in general have a longer length than $dn$ which is why there is a truncation in diagram \eqref{eq: def pullpush gen morph}. 
			We may apply this truncation as $f^*$ maps $\Mod^\infty$ into $\Mod^\infty$, this can be seen using the description of Lemma \ref{lem: char (Q)Coh}.
			
			The claims concerning the restrictions to (quasi-)coherent modules follow immediately from the corresponding statements for usual schemes.
			The composition claims follow from the \emph{equalities}
			\[
			(g,e)_{*}\circ(f,d)_{*}=(-)^{e}g_{*}(-)^{d}f_{*}=(-)^{e}(-)^{d}g_{*}f_{*}=(-)^{ed}(g\circ f)_{*}=(g\circ f,de)_{*}.\qedhere
			\]
		\end{proof}
		
		\subsubsection{Refinements}
		The operation of taking $d$-refinements can be used to, in some sense, remove the non-reducedness of a filtered scheme\footnote{
			In \cite{KuznetsovLunts}, they consider $\Acal$-spaces to do exactly this (we briefly recall $\Acal$-spaces in the next subsection).
			They associate to a non-reduced scheme $X$ an $\Acal$-space $(X,\Acal_X)$ to remove the non-reducedness (if $X_{\operatorname{red}}$ is smooth, $(X,\Acal_X)$ is smooth). 
			This is in fact a special instance of taking a refinement.
		}.
		The following shows that in many cases refinements exist and can be made functorial for morphisms between filtered schemes of the same length.
		
		\begin{proposition}\label{prop: existence refinements}
			Let $(f,f^\sharp):(X,F^*)\to (Y,F^*)$ be a morphism of $n$-filtered schemes and suppose we are given ideal sheaves $\Ical_X\subseteq\Ocal_X$ and $\Ical_Y\subseteq\Ocal_Y$ such that
			\[
			\Ical_?^d\subseteq F^{-1}\Ocal_?\quad\text{and}\quad f^\sharp(\Ical_Y)\subseteq f_*\Ical_X.
			\]
			Then, there exist compatible $d$-refinements, i.e.\ a commutative square
			\[
			\begin{tikzcd}[row sep=1.5em, column sep=2.5em]
				(X,F^*)\arrow["f"', d] & (X,G^*)\arrow["{(\id,d)}"', l,squiggly]\arrow[d] \\
				(Y,F^*) & (Y,G^*)\rlap{ ,}\arrow["{(\id,d)}"', l,squiggly]
			\end{tikzcd}
			\]
			where $G^*$ is a $d$-refinement of $F^*$.
			Moreover, if $F^{-1}\Ocal_? \subseteq \Ical_?$ one has $G^{-1}\Ocal_?=\Ical_?$.
		\end{proposition}
		\begin{proof}
			Defining
			\[
			G^{i}\Ocal_? :=	\begin{cases}
				F^{0}\Ocal_? & \text{for }i\geq 0 \\
				F^{j}\Ocal_?\Ical_?^r + F^{j-1}\Ocal_? & \text{for }i<0\text{ and }i=jd-r\text{ with }j\leq0\leq r<d
			\end{cases}
			\]
			does the trick.
		\end{proof}
		
		Recall that $n$-$\fSch$ denotes the category of $n$-filtered schemes with ordinary (i.e.\ not generalised) morphisms between them. 
		\begin{corollary}\label{cor: refinements}
			We obtain a functor 
			\begin{align*}
				d\text{-}\mathop\mathrm{ref}:n\text{-}\fSch/(Y,F^*)&\to dn\text{-}\fSch \\
				(f:(X,F^*)\to(Y,F^*)) &\mapsto \text{refinement using }f^{-1}\Ical_Y\cdot \Ocal_X
			\end{align*}
			together with a natural transformation $d\text{-}\mathop\mathrm{ref}\to \mathop\mathrm{id}$ consisting of the generalised morphisms $(\id,d)$. 
		\end{corollary}
		
	\subsection{Relation to schemes and \texorpdfstring{$\Acal$}{A}-spaces}\label{subsec: rel schemes and A-spaces}
		Succinctly, we have
		\[
		\{\text{schemes}\}\subset \{\text{$\Acal$-spaces of \cite{KuznetsovLunts}}\}\subset \{\text{filtered schemes}\}.
		\]
		In this section, we briefly elaborate on this.
		
		\subsubsection{Schemes}
		A $1$-filtered scheme $(X,F^*)$ is simply a scheme $X$ endowed with the trivial filtration
		\[
		F^{\geq 0}\Ocal_X=\Ocal_X\quad\text{and}\quad F^{<0}\Ocal_X=0.
		\]
		It follows immediately that the category of $1$-filtered schemes $1$-$\fSch$ is simply the category of schemes $\Sch$.
		In addition, the module categories are compatible
		\begin{align*}
			\Mod^1(X,F)&\cong \Mod(X), \\
			\QCoh^1(X,F)&\cong \QCoh(X), \\
			\Coh^1(X,F)&\cong \Coh(X),
		\end{align*}
		where the right-hand sides are the `usual ones', i.e.\ respectively the category of sheaves of modules, quasi-coherent sheaves of modules and coherent sheaves of modules over the scheme $X$.
		Moreover, it is clear from their definition in Lemma \ref{lem: pull/push} that the pullback/pushforward functor obtained from a morphism of $1$-filtered schemes is simply the ordinary pullback/pushforward functor of schemes in this case.
		
		For future reference, we note the following.
		\begin{lemma}\label{lem: adjunction with gr-i}
			Let $(X,F^*)$ be an $n$-filtered scheme.
			For any $0\leq i<n$ we have an adjunction
			\[
			\begin{tikzcd}[sep=2.5em]
				{\Mod^{n}(X,F^*)} \\ {\Mod(X)}\rlap{ .}
				\arrow[""{name=0, anchor=center, inner sep=0}, "{\Ncal = \oplus_j \Ncal^j\mapsto \Ncal^{-i}=:\gr_{-i}\Ncal}", bend left=45, from=1-1, to=2-1]
				\arrow[""{name=1, anchor=center, inner sep=0}, "{\Mcal\mapsto \Mcal\otimes_{\Ocal_X}l^n(\tOcal_X(i))}", bend left=45, from=2-1, to=1-1]
				\arrow["\dashv"{anchor=center}, draw=none, from=1, to=0]
			\end{tikzcd}
			\]
		\end{lemma}
		\begin{proof}
			By viewing $l^n(\tOcal_X(i))$ as an $( \Ocal_X, \tOcal_X )$-bimodule, this essentially boils down to some form of the tensor-hom adjunction
			\begin{align*}
				&\Mod^{n}(X,F^*)(\Mcal \otimes_{\Ocal_X} l^n(\tOcal_X(i)), \Ncal)  \\ 
				\cong& \Mod(X)(\Mcal, \sHom_{\tOcal_X}(l^n(\tOcal_X(i)), \Ncal)) \\
				\cong& \Mod(X)(\Mcal, \sHom_{\tOcal_X}(\tOcal_X(i), \Ncal)) \\
				\cong& \Mod(X)(\Mcal, \Ncal^{-i}). \qedhere
			\end{align*}					
		\end{proof}
		\begin{remark}
			The case $i=0$ corresponds to `forgetting the filtration'. 
			This is the pullback/ pushforward adjunction along the generalised morphism $(\id,n):(X,F^*)\to X$.
			Indeed, $(\id,n)^*(\Mcal)=\Mcal\otimes_{\Ocal_X}\tOcal_X$ and $(\id,n)_*(\oplus_j \Ncal^j)=\Ncal^0$. 
		\end{remark}
		\begin{remark}\label{rem: adjunction with oplus gr-i}
			The functor $\oplus_{i=0}^{n-1}\gr_{-i}:\Mod^n(X,F^*)\to\Mod(X)$ is (essentially) restriction of scalars along the inclusion of the zeroth graded piece $\Ocal_X\hookrightarrow\tOcal_X$ (together with forgetting the positive degree terms, but there is no extra information in there anyway).
			It follows that this has a left adjoint given by $\Mcal\mapsto \Mcal\otimes_{\Ocal_X} \oplus_{i=0}^{n-1}l^n(\tOcal_X(i))$.
		\end{remark}
		
		\subsubsection{\texorpdfstring{$\Acal$}{A}-spaces}
		Let $X$ be a scheme and $\Acal$ a coherent sheaf of $\Ocal_X$-algebras.
		We denote by $\Mod(X,\Acal)$ the category of right $\Acal$-modules, and by $\QCoh(X,\Acal)$ (respectively, $\Coh(X,\Acal)$) the subcategories consisting of those right $\Acal$-modules that are quasi-coherent (respectively, coherent); as $\Acal$ is coherent as $\Ocal_X$-module this is equivalent to the underlying $\Ocal_X$-module being quasi-coherent (respectively, coherent).
		In \cite{KuznetsovLunts} the authors consider specific types of sheaves of coherent $\Ocal_X$-algebras, called \emph{Auslander algebras} (see also Remark \ref{rem: Auslander} below) after \cite[Corollary on page 551]{Auslander} where similar algebras, over Artin algebras, first appeared.
		A pair of a scheme together with an Auslander algebra over it is called an \emph{$\Acal$-space} in \cite{KuznetsovLunts}.
		As the following proposition shows, these $\Acal$-spaces are exactly the finite length filtered schemes equipped with the $\Ical$-adic filtration for some nilpotent ideal $\Ical\subset\Ocal_X$.
		
		\begin{proposition}\label{prop: Auslander}
			Let $(X,F^*)$ be an $n$-filtered scheme and put 
			\[
			\Pcal:=\oplus_{i=0}^{n-1}l^n(\tOcal_X(i))\text{ and }\Acal_{F^*}:=\sEnd_{\tOcal_X}(\Pcal).
			\]
			Then, $E(-):=\sHom_{\tOcal_X}(\Pcal,-)$ induces equivalences
			\begin{align*}
				\Mod^n(X,F)&\cong \Mod(X,\Acal_{F^*}), \\
				\QCoh^n(X,F)&\cong \QCoh(X,\Acal_{F^*}), \\
				\Coh^n(X,F)&\cong \Coh(X,\Acal_{F^*}).
			\end{align*}
		\end{proposition}
		\begin{remark}\label{rem: Auslander}
			To make the link with $\Acal$-spaces complete, note that $\Acal_{F^*}$ can be interpreted as the following matrix
			\[	
			\begin{pmatrix}
				\Ocal_X & F^{-1}\Ocal_X & F^{-2}\Ocal_X & \dots & F^{1-n}\Ocal_X \\
				\Ocal_X/F^{1-n}\Ocal_X & \Ocal_X/F^{1-n}\Ocal_X & F^{-1}\Ocal_X/F^{1-n}\Ocal_X & \dots & 	F^{2-n}\Ocal_X/F^{1-n}\Ocal_X \\
				\Ocal_X/F^{2-n}\Ocal_X & \Ocal_X/F^{2-n}\Ocal_X & \Ocal_X/F^{2-n}\Ocal_X & \dots & 	F^{3-n}\Ocal_X/F^{2-n}\Ocal_X \\
				\vdots & \vdots & \vdots & \ddots & \vdots \\
				\Ocal_X/F^{-1}\Ocal_X & \Ocal_X/F^{-1}\Ocal_X & \Ocal_X/F^{-1}\Ocal_X & \dots & 	\Ocal_X/F^{-1}\Ocal_X 
			\end{pmatrix}.
			\]
			We will refer to (sheaves of) algebras of this form as Auslander algebras.
			In \cite{KuznetsovLunts} they considered algebras of this form where $F^*$ is the $\Ical$-adic filtration for some nilpotent ideal $\Ical\subset\Ocal_X$.
		\end{remark}
		\begin{proof}
			The functor $E$ has a left adjoint $(-\otimes_{\Acal_{F^*}} \Pcal)$.
			Therefore, checking that $E$ is an equivalence can be done locally at the level of stalks (i.e.\ checking that the (co)unit is an isomorphism of sheaves can be done stalk-wise).
			The claim then follows by Morita theory as, in the affine case, $\Pcal$ is a progenerator\footnote{That is, a small projective generator. An object $X$ in an abelian category with coproducts is called a progenerator if the covariant hom functor $\Hom(X,-)$ commutes with all colimits and is faithful.} for $\Mod^n(X,F)$. (This follows as $\sHom_{\tOcal_X}(l^n(\tOcal_X(i)),\Mcal)\cong \Mcal^{-i}$.)
			
			As (quasi-)coherence is determined at the underlying $\Ocal_X$-module level, $E$ restricts suitably since it does not change the underlying module, it only `switches' the $\tOcal_X$- and $\Acal$-module structures.
			A module $\Mcal=\oplus_{i}\Mcal^{i}$ in $\Mod^n(X,F)$ gets mapped to the row $(	\Mcal^{0}\  \Mcal^{-1}\  \dots\   \Mcal^{-n+1} )$ where the $\Acal$-action is simply matrix multiplication from the right, induced by the $\tOcal_X$-action of $\Mcal$. 
		\end{proof}
		
		This equivalence is compatible with pullback/pushforward functors.
		Let $f:(X,F^*)\to(Y,F^*)$ be a morphism of $n$-filtered schemes (we consider only ordinary morphisms since refinements/Veroneses, and hence generalised morphisms, do not have a natural description on the $\Acal$-space side), and consider
		\begin{align*}
			\Acal_X:=\sEnd_{\tOcal_X}(\oplus_{i=0}^{n-1}l^n(\tOcal_X(i)))\quad\text{and}\quad 					\Acal_Y:=\sEnd_{\tOcal_Y}(\oplus_{i=0}^{n-1}l^n(\tOcal_Y(i))).
		\end{align*}
		Then, $f$ induces a morphism $(X,\Acal_X)\to (Y,\Acal_Y)$ of the corresponding $\Acal$-spaces, which we also denote by $f$, and the following diagrams (2-)commute:
		\begin{equation}\label{eq: pull/push compat Auslander}
			\begin{tikzcd}[sep=1.5em]
				\Mod^n(X,F^{*}) & \Mod(X,\Acal_X) \\
				\Mod^n(Y,F^{*}) & \Mod(Y,\Acal_Y)
				\arrow["\sim", from=1-1, to=1-2]
				\arrow["\sim", from=2-1, to=2-2]
				\arrow["{f^{*}}", from=2-1, to=1-1]
				\arrow["{f^{*}}"', from=2-2, to=1-2]
			\end{tikzcd}\quad\text{and}\quad
			\begin{tikzcd}[sep=1.5em]
				\Mod^n(X,F^{*}) & \Mod(X,\Acal_X) \\
				\Mod^n(Y,F^{*}) & \Mod(Y,\Acal_Y)\rlap{ .}
				\arrow["\sim", from=1-1, to=1-2]
				\arrow["\sim", from=2-1, to=2-2]
				\arrow["{f_{*}}"', from=1-1, to=2-1]
				\arrow["{f_{*}}", from=1-2, to=2-2]
			\end{tikzcd}
		\end{equation}
		Of course, these also suitably restrict to $\QCoh$ and $\Coh$ in the appropriate settings.
		For example, the commutativity of the left diagram 
		follows essentially from the fact that both legs of the diagram send $l^n(\tOcal_X(i))$ to the $i$th row of $\Acal_Y$ (we start counting rows and columns of matrices at zero) or 
		can be seen using (a slight generalisation of) Lemma \ref{lem: E(-otimes-) = E(-)otimes tE(-)} below.
		The commutativity of the right diagram then follows by uniqueness of adjoints.
		
		In \S\ref{subsubsec: derived compat}, we observe that these compatibilities descend to the derived level.
		
	\subsection{Derived category}\label{subsec: dercat}
		We define \emph{the derived category of an $n$-filtered scheme $(X,F^*)$} as
		\[
		\bD(X,F^*):=\bD(\QCoh^n(X,F^*)).
		\]
		Here, the length of the filtered scheme is left out the notation, if we want to make the length explicit, we write $\bD(X,{}_nF^*)$.
		\begin{remark}\label{rem: D_{Qc}=D(QCoh)}
			For $(X,F^*)$ quasi-compact and separated, the natural functor $$\bD(\QCoh^n(X,F^*))\to \bD_{\QCoh}(\Mod^n(X,F^*))$$ is an equivalence. 
			The same proof as \cite[Proposition 1.3]{AlonsoJeremiasLipman} works. 
			See also \cite[Theorem 2.10, pg.\ 222]{Boreletal} for the bounded case. 
			Moreover, in fact, one can actually deduce the filtered case from the non-filtered case without needing to reprove anything (by using compatibilities with the forgetful functor to the underlying $\Ocal_X$-module structure).
			
			Calling $\bD(\QCoh^n(X,F^*))$ the derived category of $(X,F^*)$ is therefore only really sensible when $(X,F^*)$ is quasi-compact and separated. 
			For more general filtered schemes one should look at $\bD_{\QCoh}(\Mod^n(X,F^*))$, the full subcategory of the derived category of modules consisting of complexes with quasi-coherent cohomology, instead of $\bD(\QCoh^n(X,F^*))$ (e.g.\ to have the existence of enough flat objects and h-flat complexes).
		\end{remark}
		\begin{remark}
			By Proposition \ref{prop: Rees gives abelian hull} and Remark \ref{rem: Rees gives abelian hull}, we see that $\bD(X,F^*)$ is equivalent to the derived category of length $n$ filtered quasi-coherent modules over $(X,F^*)$.
			Of course, by Proposition \ref{prop: Auslander} it is moreover equivalent to the derived category of right quasi-coherent modules over the associated Auslander algebra.
		\end{remark}
		
		In this subsection, we show that $\QCoh^n(X,F^*)$ has enough h-injective and h-flat complexes.
		As these are `adapted' to, respectively, the pushforward and pullback functor associated to any generalised morphism, it follows that we obtain an induced adjoint pair at the derived level.
		Showing the existence of h-flat complexes requires some work, e.g.\ we have to define what this means in the filtered context.  
		Moreover, as the existence of the latter requires our filtered schemes to be quasi-compact and separated, we will usually make this assumption from here on out.
		
		Before doing this, let us note the following.
		
		\begin{lemma}\label{lem: Db(Coh)=Bcoh(QCoh)}
			Let $(X,F^*)$ be a Noetherian $n$-filtered scheme. 
			We have 
			\[
			\bD^b_{\Coh}(X,F^*)\cong \bD^b(\Coh^n(X,F^*)),
			\]
			where the left term is the full subcategory of $\bD(X,F^*)$ consisting of complexes of modules with bounded and coherent cohomology and the right term is the full subcategory of $\bD(\Coh^n(X,F^*))$ of complexes with bounded cohomology.
		\end{lemma}
		\begin{proof}
			Copy the usual proof, see e.g.\ \cite[\href{https://stacks.math.columbia.edu/tag/0FDA}{Lemma 0FDA}]{stacks-project}.
			%
		\end{proof}
		
		\subsubsection{H-injective complexes}
		For the existence of enough injective modules and h-injective complexes it is enough to note that $\QCoh^n(X,F^*)$ is a Grothendieck abelian category.
		\begin{lemma}\label{lem: h-inj exist}
			Let $(X,F^*)$ be an $n$-filtered scheme.
			The category  $\QCoh^n(X,F^*)$ is a Grothendieck abelian category.
			Consequently, there exist functorial embeddings into injective modules and every complex functorially admits an injective quasi-isomorphism into an h-injective complex with injective components.
		\end{lemma} 
		\begin{proof}
			That $\QCoh^n(X,F^*)$ is Grothendieck follows immediately from Proposition \ref{prop: Auslander} and the corresponding statement on the $\Acal$-space side \cite[Lemma 5.6]{KuznetsovLunts}.
			The other claims are well-known to follow from this (see e.g.\ \cite[\href{https://stacks.math.columbia.edu/tag/079H}{Theorem 079H}]{stacks-project} and \cite[\href{https://stacks.math.columbia.edu/tag/079P}{Theorem 079P}]{stacks-project}).
			
			Alternatively, showing that $\QCoh^n(X,F^*)$ is Grothendieck can be done directly at the filtered level.
			The fact that it is abelian and has exact filtered colimits is clear. 
			To show that it has a generator, use Remark \ref{rem: adjunction with oplus gr-i} to `lift' the generator of $\QCoh(X)$ (the left adjoint to a faithful functor preserves generators).
		\end{proof}
		As h-injective complexes are adapted to any additive functor, these can be used to compute right derived functors.
		Thus, the right derived pushforward exists and can be computed using h-injective complexes.
		
		\subsubsection{H-flat complexes}
		We define flat modules and h-flat complexes for filtered schemes.
		This is done in such a way that they are compatible with Proposition \ref{prop: Auslander}.
		
		\paragraph{\textit{The monoidal structure}}
		Let $(X,F^*)$ be an $n$-filtered scheme.
		We start by defining a monoidal structure on $\Mod^n(X,F^*)$, which, using the descriptions of Lemma \ref{lem: char (Q)Coh}, is seen to restrict to $\QCoh^n(X,F^*)$ and $\Coh^n(X,F^*)$.
		First, note that the usual graded tensor product of two modules in $\Mod^\infty(X,F^*)$ remains in $\Mod^\infty(X,F^*)$. 
		Hence, it makes sense to truncate this resulting tensor product.
		\begin{definition}
			For $\Mcal$ and $\Ncal$ in $\Mod^n(X,F^*)$ we define their \emph{tensor product} 
			\[
			\Mcal\otimes_{(X,F^*)|n}\Ncal := l^n( \Mcal\otimes_{\tOcal_X}\Ncal ) \in \Mod^n(X,F^*),
			\]
			where $-\otimes_{\tOcal_X}-$ is the usual tensor product of sheaves of graded modules.
		\end{definition}
		As the filtration is usually clear from context, we often use the slightly abusive notation $\otimes_{X|n}$ instead of $\otimes_{(X,F^*)|n}$.
		
		The following lemma immediately implies that we obtain a symmetric monoidal structure on $\Mod^n(X,F^*)$.
		\begin{lemma}\label{lem: truncating tensor}
			For $\Mcal$ and $\Ncal$ in $\Mod^\infty(X,F^*)$ the natural morphism
			\[
			l^n( \Mcal\otimes_{\tOcal_X}\Ncal ) \isoto l^n( l^n\Mcal\otimes_{\tOcal_X}\Ncal ),
			\]
			induced by the natural morphism $\Mcal\to l^n\Mcal$, is an isomorphism.
		\end{lemma}
		\begin{proof}
			By looking at the stalks we reduce to the case of graded modules over the Rees algebra $\tR$ of a finite length filtered ring $(R,F^*)$.
			In this case it suffices to take $\Ncal=\tR(i)$ for $i\geq 0$ as both sides are right exact in $\Ncal$. (The description of Lemma \ref{lem: char (Q)Coh} holds globally in the affine case.)
			The required isomorphism then reduces to the following string of isomorphisms:
			\begin{align*}
				l^n(\Mcal\otimes_{\tR}\tR(i)) &\cong l^n(\Mcal(i)) \\	
				&\cong l^n(l^{n+i}( \Mcal(i) ) )\rlap{\quad (Lemma \ref{lem: rel ln's})} \\	
				&\cong l^n(l^{n}(\Mcal) (i) )\rlap{\quad (Lemma \ref{lem: rel ln's})} \\
				&\cong l^n(l^{n}\Mcal\otimes_{\tR} \tR (i) ).\qedhere
			\end{align*}
		\end{proof}
		
		Next, we show that this monoidal structure is compatible, in a suitable way, with the equivalence of Proposition \ref{prop: Auslander}.
		So, let, as in the proposition, $\Pcal:=\oplus_{i=0}^{n-1}l^n(\tOcal_X(i))$, $\Acal:=\Acal_{F^*}:=\sEnd_{\tOcal_X}(\Pcal)$ and $E(-):=\sHom_{\tOcal_X}(\Pcal,-)$.
		We start by associating to any $\Mcal$ in $\Mod^n(X,F^*)$ a natural $(\Acal,\Acal)$-bimodule.
		Define
		\[
		\tE(\Mcal):=E(\Pcal\otimes_{X|n}\Mcal)=\sHom_{\tOcal_X}(\Pcal, \Pcal\otimes_{X|n}\Mcal).
		\]
		This carries a natural left $\Acal$-module structure by functoriality of the tensor product. 
		Concisely, a section $a$ of $\Acal$ acts by mapping sections $g\mapsto (a\otimes_{X|n}1_\Mcal)\circ g$.
		Note that $\Acal=\tE(\tOcal_X)$.
		By writing 
		\begin{align*}
			\tE(\Mcal)&=E(\Pcal\otimes_{X|n}\Mcal)=E(\oplus_{i=0}^{n-1}l^n(\tOcal_X(i))\otimes_{X|n}\Mcal) \\
			&\cong\oplus_{i=0}^{n-1}E(l^n(\tOcal_X(i))\otimes_{X|n}\Mcal)\cong\oplus_{i=0}^{n-1}E(l^n(\Mcal(i)),
		\end{align*}
		we can think of $\tE(\Mcal)$ as an $n$-by-$n$ matrix, the left and right $\Acal$-module structures are simply matrix multiplication.
		When considering matrices we will number the rows and columns starting at zero.
		We write $e_i$ for the idempotent of $\Acal$ corresponding to projection-inclusion of the $l^n(\tR(i))$ component of $\Pcal$ (thinking of $\Acal$ as a matrix, $e_i$ corresponds to the matrix with a one in the $(i,i)$-position and zeroes everywhere else).
		
		\begin{lemma}\label{lem: E(-otimes-) = E(-)otimes tE(-)}
			For $\Mcal$ and $\Ncal$ in $\Mod^n(X,F^*)$ composition induces a functorial isomorphism
			\begin{align*}
				E(\Mcal)\otimes_\Acal \tE(\Ncal) &\isoto E(\Mcal\otimes_{X|n}\Ncal),\\
				f\otimes_\Acal g  &\longmapsto (f\otimes_{X|n}1_\Ncal)\circ g
			\end{align*}
			of right $\Acal$-modules.
		\end{lemma}
		\begin{proof}
			As in the proof of Lemma \ref{lem: truncating tensor} we reduce to the case of graded modules over the Rees algebra $\tR$ of a finite length filtered ring $(R,F^*)$ and $\Mcal=l^n(\tR(i))$ for $0\leq i<n$.
			The required isomorphism then reduces to the following string of isomorphisms:
			\begin{align*}
				E(l^n(\tR(i))) \otimes_{\Acal} \tE(\Ncal) &= e_i\Acal \otimes_{\Acal} \tE(\Ncal) \\
				&\cong e_i\tE(\Ncal) \\
				&\cong \Hom(\Pcal, l^n(\tR(i))\otimes_{X|n}\Ncal) \\
				&= E(l^n(\tR(i))\otimes_{X|n}\Ncal). \qedhere
			\end{align*}
		\end{proof}
		
		We introduce a third functor, mapping to left $\Acal$-modules.
		For this let $\gr_0$ be the projection on the zeroth graded component (defined in Lemma \ref{lem: adjunction with gr-i}).
		Define 
		\[
		E'(\Mcal):=\tE(\Mcal)e_0=\gr_0(\Pcal\otimes_{X|n}\Mcal),
		\]
		which has a natural left $\Acal$-action through the action on $\Pcal$.
		Writing $E'(\Mcal)$ as a column vector, this is $(	\Mcal^{0}\  \Mcal^{0}/t\Mcal^{-1}\  \dots\   \Mcal^{0}/t^{n-1}\Mcal^{-n+1} )^T$.
		Recall that we have the subcategory $\Filt^n(X,F^*)$ of $\Mod(X,F^*)$ consisting of filtered modules, i.e.\ those modules for which multiplication by the distinguished degree one element $t$ in $\tOcal_X$ is injective.
		The functor $E'$ is right exact and is furthermore fully faithful, exact and reflects exactness on filtered modules (see Lemma \ref{lem: E' on filt} below).
		It should be stressed that this functor is generally not an equivalence; it need not be fully faithful on non-filtered modules and need not be essentially surjective.
		
		\begin{lemma}\label{lem: gr0(-otimes-) = E(-)otimes E'(-)}
			With notation as in the previous lemma.
			We have a functorial isomorphism
			\[
			E(\Mcal)\otimes_\Acal E'(\Ncal) \isoto\gr_0(\Mcal\otimes_{X|n}\Ncal)
			\]
			of $\Ocal_X$-modules.
		\end{lemma}
		\begin{proof}
			This follows from Lemma \ref{lem: E(-otimes-) = E(-)otimes tE(-)} by multiplying on the right with $e_0$ and noting that $E(-)e_0=\gr_0(-)$.
		\end{proof}
		
		In the affine setting, the next lemma follows immediately (as there are enough projectives).
		This will be true more generally once we can speak of h-flat complexes.
		\begin{lemma}\label{lem: derived E(-otimes-) = E(-)otimes tE(-)}
			Let $M^\bullet$ and $N^\bullet$ be complexes over\:\!\footnote{To avoid confusion, here we simply mean modules, i.e.\ this is $\QCoh^n(\Spec R,F^*)$.} $\Mod^n(R,F^*)$.
			We have
			\[
			E(M^\bullet)\otimes_\Acal^\bL \bL E'(N^\bullet) \cong \gr_0(M^\bullet\otimes_{X|n}^\bL N^\bullet),
			\]
			where the derived functors are computed with h-projective resolutions (these are cheap to construct since filtered colimits are exact \cite[Corollary 3.5]{Spaltenstein}).
		\end{lemma}
		
		In the sequel it will be important to have some control over the essential image of $E'$.
		We describe this now.
		Let $\Mcal$ be a left $\Acal$-module, we can think of it as column vector $(	\Mcal^{0}\  \Mcal^{-1}\  \dots\   \Mcal^{-n+1} )^T$ where the left $\Acal$-action is given by matrix multiplication (i.e.\ $\Mcal^{-i}:=e_i\Mcal$).
		For any $0< i<n$ we have inclusions $\Acal e_i\hookrightarrow \Acal e_0$ of the $i$th column into the $0$th column which, by applying $\sHom_{\Acal}(-, \Mcal)$, gives rise to a morphism
		\begin{equation}\label{eq: s}
			s_i: \Mcal^0\to \Mcal^{-i}.
		\end{equation}
		(When thinking of $\Acal$ as a matrix, $s_i$ is given by multiplication with the matrix consisting of a one in the $(i,0)$th position and zeroes everywhere else.)
		These give us control over which left $\Acal$-modules lie in the image of $E'$.
		
		\begin{lemma}
			A left $\Acal$-module $\Mcal$ lies in the essential image of $E'$ if and only if the morphisms $s_i$ of equation \eqref{eq: s} are surjective.
			Moreover, in this case $\Mcal$ is of the form $E'(\Mcal'$) with $\Mcal'$ in $\Filt^n(X,F^*)$.
		\end{lemma}
		\begin{proof}
			For $\Mcal$ of the form $E'(\Mcal')$ the $s_i$ are exactly the quotient maps ${\Mcal'}^0\to{\Mcal'}^0/{\Mcal'}^{-i}$, so one implication is clear.
			For the other direction put ${\Mcal'}^{-i}=\ker(s_i)$. 
			This gives a filtration on $\Mcal_0$ which, viewed as an object of $\Mod^n(X,F^*)$, gets mapped by $E'$ to an $\Acal$-module isomorphic to $\Mcal$ (the $s_i$'s induce an isomorphism).
		\end{proof}
		
		\begin{lemma}\label{lem: Im E' closed under coker and contains injective}
			The essential image of $E'$ is closed under quotient objects and contains the injective left $\Acal$-modules.
		\end{lemma}
		\begin{proof}
			Clearly, the condition of the $s_i$'s being surjective is closed under quotients.
			Moreover, by the definition of the $s_i$'s they are surjective for injective modules.
		\end{proof}
		
		\begin{lemma}\label{lem: E' on filt}
			The functor $E'$ is fully faithful, exact and reflects exactness on $\Filt^n(X,F^*)$.  
		\end{lemma}
		\begin{proof}
			Given a sequence $\Kcal\to\Mcal\to\Ncal$ in $\Filt^n(X,F^*)$ the exactness claims follow from the diagram 
			\[
			\begin{tikzcd}
				& 0 & 0 & 0 \\
				0 & {\Kcal^{-i}} & {\Mcal^{-i}} & {\Ncal^{-i}} & 0 \\
				0 & {\Kcal^0} & {\Mcal^0} & {\Ncal^0} & 0 \\
				0 & {\Kcal^0/\Kcal^{-i}} & {\Mcal^0/\Mcal^{-i}} & {\Ncal^0/\Ncal^{-i}} & 0 \\
				& 0 & 0 & 0
				\arrow[from=4-1, to=4-2]
				\arrow[from=4-2, to=4-3]
				\arrow[from=4-3, to=4-4]
				\arrow[from=4-4, to=4-5]
				\arrow[from=3-4, to=3-5]
				\arrow[from=2-4, to=2-5]
				\arrow[from=2-3, to=2-4]
				\arrow[from=3-3, to=3-4]
				\arrow[from=3-2, to=3-3]
				\arrow[from=2-2, to=2-3]
				\arrow[from=2-1, to=2-2]
				\arrow[from=3-1, to=3-2]
				\arrow[from=1-2, to=2-2]
				\arrow[from=2-2, to=3-2]
				\arrow["{s_i}", from=3-2, to=4-2]
				\arrow[from=4-2, to=5-2]
				\arrow[from=1-3, to=2-3]
				\arrow[from=2-3, to=3-3]
				\arrow["{s_i}", from=3-3, to=4-3]
				\arrow[from=4-3, to=5-3]
				\arrow[from=1-4, to=2-4]
				\arrow[from=2-4, to=3-4]
				\arrow["{s_i}", from=3-4, to=4-4]
				\arrow[from=4-4, to=5-4]
			\end{tikzcd}
			\]
			with exact columns and the fact that exactness of the upper or lower two rows implies the exactness of the third.
			Moreover, it is clear from the diagram how morphisms $\Kcal\to\Mcal$ and $E'(\Kcal)\to E'(\Mcal)$ determine each other.
		\end{proof}
		
		\paragraph{\textit{Flat modules}}
		We define flat modules as those that are flat after applying $E$, and will then give an intrinsic definition using $\otimes_{X|n}$.
		
		\begin{definition}
			A module $\Mcal$ in $\Mod^n(X,F^*)$ is called \emph{flat} if $E(\Mcal)$ is flat in $\Mod(X,\Acal)$.
		\end{definition}
		
		\begin{lemma}\label{lem: Lazard}
			The stalk at $x\in X$ of any flat module $\Mcal$ is a colimit of a directed system consisting of direct sums of $l^n(\tOcal_{X,x}(i))$'s.
			Consequently, $\Mcal \otimes_{X|n} -$ maps filtered modules to filtered modules.
			In particular, $\Mcal$ is filtered.
		\end{lemma}
		\begin{proof}
			By the Govorov--Lazard Theorem applied to $E(\Mcal)_x$ we see that $\Mcal_x$ is a direct limit of the $l^n(\tOcal_{X,x}(i))$'s. 
			
			For the second statement we can reduce, as being filtered can be checked at stalks and is closed under filtered colimits, to the case $\Mcal = l^n(\tOcal_X(i))$ in which case the claim is clear.
		\end{proof}
		
		\begin{proposition}
			A module $\Mcal$ is flat if and only if $\Mcal\otimes_{X|n}-$ is exact on $\Filt^n(X,F^*)$.
		\end{proposition}
		\begin{proof}
			As $l^n$ is exact on filtered modules, we observe that $\tE(-)$ is exact on filtered modules.
			The only if direction thus follows from Lemma \ref{lem: E(-otimes-) = E(-)otimes tE(-)} by noting that $E$ is an equivalence and hence reflects exactness.
			
			For the other direction, we reduce to the usual affine setting, as everything in the statement can be checked stalkwise ($E$ commutes with taking stalks and skyscraper sheaves preserve filtered objects).
			We omit the calligraphic font.
			So, suppose $M\otimes_{R|n}-$ is exact on $\Filt^n(R,F^*)$.	
			Let $N$ be an arbitrary left $A$-module and consider a short exact sequence 
			\[
			0\to N\to I\to K\to 0
			\]
			with $I$ injective.
			By Lemma \ref{lem: Im E' closed under coker and contains injective} this short exact sequence is of the form 
			\[
			0\to N\to E'(N_1)\to E'(N_2)\to 0
			\]
			with $N_i$ in $\Filt^n(R,F^*)$.
			Hence, $\mathop\mathrm{Tor}^1_\Acal(E(M),N)=0$ by Lemma \ref{lem: derived E(-otimes-) = E(-)otimes tE(-)}.
			As $N$ was arbitrary, this implies that $E(M)$ is flat. 
		\end{proof}
		
		\paragraph{\textit{H-flat complexes}}
		Just like for flat modules we define h-flat complexes as those that are h-flat after applying $E$.
		Unfortunately we cannot give an intrinsic definition for these complexes using $\otimes_{X|n}$, but we give one for h-flat complexes with flat components.
		
		\begin{definition}
			A complex $\Mcal^\bullet$ over $\Mod^n(X,F^*)$ is called \emph{h-flat} if $E(\Mcal^\bullet)$ is h-flat over $\Mod(X,\Acal)$.
		\end{definition}
		
		\begin{proposition}\label{prop: h-flat intrinsic}
			A complex $\Mcal^\bullet$ with flat components is h-flat if and only if $\Mcal^{\bullet}\otimes_{X|n}\Ncal^{\bullet}$ is acyclic for any acyclic complex $\Ncal^\bullet$ over $\Filt^n(X,F^*)$.
		\end{proposition}
		\begin{proof}
			The only if direction follows again from Lemma \ref{lem: E(-otimes-) = E(-)otimes tE(-)}.
			
			For the other direction, we reduce to the usual affine setting, omitting the calligraphic font, as everything can again be checked at stalks.
			So, let $N^\bullet$ be an acyclic complex of left $A$-modules. 
			We have to show that $E(M^\bullet)\otimes_A N^\bullet$ is acyclic.
			By Lemma \ref{lem: h-inj exist} there exists a short exact sequence
			\[
			0\to N^\bullet \to I^\bullet \to K^\bullet\to 0,
			\]
			where every component of $I^\bullet $ is injective.
			Using Lemma \ref{lem: Im E' closed under coker and contains injective} and the fact that $E'$ is fully faithful and reflects exactness on filtered modules by Lemma \ref{lem: E' on filt}, this short exact sequence is of the form
			\[
			0\to N^\bullet\to E'(N_1^\bullet)\to E'(N_2^\bullet)\to 0
			\]
			with $N_i^\bullet$ acyclic complexes over $\Filt^N(R,F^*)$.
			As the components $E(M^i)$ are flat by assumption, this induces a short exact sequence of complexes
			\[
			0\to E(M^\bullet)\otimes_A N^\bullet\to E(M^\bullet)\otimes_A E'(N_1^\bullet)\to E(M^\bullet)\otimes_A E'(N_2^\bullet)\to 0.
			\]
			The claim now follows from Lemma \ref{lem: gr0(-otimes-) = E(-)otimes E'(-)} and the long exact sequence of cohomology.
		\end{proof}
		\begin{remark}
			Note that having flat components is only needed for the if direction.
		\end{remark}
		
		Because of the previous proposition, we will always assume that our h-flat complexes have flat components.
		We can make this assumption at no cost, as the usual construction \cite[Proposition 5.6]{Spaltenstein} of h-flat complexes automatically gives h-flat complexes with flat components.
		\begin{lemma}\label{lem: h-flat exist}
			Let $(X,F^*)$ be quasi-compact separated $n$-filtered scheme.
			The category $\QCoh^n(X,F^*)$ has enough flat modules, h-flat complexes and h-flat complexes with flat components.
		\end{lemma}
		\begin{proof}
			This follows immediately from Proposition \ref{prop: Auslander} and the corresponding statement on the $\Acal$-space side \cite[Lemma 5.6]{KuznetsovLunts}.
			
			Alternatively, showing that $\QCoh^n(X,F^*)$ has enough flat modules can be done by reducing to the affine situation following the argument of the proof in the non-filtered case, see e.g.\ the proof of \cite[Proposition 1.1]{AlonsoJeremiasLipman}.
			The existence of enough h-flat complexes and h-flat complexes with flat components then follows from this, see e.g.\ \cite[Theorem 3.4]{Spaltenstein}.
		\end{proof}
		\begin{remark}
			For the existence of enough quasi-coherent flat modules, one needs at least quasi-compactness with affine diagonal in order for the usual proof for schemes to go through,	as roughly it goes by showing the existence of enough flat modules by patching together over a finite affine cover (whose finite intersections remain affine).
		\end{remark}
		
		\paragraph{\textit{Derived pullback}}
		
		Let $(f,d):(X,{}_{dn}F^*)\to(Y,{}_{n}G^*)$ be a generalised morphism of a filtered scheme. 
		The following shows that the left derived pullback exists and can be computed using h-flat complexes (which by our convention have flat components).
		
		\begin{proposition}\label{prop: pullback preserve hflat}
			The pullback $(f,d)^*$ preserves flat modules, h-flat complexes and acyclic h-flat complexes.
		\end{proposition}
		\begin{proof}
			Define a $dn$-filtered scheme $(Y,G'^{*})$ via $G'^{i}\Ocal_Y:=G^{\lfloor i/d\rfloor}\Ocal_Y$ and let 
			\begin{align*}
				\epsilon:\Mod^n(Y,G^*)&\to\Mod^{dn}(Y,G'^*),	\\
				\oplus_{i}\Mcal^i&\mapsto \oplus_{i}\Mcal^{\lfloor i/d \rfloor}		\rlap{ .}
			\end{align*}
			By definition the pullback 
			\[
			(f,d)^*\Mcal =  f^{-1}(\epsilon(\Mcal)) \otimes_{f^{-1}(\tOcal_{(Y,G'^*)})|dn} \tOcal_{(X,F^*)} .
			\]
			Using (a slight generalisation of) Lemma \ref{lem: E(-otimes-) = E(-)otimes tE(-)} it follows that the `tensor part' preserves flat modules and (acyclic) h-flat complexes.
			Therefore it suffices to show that $\epsilon$ preserves flat modules and h-flat complexes (it preserves acyclic complexes as it is exact).
			
			The fact that $\epsilon$ preserves flat modules follows from Lemma \ref{lem: Lazard} since 
			\[
			\epsilon( l^n(\tOcal_{(Y,G^*)}(i)) ) = l^{dn}(\tOcal_{(Y,G'^*)}(di))
			\]
			is flat.
			
			To show that $\epsilon$ preserves h-flat complexes, we will first show that the morphism 
			\begin{equation}\label{eq: proj formula epsi-ver}
				\Mcal^\bullet \otimes_{(Y,G^*)|n} (\Ncal^\bullet)^{(d)} \to ( \epsilon( \Mcal^\bullet ) \otimes_{(Y,G'^*)|dn} \Ncal^\bullet )^{(d)} 
			\end{equation}
			induced by $(-)^{(d)}\circ \epsilon = 1$ and lax monoidality of $(-)^{(d)}$ is an isomorphism\footnote{This is formally somewhat similar to the projection formula.}.
			First of all note that, by definition of the tensor product of complexes, it suffices to show the morphism \eqref{eq: proj formula epsi-ver} is an isomorphism for $\Mcal^\bullet=\Mcal$ and $\Ncal^\bullet=\Ncal$ being objects concentrated in degree zero.
			Moreover, we can reduce, as usual, to the affine case with $\Mcal=l^{n}(\tOcal_{Y,G}(i))$ and $0\leq i <n$.
			The required isomorphism then reduces to the following string of isomorphisms:
			\begin{align*}
				l^{n}(\tOcal_{(Y,G^*)}(i)) \otimes_{(Y,G^*)|n} \Ncal^{(d)} &\cong l^{n}( \Ncal^{(d)}(i) ) \\
				&= l^{dn}( \Ncal(di) )^{(d)} \\
				&\cong ( l^{dn}(\tOcal_{(Y,G'^*)}(di))  \otimes_{(Y,G'^*)|dn} \Ncal )^{(d)} \\
				&= ( \epsilon( l^{n}(\tOcal_{(Y,G^*)}(i)) ) \otimes_{(Y,G'^*)|dn} \Ncal )^{(d)}.
			\end{align*}
			
			Now, let $\Mcal^\bullet$ be an h-flat complex with flat components over $\Mod^n(Y,G^*)$.
			We have to show that 
			\[
			\Kcal^\bullet:=\epsilon( \Mcal^\bullet ) \otimes_{(Y,G'^*)|dn} \Ncal^\bullet
			\]
			is acyclic for every acyclic complex $\Ncal^\bullet$ over $\Filt^n(X,F^*)$.
			By shifting and using the isomorphism \eqref{eq: proj formula epsi-ver} we see that $l^{dn}(\Kcal^\bullet (i))^{(d)}$ is acyclic for all $i\geq 0$.
			As $\Mcal^\bullet$ has flat components, and we already know that $\epsilon$ preserves flatness, by Lemma \ref{lem: Lazard} we have that $\Kcal^\bullet$ has filtered components.
			Hence, it follows by Lemma \ref{lem: ver + shift reflect exact on filt} below that $\Kcal^\bullet$ is acyclic.		
		\end{proof}
		\begin{lemma}\label{lem: ver + shift reflect exact on filt}
			With notation as in the proof of the above proposition. 
			Let $\Mcal^\bullet$ be a complex over $\Filt^{dn}(Y,G'^*)$ and suppose $l^{dn}(\Mcal^\bullet(i))^{(d)}$ is acyclic for all $i\geq 0$.
			Then, $\Mcal^\bullet$ is acyclic.
		\end{lemma}
		\begin{proof}
			We have to show that $\gr_j\Mcal^\bullet$ is acyclic for all $-dn< j \leq 0$.
			This is clear for $j=0$ as 
			\[
			\gr_0\Mcal^\bullet = \gr_0l^{dn}(\Mcal^\bullet(0))^{(d)} .
			\]
			For $j<0$ note that 
			\[
			\gr_0l^{dn}(\Mcal^\bullet(j+dn))^{(d)} = \gr_0\Mcal^\bullet/\gr_j\Mcal^\bullet.
			\]
			Hence, the acyclicity of $\gr_j\Mcal^\bullet$ follows by the long exact sequence of cohomology obtained from the short exact sequence
			\[
			0 \to \gr_j\Mcal^\bullet \to \gr_0\Mcal^\bullet \to \gr_0\Mcal^\bullet/\gr_j\Mcal^\bullet \to 0. \qedhere
			\]
		\end{proof}
		\begin{remark}
			It is somewhat unfortunate that we had to use Proposition \ref{prop: h-flat intrinsic}, and thus as a result need to assume our h-flat complexes have flat components.
			The main reason for this is that taking refinements is not very natural on the $\Acal$-space side.
			So, using an intrinsic characterisation seems necessary to show that $\epsilon$ preserves h-flats.
		\end{remark}
		
		\subsubsection{Induced derived adjunction}
		As a direct result of the above we have the following proposition.
		\begin{proposition}
			Any generalised morphism of quasi-compact separated filtered schemes $(f,d):(X,{}_{dn}F^*)\rightsquigarrow(Y,{}_{n}G^*)$ induces a derived pullback/pushforward-like adjunction
			\[
			\begin{tikzcd}[sep=2.5em]
				{\bD(X,F^*)} \\ {\bD(Y,G^*)}\rlap{ .}
				\arrow[""{name=0, anchor=center, inner sep=0}, "{\bR(f,d)_{*}}", bend left=45, from=1-1, to=2-1]
				\arrow[""{name=1, anchor=center, inner sep=0}, "{\bL(f,d)^{*}}", bend left=45, from=2-1, to=1-1]
				\arrow["\dashv"{anchor=center}, draw=none, from=1, to=0]
			\end{tikzcd}
			\]
			Moreover, these compose well: if $n=em$ and $(g,e):(Y,{}_{em}G^*)\rightsquigarrow(Z,{}_{m}H^*)$ is an $e$-morphism of quasi-compact separated filtered schemes, then $\bL(f,d)^{*}\circ\bL(g,e)^{*} \cong \bL(g\circ f,de)^{*}$ and  $\bR(g,e)_{*}\circ\bR(f,d)_{*} \cong \bR(g\circ f,de)_{*}$.
		\end{proposition}
		\begin{proof}
			The derived functors exist as quasi-compact separated filtered schemes have enough h-flat and h-injective complexes and these are adapted to respectively the pullback and pushforward functor; \cite[\href{https://stacks.math.columbia.edu/tag/09T5}{Lemma 09T5}]{stacks-project} then shows that they are adjoint.
			
			To prove the composition claims, we note that, by Proposition \ref{prop: pullback preserve hflat}, we have $\bL(f,d)^{*}\circ\bL(g,e)^{*} \cong \bL(g\circ f,de)^{*}$.
			Consequently, by adjointness\footnote{This implies that the canonical map, obtained from the universal property of the right derived functor, is an isomorphism.} $\bR(g,e)_{*}\circ\bR(f,d)_{*} \cong \bR(g\circ f,de)_{*}$.
		\end{proof}
		
		As a corollary, we have the following.
		
		\begin{corollary}\label{cor: adjunction pullpush on RHom}
			With notation as in the above proposition, we have functorial isomorphisms
			\[
			\RHom_X({\bL(f,d)^{*}}M,N) \isoto \RHom_Y(M,{\bR(f,d)_{*}}N)\quad(\text{in }\bD(\Ab))
			\]
			for $M\in \bD(Y,G^*)$ and $N\in \bD(X,F^*)$.
		\end{corollary}
		\begin{proof}
			The trick of Lipman \cite[Corollary 3.2.2]{LipmanGrothendieck}, showing that the natural morphism $$\Hom^\bullet_Y(M,{(f,d)_{*}}N)\to \RHom_Y(M,{(f,d)_{*}}N)$$ is a quasi-isomorphism for $M$ h-flat and $N$ h-injective works here too, allowing one to reduce to the non-derived setting by picking an h-flat and h-injective resolution.
		\end{proof}
		
		\subsubsection{Derived compatibilities with schemes and \texorpdfstring{$\Acal$}{A}-spaces}\label{subsubsec: derived compat}
		We collect some compatibilities.
		Let $\bD(X):=\bD(\QCoh(X))$ denote the usual derived category of quasi-coherent modules over a scheme $X$. 
		Moreover, note that the functor $\gr_{-i}$ from Lemma \ref{lem: adjunction with gr-i} is exact and thus graciously descends to the derived category.
		
		\begin{lemma}\label{lem: compat filtered der push with underlying der push}
			Let $(f,d):(X,{}_{dn}F^*)\rightsquigarrow(Y,{}_{n}G^*)$ be a generalised morphism of quasi-compact separated filtered schemes.
			Then, the diagram
			\[
			\begin{tikzcd}[column sep=4em]
				{\bD(X,F^*)} & {\bD(X)} \\
				{\bD(Y,G^*)} & {\bD(Y)}
				\arrow["{\bR(f,d)_*}"', from=1-1, to=2-1]
				\arrow["{\gr_{X,-di}}", from=1-1, to=1-2]
				\arrow["{\gr_{Y,-i}}", from=2-1, to=2-2]
				\arrow["{\bR f_*}", from=1-2, to=2-2]
			\end{tikzcd}
			\]
			(2-)commutes for all $0\leq i <n$.
			In particular, when the filtered schemes are in addition Noetherian and $f$ is proper, the derived pushforward preserves coherence, i.e.\ \allowbreak 
			\[
			\bR (f,d)_*( \bD^b(\Coh^{dn}(X,F^*)) )\subseteq \bD^b(\Coh^n(Y,G^*)).
			\]
		\end{lemma}
		\begin{proof}
			Both sides are compositions of right adjoints, so it suffices to check that the compositions of their left adjoints are isomorphic.
			As the left adjoints to $\gr_{X,-di}$ and $\gr_{Y,-i}$ are given by tensoring (see Lemma \ref{lem: adjunction with gr-i}) these preserve h-flat complexes (with flat components).
			The isomorphism thus follows as, with notation from Lemma \ref{lem: pull/push},
			\begin{align*}
				&(f,d)^*( \Mcal\otimes_{\Ocal_Y} l^n(\tOcal_{(Y,G^*)}(i)) ) \\
				=& f^{-1}(\epsilon(\Mcal\otimes_{\Ocal_Y} l^n(\tOcal_{(Y,G^*)}(i)))) \otimes_{f^{-1}(\tOcal_{(Y,G'^*)})|dn} \tOcal_{(X,F^*)} \\
				=& f^{-1}(\Mcal\otimes_{\Ocal_Y} l^{dn}(\tOcal_{(Y,G'^*)}(di))) \otimes_{f^{-1}(\tOcal_{(Y,G'^*)})|dn} \tOcal_{(X,F^*)} \\
				=& f^{-1}(\Mcal)\otimes_{f^{-1}(\Ocal_Y)} l^{dn}(\tOcal_{(X,F^*)}(di)) \\
				=& f^{*}(\Mcal)\otimes_{\Ocal_X} l^{dn}(\tOcal_{(X,F^*)}(di)).
			\end{align*}
			
			The second statement follows immediately from the corresponding statement for schemes, see e.g.\ \cite[Theorem 3.2.1]{EGAIII1}, since we can check coherence componentwise.
		\end{proof}
		
		Let us finish this subsection by mentioning that the diagrams \eqref{eq: pull/push compat Auslander}, stating the compatibility of the pullback/pushforward functor with the $\Acal$-space side, descend to the derived level.
		Indeed, the equivalences $\bD(X,F^*)\cong\bD(X,\Acal_X)$ and  $\bD(Y,F^*)\cong\bD(Y,\Acal_Y)$, induced by Proposition \ref{prop: Auslander}, clearly preserve h-injective and h-flat complexes (with flat components). 
		
	\subsection{Perfect complexes}\label{subsec: perf comp}
		Let $(X,F^*)$ be an $n$-filtered scheme.
		We define and collect some facts concerning perfect complexes over $(X,F^*)$.
		
		Analogous to schemes, we say that a complex is \emph{strictly perfect} if it is bounded with terms consisting of finite direct sums of the objects $l^n(\tOcal_X(i))$ for $0\leq i<n$.
		(Direct summands of these objects would again be (locally) of this form\footnote{
			Namely, $\tOcal_{X,x}$ is graded local, with unique homogeneous maximal ideal $\widetilde{\mathfrak{m}}$ induced by the maximal ideal $\mathfrak{m}$ of $\Ocal_{X,x}$ ($\widetilde{\mathfrak{m}}$ equals $\tOcal_{X,x}$ except that it has $\mathfrak{m}$ as 0th piece).
			Using this, one can show that every finitely generated projective module in  $\Mod^n(\Ocal_{X,x},F^*)$ is a direct sum of $l^n(\tOcal_X(i))_{x}$'s (the proof is exactly the same as in the non-filtered $n=1$ case, $\tOcal_{X,x}/\widetilde{\mathfrak{m}}=\Ocal_{X,x}/\mathfrak{m}$ is a field).
			Now, use an extension of \cite[\href{https://stacks.math.columbia.edu/tag/0B8J}{Lemma 0B8J}]{stacks-project}.
		}.)
		We then define \emph{the category of perfect complexes} $\Perf(X,F^*)$ to be the full subcategory of $\bD(X,F^*)$ consisting of those complexes that are locally quasi-isomorphic to strictly perfect complexes.
		\begin{remark}
			Under the equivalence of Proposition \ref{prop: Auslander} the perfect complexes over $(X,F^*)$ correspond to the perfect complexes over $(X,\Acal_{F^*})$, i.e.\ those complexes that are locally quasi-isomorphic to bounded complexes having direct summands of finite free $\Acal_{F^*}$-modules as components.
			This follows as, with notation as in the proposition, $\Pcal$ is mapped to $\Acal_{F^*}$ and the equivalence is compatible with restriction to open subsets.
		\end{remark}
		
		\begin{proposition}\label{prop: compact iff perfect on nice fSch}
			Let $(X,F^*)$ be a quasi-compact separated $n$-filtered scheme.
			An object $M$ of $\bD(X,F^*)$ is compact if and only if it is perfect.
		\end{proposition}
		\begin{proof}
			Suppose $M$ is compact.
			As the underlying scheme is quasi-compact and separated, derived pushforwards along open immersions preserve direct sums\footnote{One can reprove this in our setting; the proof for schemes works, but one can also reduce to the non-filtered case using Lemma \ref{lem: compat filtered der push with underlying der push}.}.
			Hence, restriction to open subsets preserves compactness by Lemma \ref{lem: nice functor preserving or reflecting compactness}.
			Thus, we reduce to the case where $X$ is affine and consequently $\oplus_{i=0}^{n-1}l^n(\tOcal_X(i))$ is a compact generator.
			So $M\in\bD(X,F^*)^c=\Thick(\oplus_{i=0}^{n-1}l^n(\tOcal_X(i)))$ is strictly perfect.
			
			Conversely, suppose $M$ is perfect.
			As in the proof of \cite[Lemma 3.3.7]{BondalVandenBergh} we can reduce to the case where $X$ is affine (essentially by looking at Mayer-Vietoris type distinguished triangles).
			Then, suitably adjusting \cite[\href{https://stacks.math.columbia.edu/tag/08EB}{Lemma 08EB}]{stacks-project}, one can show that $M$ is quasi-isomorphic to a bounded complex of projective modules, in which case it is clearly compact.
			Alternatively, when $X$ is additionally Noetherian, one can, using Lemma \ref{lem: Db(Coh)=Bcoh(QCoh)}, copy \cite[Theorem 4.1 (iii) $\Rightarrow$ (i)]{AvramovIyengarLipman} to show that $M$ is quasi-isomorphic to a bounded complex of projective modules.
		\end{proof}
		
		\begin{proposition}\label{prop: perf ess small}
			Let $(X,F^*)$ be a quasi-compact separated $n$-filtered scheme.
			The derived category $\bD(X,F^*)$ is compactly generated, consequently $\Perf(X,F^*)$ is essentially small (i.e.\ equivalent to a small category). 
		\end{proposition}
		\begin{proof}
			Let $\Pcal:=\oplus_{i=0}^{n-1}l^n(\tOcal_X(i))$.
			By Remark \ref{rem: adjunction with oplus gr-i} we have an adjunction (ignoring the components in positive degree) 
			\[
			\begin{tikzcd}[sep=2.5em]
				\bD(X,F^*) \\ \bD(X) \rlap{ .}
				\arrow[""{name=0, anchor=center, inner sep=0}, "{(-)|_{\Ocal_X}}", bend left=45, from=1-1, to=2-1]
				\arrow[""{name=1, anchor=center, inner sep=0}, "{(-\otimes_{\Ocal_X}^\bL \Pcal)}", bend left=45, from=2-1, to=1-1]
				\arrow["\dashv"{anchor=center}, draw=none, from=1, to=0]
			\end{tikzcd}
			\]
			It is well-known \cite[Theorem 3.1.1]{BondalVandenBergh} that $\bD(X)$ admits a compact generator ${U}$.
			As $(-)|_{\Ocal_X}$ is exact, commutes with coproducts and moreover reflects the zero object, it follows by abstract nonsense that ${U}\otimes_{\Ocal_X}^\bL \Pcal$ is a compact generator for $\bD(X,F^*)$.
		\end{proof}
		
	\subsection{Semi-orthogonal decompositions}\label{subsec: SODs}
		
		We make use of the following notation.
		Let $X$ be a scheme and $\Ical$ a quasi-coherent ideal sheaf, then $\VV_X(\Ical)$ denotes the closed subscheme of $X$ defined by $\Ical$.
		
		The following are two results of \cite{KuznetsovLunts} translated into the filtered language.
		\begin{proposition}\label{prop: filtered SODs}
			Let $(X,F^*)$ be a Noetherian $n$-filtered scheme, put $X_0:=\VV_X(F^{-1}\Ocal_X)$.
			There are semi-orthogonal decompositions
			\begin{align*}
				\bD(X,F^*) &= \langle \overbrace{\bD(X_0),\bD(X_0),\dots,\bD(X_0)}^{n\text{ components}} \rangle \\
				\bD^b(\Coh^n(X,F^*)) &= \langle \underbrace{\bD^b(\Coh(X_0)),\bD^b(\Coh(X_0)),\dots,\bD^b(\Coh(X_0))}_{n\text{ components}} \rangle.
			\end{align*}
		\end{proposition}
		\begin{proof}
			This follows from Proposition \ref{prop: Auslander} and \cite[Corollary 5.15]{KuznetsovLunts} (in loc.\ cit.\ $X$ is assumed to be separated and of finite type over a field, but this is not needed).
		\end{proof}
		
		\begin{proposition}\label{prop: Perf(X,F)=Db(Coh(X,F))}
			Let $(X,F^*)$ be a separated $n$-filtered scheme of finite type over $\kk$.
			Suppose $X_0:=\VV_X(F^{-1}\Ocal_X)$ is smooth, then $\Perf(X,F^*)=\bD^b(\Coh^n(X,F^*))$.
		\end{proposition}
		\begin{proof}
			This follows from Proposition \ref{prop: Auslander} and \cite[Theorem 5.17]{KuznetsovLunts}.
			It makes use of the semi-orthogonal decompositions of Proposition \ref{prop: filtered SODs} being `nice'.
		\end{proof}
		
	\subsection{Functorial enhancements}\label{subsec: the enhancement}
		We construct \emph{functorial} small enhancements of separated finite length filtered schemes of finite type by combining \cite[Theorem 3.11]{KuznetsovLunts} with the usual passage from a pseudo-functor to a functor, see e.g.\ \cite[Subsection 4.1]{BlockHolsteinWei} or \cite[Theorem 3.45]{Vistoli}.
		
		Let $\underline{\fSch}$ be the category consisting of the following:
		\begin{itemize}
			\item objects are separated finite length filtered schemes of finite type over a field $\kk$ of characteristic zero,
			\item morphisms are generalised morphisms.
		\end{itemize}
		Furthermore, let $\dgCat$ and $\Tri$ denote, respectively, the (2-)`category' of big (i.e.\ not necessarily small) dg categories and triangulated categories.
		
		For any filtered scheme $(X,F^*)$ in $\underline{\fSch}$ we have enough h-flat complexes and therefore an enhancement of its derived category by taking the Drinfeld dg quotient of the (big) dg category  $\hflat(X,F^*)$  of h-flat complexes with flat components by the full (big) dg subcategory  $\hflat^{\circ}(X,F^*)$ of those complexes that are additionally acyclic.
		For convenience, we henceforth identify\footnote{Note that, if we define $\hflat(X,\Acal_{F^*})$ as the dg category of h-flat complexes with flat components of right modules over the corresponding Auslander algebra, then $\hflat(X,F^*)$ and $\hflat(X,\Ascr_{F^*})$ are dg equivalent.	So, the dg enhancements constructed in \cite{KuznetsovLunts} on the $\Acal$-space side are quasi-equivalent to the ones defined in this section.}
		\[
		\bD(X,F^*)=[\hflat(X,F^*)/\hflat^{\circ}(X,F^*)].
		\]
		Moreover, as the pullback along any generalised morphism preserves (acyclic) h-flat complexes and flat objects it induces a dg functor between these enhancements.
		This gives a pseudo-functor
		\begin{equation}\label{eq: big enhancement of fSch}
			\hflat/\hflat^{\circ}:\underline{\fSch}^{\op}\to \dgCat.
		\end{equation}
		(One way of seeing that this is a pseudo-functor is as follows. 
		Pulling back gives a pseudo-functor on the level of complexes of quasi-coherent modules, as the pullback is adjoint to pushforward and we have actual equality of composition for the pushforward, c.f.\ \cite[Subsection 3.2.1]{Vistoli}.
		This then descends to the category $\hflat$ and to its dg quotient.)
		Moreover, by post-composing with $H^0$ we obtain a pseudo-functor $\bD:\underline{\fSch}^{\op}\to \Tri$.
		
		Let $\hflatperf(X,F^*)$ denote the full (big) dg subcategory of $\hflat(X,F^*)$ of those complexes that are additionally perfect.
		Then, the pseudo-functor \eqref{eq: big enhancement of fSch} restricts to a pseudo-functor $\hflatperf/\hflat^{\circ}:\underline{\fSch}^{\op}\to \dgCat$ as pullback preserves perfectness.
		This gives big enhancements of the perfect complexes over filtered schemes.
		
		Lastly, we also have a pseudo-functor $\bD:\dgcat\to\Tri$ by associating to any dg category its derived category and any dg functor its induction functor.
		(A similar reasoning as before shows that this is a pseudo-functor.)
		
		The main result of this section is the following proposition.
		\begin{proposition}\label{prop: funct enhancements}
			There exists a functor $\Dscr : \underline{\fSch}^{\op} \to \dgcat$ such that the diagram
			\begin{equation}\label{eq: compatibility of (pseudo)-functors}
				\begin{tikzcd}[  sep=normal]
					{\underline{\fSch}}^{\op} && \dgCat \\
					& \dgcat \\[1em]
					& \Tri
					\arrow["\Dscr"{description}, from=1-1, to=2-2]
					\arrow["inc"{description}, from=2-2, to=1-3]
					\arrow["\bD"{description}, from=2-2, to=3-2]
					\arrow["\bD"{description}, bend right, from=1-1, to=3-2]
					\arrow["{\hflatperf/\hflat^{\circ}}"{description}, bend left, from=1-1, to=1-3]
					\arrow["\bD"{description}, bend left, from=1-3, to=3-2]
				\end{tikzcd}
			\end{equation}
			is 2-commutative (i.e.\ the triangles commutate up to `natural equivalences', the upper triangle is up to fully faithful quasi-equivalence, the bottom left is up to equivalence and the bottom right is simply equality).
		\end{proposition}
		\begin{proof}
			First note that the commutativity of the bottom left triangle follows from that of the upper one, as we have a natural transformation of pseudo-functors
			\[
			\bD(\hflatperf/\hflat^{\circ})\cong[\hflat/\hflat^{\circ}]=\bD,
			\]
			which is proven in exactly the same manner as the latter part of \cite[Theorem 3.11]{KuznetsovLunts}.
			
			Thus, it suffices to construct $\Dscr$ and to show the commutativity of the upper triangle.
			For the construction note that the category $\underline{\fSch}$ is essentially small, as finite type schemes are essentially small and the addition of a filtration does not change this.
			Let $S$ denote a set of representatives of isomorphism classes of objects in $\underline{\fSch}$.
			
			Take $(Y,G^*)$ in $S$.
			By Proposition \ref{prop: perf ess small} $\Perf(Y,G^*)$ is essentially small.
			Hence we can find a small dg subcategory $\Dscr_{0}(Y,G^*)$ of $\hflatperf/\hflat^{\circ}(Y,G^*)$ inducing an equivalence on the homotopy categories, i.e.
			\[
			[\Dscr_0(Y,G^*)]\cong [\hflatperf/\hflat^{\circ}(Y,G^*)] \cong \Perf(Y,G^*).
			\]
			
			Now, let $(X,F^*)$ be arbitrary, we define $\Dscr(X,F^*)$ as follows.
			\begin{itemize}
				\item Objects: pairs $(g,M)$ where $g:(X,F^*)\rightsquigarrow(Y,G^*)$ is a morphism with target in $S$ and $M$ is an object of $\Dscr_{0}(Y,G^*)$.
				\item Morphisms: $\Hom( (g_{1},M_{1}), (g_{2},M_{2}) ) := \Hom( g_{1}^*M_{1}, g_{2}^*M_{2} ) $ where the latter $\Hom$ is taken in $\hflatperf/\hflat^{\circ}(X,F^*)$.
			\end{itemize}
			As $S$ is a set, we also have a set of morphisms with target in $S$. 
			Therefore, as the $\Dscr_{0}$'s are small, $\Dscr$ is small.
			
			In the following we will denote by 
			\[
			\alpha_{f,g}:f^*g^*\isoto(gf)^*\text{  and }\epsilon_{(X,F^*)}:(\id_{(X,F^*)})^*\isoto\id_{\hflatperf/\hflat^{\circ}(X,F^*)}
			\]
			the coherence isomorphisms of the pseudo-functor $\hflatperf/\hflat^{\circ}$.
			In order to make $\Dscr$ into a functor we define for any morphism $f:(X,F^*)\rightsquigarrow (X',F'^*)$ a dg functor $f^*:\Dscr(X',F'^*)\rightsquigarrow\Dscr(X,F^*)$ as follows.
			\begin{itemize}
				\item On objects: $(g,M)\mapsto (gf,M)$.
				\item On morphisms: sent $\phi: (g_{1},M_{1})\to (g_{2},M_{2})$ to the upper morphism in the diagram
				\begin{equation}\label{eq: f*:Dscr->Dscr on morphisms}
					\begin{tikzcd}[row sep=2.5em]
						{(g_1f)^*M_1} & {(g_2f)^*M_2} \\
						{f^*g_1^*M_1} & {f^*g_2^*M_2}\rlap{ .}
						\arrow[dashed, from=1-1, to=1-2]
						\arrow["{f^*\phi}", from=2-1, to=2-2]
						\arrow["{\alpha_{f,g_1}}"{description}, from=2-1, to=1-1]
						\arrow["{\alpha_{f,g_1}}"{description}, from=2-2, to=1-2]	
					\end{tikzcd}
				\end{equation}
			\end{itemize}
			This gives us the sought after functor $\Dscr:\underline{\fSch}\to\dgcat$; functoriality follows by the naturality and coherence conditions of the $\alpha$'s and $\epsilon$'s.
			
			It remains to show the claim concerning the upper triangle in \eqref{eq: compatibility of (pseudo)-functors}.
			Define a dg functor $\Dscr(X,F^*)\to \hflatperf/\hflat^{\circ}(X,F^*)$ on objects by $(g,M)\mapsto g^*M$ and on morphisms as equality. 
			Clearly, this functor is fully faithful. 
			Moreover, it gives an equivalence on $H^0$.
			Indeed, let $M$ be an arbitrary h-flat perfect complex over $(X,F^*)$ and pick an isomorphism $s:(X,F^*)\to(Y,F^*)$ with target in $S$. Then, $M$ is homotopic to $s^*M'$ where $M'$ is any object of $\Dscr_{0}(Y,F^*)$ homotopic to $s_{*}M'$. 
			
			Lastly, for any $f:(X,F^*)\rightsquigarrow (X',F'^*)$, the $\alpha_{f,-}$'s give, using \eqref{eq: f*:Dscr->Dscr on morphisms}, a natural isomorphism making the diagram
			\[\begin{tikzcd}[row sep=2.5em]
				{\Dscr(X',F'^*)} & {\hflatperf/\hflat^{\circ}(X',F'^*)} \\
				{\Dscr(X,F^*)} & {\hflatperf/\hflat^{\circ}(X,F^*)}
				\arrow[from=2-1, to=2-2]
				\arrow[from=1-1, to=1-2]
				\arrow["{f^*}"{description}, from=1-1, to=2-1]
				\arrow["{f^*}"{description}, from=1-2, to=2-2]
				\arrow["\alpha_{f,-}"', shorten <=4pt, shorten >=4pt, Rightarrow, from=1-2, to=2-1]
			\end{tikzcd}\]
			2-commute. 
		\end{proof}
		
		\begin{remark}
			In particular, with notation as in the proposition, we have for each generalised morphism $f:(X,F^*)\rightsquigarrow (Y,G^*)$ in $\underline{\fSch}$ a 2-commutative square
			\[
			\begin{tikzcd}[column sep=3em]
				\bD(\Dscr(Y,G^*))\arrow[r, "\bL\!\Ind_{{f^*}}"]\arrow[d, "\rotatebox{90}{\(\sim\)}"] & \bD(\Dscr(X,F^*))\arrow[d, "\rotatebox{90}{\(\sim\)}"] \\
				\bD(Y,G^*)\arrow[r, "\bL{f^*}"] & \bD(X,F^*)
			\end{tikzcd}
			\]
			where the vertical morphisms are equivalences.
		\end{remark}
		
		\begin{remark}
			It seems likely that the proof of \cite[Theorem B]{CanonacoNeemanStellari} generalises to the filtered setting, showing that the category of perfect complexes over a quasi-compact (quasi-)separated finite length filtered scheme has a unique dg enhancement.
			For the derived category this is known by Theorem A of loc.\ cit.\ which shows that for any abelian category $\Asf$ the derived category $\bD^?(\Asf)$, for $?= b, +, -, \varnothing$, has a unique dg enhancement.
		\end{remark}
		
		The following is \cite[Theorem 5.20]{KuznetsovLunts} translated into the filtered language.
		
		\begin{proposition}\label{prop: smoothness Perf(X,F)}
			Let $(X,F^*)$ be a filtered scheme in $\underline{\fSch}$.
			Suppose $X_0:=\VV_X(F^{-1}\Ocal_X)$ is smooth (over $\kk$), then $\Dscr(X,F^*)$ is a smooth dg category (over $\kk$).
			If $X_0$ is additionally proper, then $\Dscr(X,F^*)$ is dg proper.
		\end{proposition}
		\begin{proof}
			This follows from Proposition \ref{prop: Auslander} and \cite[Theorem 5.20]{KuznetsovLunts}.
			It is a consequence of the semi-orthogonal decomposition of Proposition \ref{prop: filtered SODs} having appropriate perfect gluing bimodules.
		\end{proof}
		
	\subsection{Filtered blow-ups}\label{subsec: filt blow}
		We start by briefly recalling some aspects of the relative Proj construction in the non-filtered setting, see for example \cite[Section 3]{EGAII} for more information.
		
		Thus, let $X$ be a scheme and $\Acal$ a commutative quasi-coherent graded $\Ocal_X$-algebra.
		One can construct a scheme $\rProj_X\Acal$ by gluing together the usual Proj of a graded ring $\Proj\Acal(U)$ for $\Spec R=U\subseteq X$ affine open (see e.g.\ \cite[\href{https://stacks.math.columbia.edu/tag/01M3}{Section 01M3}]{stacks-project} for its construction).
		As $\Acal(U)$ is an $R$-algebra, $\Proj\Acal(U)$ is equipped with a natural morphism to $U$.
		These glue to give a natural morphism $\rProj_X\Acal\to X$.
		Moreover, to any quasi-coherent $\ZZ$-graded $\Acal$-module $\Mcal$ we can associate a quasi-coherent sheaf over $\rProj_X \Acal$. 
		We will use $\widetilde{\Mcal}$ or $\Mcal^\sim$ to denote this associated sheaf, as is customary.
		In order to avoid confusion with our notation for the associated Rees algebra of a filtered algebra, we will simply denote the latter $\Rees(\dots)$ in this subsection and the next. 
		The functor $^{\sim}$ is exact.
		
		Now, let $(X,F^*)$ be an $n$-filtered scheme.
		We extend the relative Proj construction to the filtered setting.
		Recall that $\Filt^n(X,F^*)$ is the category of length $n$ filtered $(\Ocal_X,F^*)$-modules; it consists of those filtered modules $(\Mcal,F^*)$ with $F^0\Mcal=\Mcal$ and $F^{-n}\Mcal=0$. 
		Consider an $\NN$-graded commutative algebra object $(\Acal,F^*)$ of $\Filt^n(X,F^*)$ whose filtration pieces are quasi-coherent. Explicitly, $(\Acal,F^*)$ consists of a collection\footnote{\label{foot: lower upper index}
			We use lower indices here as, in the next subsection, we will have to consider $\Rees(\Acal,F^*)=\oplus_j F^j\Acal t^j=\oplus_{i,j}F^j\Acal_i t^j$, which is a $\ZZ^2$-graded algebra; having upper and lower indexes will help differentiate between the gradings.
		} $\{(\Acal_i,F^*)\}_{i\in\NN}\subseteq \Filt^n(X,F^*)$ s.t.\ every $F^j\Acal_i$ is quasi-coherent as $\Ocal_X$-module together with commutative unitality and associativity (filtered) maps.
		Alternatively, we can think of $(\Acal,F^*)$ as a commutative quasi-coherent $\NN$-graded $\Ocal_X$-algebra equipped with a length $n$ filtration (i.e.\ satisfying $F^0\Acal=\Acal$ and $F^{-n}\Acal=0$) that is compatible with the filtration of $(\Ocal_X,F^*)$, i.e.\ $(\Acal,F^*)$ is a filtered $(\Ocal_X,F^*)$-module.
		See also the next paragraph. 
		
		By looking at the $F^0$-part of the filtration we obtain a quasi-coherent $\NN$-graded $\Ocal_X$-algebra $ F^0\Acal:=\oplus_i  F^0\Acal_i (=\oplus_i \Acal_i=\Acal)$ and we can consider the scheme $Y:=\rProj_X(F^0\Acal)$. 
		Moreover, every filtered piece $F^j\Acal$ can be viewed as a graded $ F^0\Acal$-module $\oplus_i F^j\Acal_i$ which is quasi-coherent as $\Ocal_X$-module.
		Thus, we can associate a quasi-coherent module $F^j\Ocal_Y:=({F^j\Acal})^\sim$ over $Y$, and note that $F^0\Ocal_Y=\Ocal_Y$.
		Therefore, we have a natural filtration on $\Ocal_Y$, which moreover has length $n$. 
		Denote the associated $n$-filtered scheme $(Y,F^*)$ by $\rProj_X(\Acal,F^*)$.
		Furthermore, note that the canonical morphism $\pi:Y\to X$ is compatible with the filtrations, thus giving a morphism of filtered schemes $\pi:(Y,F^*)\to (X,F^*)$.
		Of course, all of this also makes sense for filtered schemes with unbounded filtration.
		
		This leads to the following extension of blowing up in the filtered setting.
		Let $\Ical\subset\Ocal_X$ be a quasi-coherent sheaf of ideals. (There is no distinction between filtered ideal sheaves of $(\Ocal_X,F^*)$ and ideal sheaves of $\Ocal_X$ as the filtration is uniquely determined by the inclusion being a strict monomorphism, see also the beginning of the next subsection.)
		Denote by $\Acal:=\oplus_i \Ical^i$ the Rees algebra of $\Ical$ which we view as an $\NN$-graded algebra and equip every $\Acal_i=\Ical^i$ with the filtration induced from $(\Ocal_X,F^*)$, i.e.\ $F^j\Acal_i:=\Ical^i\cap F^j\Ocal_X$.
		Then, we are in the above setting and can apply the relative Proj construction to $(\Acal,F^*)$. 
		We obtain an $n$-filtered scheme $\Bl_\Ical(X,F^*):=\rProj_X(\Acal,F^*)$ which we refer to as the \emph{filtered blow-up of $(X,F^*)$ along $\Ical$}, and a canonical filtered morphism $\pi:\Bl_\Ical(X,F^*)\to(X,F^*)$.
		
		The next proposition, although straightforward, will be important as it lets us lift blow-ups in some sense, which we make precise now.
		Recall that for a scheme $X$ and a quasi-coherent ideal sheaf $\Ical$, the closed subscheme defined by $\Ical$ is denoted by $\VV_X(\Ical)$.
		We say that a morphism of filtered schemes $\pi:(Y,F^*)\to(X,F^*)$ \emph{lifts} a morphism of schemes $\overline{\pi}:\overline{Y}\to\overline{X}$ when the following statements hold:
		\begin{enumerate}
			\item\label{item: lift1} $\overline{X}=\VV_X(F^{-1}\Ocal_X)$,
			\item\label{item: lift2} $\overline{Y}=\VV_Y(F^{-1}\Ocal_Y)$,
			\item\label{item: lift3} the morphism 
			\[
			\begin{tikzcd}[row sep = 0.6em]
				{\overline{Y}} & {\overline{X}} \\
				{\VV_Y(F^{-1}\mathcal{O}_Y)} & {\VV_X(F^{-1}\mathcal{O}_X)}
				\arrow[from=2-1, to=2-2]
				\arrow[equal, from=1-1, to=2-1]
				\arrow[equal, from=1-2, to=2-2]
			\end{tikzcd}
			\] 
			induced by $\pi$ equals $\overline{\pi}$.
		\end{enumerate}
		Of course, it is always possible to find a lifting by endowing $\overline{Y}$ with the trivial filtration, but this is not always what one wants.
		Moreover, more generally, we can weaken the above, only requiring isomorphisms in \ref{item: lift1} and \ref{item: lift2} and then requiring the induced morphism in \ref{item: lift3} to be equal modulo those isomorphisms.
		(The reason we are interested in the closed subschemes defined by the $F^{-1}$-parts of the filtration is because of their presence in Proposition \ref{prop: smoothness Perf(X,F)}.)
		
		\begin{proposition}\label{prop: lifting blowups}
			Let $(X,F^*)$ be a filtered scheme and $\overline{X} := \VV_X(F^{-1}\Ocal_X)$.
			Suppose 
			\[
			\overline{\pi}:\overline{Y} := \Bl_{\overline{\Ical}}(\overline{X})\to \overline{X}
			\] 
			is the blow-up along a quasi-coherent sheaf of ideals $\overline{\Ical}\subseteq\Ocal_{\overline{X}}$.
			Then, there exists a quasi-coherent sheaf of ideals $\Ical\subseteq\Ocal_X$ such that $$\pi:(Y,F^*):=\Bl_\Ical(X,F^*)\to(X,F^*),$$ the filtered blow-up at $\Ical$, lifts $\overline{\pi}$.
		\end{proposition}
		\begin{proof}
			Let $i:\overline{X}\hookrightarrow X$ denote the closed embedding.
			Consider the quasi-coherent ideal sheaf $\Ical$ of $X$ containing $F^{-1}\Ocal_X$ with $\Ical/F^{-1}\Ocal_X=i_*\overline{\Ical}$, i.e.\ the pullback
			\[
			\begin{tikzcd}[row sep=1em]
				\Ical & i_*\overline{\Ical} \\
				\Ocal_X & i_*\Ocal_{\overline{X}}\rlap{ .}
				\arrow[from=1-1, to=2-1, phantom, sloped, "\subseteq"]
				\arrow[from=1-2, to=2-2, phantom, sloped, "\subseteq"]
				\arrow[from=1-1, to=1-2]
				\arrow[from=2-1, to=2-2]
				\arrow["\lrcorner"{anchor=center, pos=0.125}, draw=none, from=1-1, to=2-2]
			\end{tikzcd}
			\]
			Define 
			\[
			(Y,F^*):=\Bl_\Ical(X,F^*) 
			\]
			the filtered blow-up of $(X,F^*)$ along $\Ical$ and denote by $\pi$ the canonical filtered morphism to $(X,F^*)$.
			
			To show \ref{item: lift2} in the requirement of a lift note that
			\begin{align*}
				F^0\Acal/F^{-1}\Acal &= \oplus_j \Ical^j/ \oplus_j (\Ical^j\cap F^{-1}\Ocal_X) \\
				&= \oplus_j (\Ical^j/ \Ical^j\cap F^{-1}\Ocal_X) \\
				&= \oplus_j (\Ical^j+ F^{-1}\Ocal_X)/ F^{-1}\Ocal_X \\
				&= i_*(\oplus_j \overline{\Ical}^j)
			\end{align*}
			as $\mathcal{O}_X$-modules (where we sneakily used that $i_*$ commutes with direct sums).
			Therefore, $\VV_Y(F^{-1}\Ocal_Y)$ is exactly 
			\[
			\rProj_X\left(F^0\Acal/F^{-1}\Acal\right)= \rProj_X\left(i_*(\oplus_j \overline{\Ical}^j)\right)
			\]		
			but we have\footnote{
				As $i:\overline{X}\hookrightarrow X$ is affine, the subscript in the relative Proj only changes the scheme over which we view the Proj, not the actual scheme itself. 
				To see this, look at the construction; the graded algebras that get glued are the same.
				Moreover, here specifically we even have
				\[
				\rProj_X\left(i_*(\oplus_j \overline{\Ical}^j)\right) = \rProj_X\left(i_*(\oplus_j \overline{\Ical}^j)\right)\times_X \overline{X} = \rProj_{\overline{X}}\left(i^*i_*(\oplus_j \overline{\Ical}^j)\right) = \rProj_{\overline{X}}\left(\oplus_j \overline{\Ical}^j\right),
				\]
				where the first identification follows from the fact that $\rProj_X\left(i_*(\oplus_j \overline{\Ical}^j)\right)\to X$ factors through the monomorphism $\overline{X}\hookrightarrow X$ (as the ideal $F^{-1}\Ocal_X$ gets mapped to zero).
			}
			\[
			\rProj_X\left(i_* (\oplus_j \overline{\Ical}^j)\right) = \rProj_{\overline{X}}\left(\oplus_j \overline{\Ical}^j\right) =	\overline{Y}.
			\]	
			Furthermore, by looking at affine opens of $X$, one sees that the diagram
			\[
			\begin{tikzcd}[row sep = 0.7em]
				{Y=\rProj_X\left(\oplus_j \Ical^j\right)} & X \\[1.3em]
				{\rProj_X\left(i_* (\oplus_j \overline{\Ical}^j)\right)} \\
				{\overline{Y}=\rProj_{\overline{X}}\left(\oplus_j \overline{\Ical}^j\right)} & {\overline{X}}
				\arrow["{\overline{\pi}}", from=3-1, to=3-2]
				\arrow["i"', hook, from=3-2, to=1-2]
				\arrow[from=3-1, to=2-1, phantom, sloped, "="]
				\arrow[hook, from=2-1, to=1-1]
				\arrow["\pi", from=1-1, to=1-2]
			\end{tikzcd}
			\]
			commutes, showing that \ref{item: lift3} in the requirement of a lift is fulfilled.
		\end{proof}
		
		\begin{corollary}\label{cor: existence filt resolution}
			Let $(X,F^*)$ be a filtered scheme such that $\overline{X} := \VV_X(F^{-1}\Ocal_X)$ admits a resolution of singularities $\overline{\pi}:\overline{Y}\to\overline{X}$ by a blow-up.
			
			Then $\overline{\pi}$ can be lifted to a morphism of filtered schemes $\pi:(Y,F^*)\to (X,F^*)$.
			In particular, $\VV_Y(F^{-1}\Ocal_Y) = \overline{Y}$ is smooth.
		\end{corollary}
		\begin{remark}
			If $\overline{X}$ is a quasi-projective variety, a proper birational morphism being a blow-up is equivalent to it being projective. 
			Moreover, as the composition of a blow-up is a blow-up (under some mild hypotheses \cite[\href{https://stacks.math.columbia.edu/tag/080B}{Lemma 080B}]{stacks-project}), it follows from Hironaka \cite{Hironaka} that every reduced separated scheme of finite type over a field of characteristic zero can be resolved by a blow-up.
		\end{remark}
		
	\subsection{Filtered nonrational loci}\label{subsec: filt nrl}
		Let $(X,F^*)$ be a filtered scheme.
		A morphism of filtered modules $f:(\Mcal,F^*)\to(\Ncal,F^*)$ is called \emph{strict} if $f(F^i\Mcal)=f(\Mcal)\cap F^i\Ncal$ for all $i\in\ZZ$.
		It is a \emph{monomorphism}, respectively, \emph{epimorphism} in $\Filt(X,F^*)$ when $\Mcal\to\Ncal$ is so in $\Mod(X)$, equivalently the kernel, respectively, cokernel of this map is zero.
		In general only the strict monomorphisms/epimorphisms get mapped by $\iota$ from Proposition \ref{prop: Rees gives abelian hull} to monomorphisms/epimorphisms in $\grMod(\tOcal_{X})$.
		
		An ideal $\Ical\subseteq \Ocal_X$ is always endowed with the filtration induced from $(\Ocal_X,F^*)$, i.e. $F^i\Ical := \Ical\cap F^i\Ocal_X$.
		Another way of stating this is that we require the inclusion $(\Ical,F^*)\hookrightarrow(\Ocal_X, F^*)$ to be a strict monomorphism. 
		We sometimes refer to these as \emph{filtered ideal sheaves}, but as there is no choice in the filtration there is no difference between a filtered ideal sheaf and an ideal sheaf.
		
		Furthermore, to any ideal $\Ical\subseteq \Ocal_X$ we can associate a \emph{filtered closed subscheme} of $(X,F^*)$.
		This is done by considering the usual closed subscheme $i:\VV_X(\Ical)\hookrightarrow X$ endowed with the filtration induced from $(X,F^*)$, i.e.\ $  F^j\Ocal_{\VV_X(\Ical)} := i^{-1}( (F^j\Ocal_X + \Ical) / \Ical )$.
		This is exactly requiring $(\Ocal_X,F^*)\to i_*(\Ocal_{\VV_X(\Ical)},F^*)$ to be a strict epimorphism.
		
		As $(F^j\Ocal_X + \Ical) / \Ical = F^j\Ocal_X / \Ical\cap F^j\Ocal_X$ the above two constructions are nicely compatible.
		Further evidence that these are the `correct' things to do is given by the following lemma, which is proven in exactly the same way as in the non-filtered case.
		
		\begin{lemma}
			Let $(X,F^*)$ be a filtered scheme. 
			There is a one-to-one correspondence between quasi-coherent filtered ideal sheaves $(\Ical,F^*)$ of $(\Ocal_X, F^*)$ and filtered closed subschemes\:\!\footnote{
				The map $i$ is the inclusion of a closed subset $Y$ of $X$ and the morphism of structure sheaves $i^\sharp:(X,F^*)\to i_*(Y,F^*)$ is a strict epimorphism.} 
			$i:(Y, F^*)\hookrightarrow(X, F^*)$. More precisely,
			\[
			(\Ical,F^*)\mapsto (Z,F^*),
			\]
			where $Z$ is the support of $\Ocal_X/\Ical$ and $(\Ocal_Z,F^*)$ is the sheaf of filtered rings on Z corresponding to $(\Ocal_X,F^*)/(\Ical,F^*)$, and
			\[
			\ker (i^\sharp) \mapsfrom  (Y, F^*),
			\]
			where $i^\sharp$ is the morphism $(\Ocal_X,F^*)\to i_*(\Ocal_Y,F^*)$ of structure sheaves.
		\end{lemma}
		\begin{remark}
			If we allowed non-strict monomorphisms $(\Ical,F^*)\hookrightarrow(\Ocal_X, F^*)$ as filtered ideals, we would not obtain this bijection. 
			To go from a closed subscheme to the associated ideal, one takes a kernel and kernels are always strict monomorphisms.
			Similarly, if, for a filtered closed subscheme, one wants $i_*(\Ocal_Y, F^*)$ to be the cokernel of $(\Ical,F^*)\hookrightarrow(\Ocal_X,F^*)$, one is lead to require $(\Ocal_X,F^*)\to i_*(\Ocal_Z, F^*)$ to be a strict epimorphism.
		\end{remark} 
		
		The following is the natural generalisation of the definition of a nonrational locus of \cite[Definition 6.1]{KuznetsovLunts} to the filtered setting (but we choose to state it using the ideal sheaf instead of the corresponding closed subscheme).
		\begin{definition}
			Let $f:(Y,F^*)\to(X,F^*)$ be a morphism of $n$-filtered schemes.
			A quasi-coherent ideal sheaf $\Ical\subseteq\Ocal_X$ of $X$ is called \emph{the ideal sheaf of a filtered nonrational locus of $(X,F^*)$ with respect to $f$} if the canonical morphisms\footnote{Here, the first morphism is obtained by adjunction from the morphism $f^*\Ical\to f^{-1}\Ical\cdot\Ocal_Y$ (obtained in its turn from the factorisation through the image of $f^*\Ical\to f^*\Ocal_X=\Ocal_Y$) which is compatible with the filtrations (this follows immediately from the fact that $f$ is a filtered morphism since ideals carry the induced filtration). The second morphism is the usual one which is part of the data of a derived functor.}
			\begin{equation}\label{eq: nrl}
				(\Ical,F^*)\to  f_* (f^{-1}\Ical\cdot\Ocal_Y,F^*)\to \bR f_* (f^{-1}\Ical\cdot\Ocal_Y,F^*)
			\end{equation}
			are isomorphisms.
		\end{definition}
		\begin{remark}
			Of course, we should really have written
			\[
			\iota{(\Ical,F^*)}\to  f_* \left(\iota{(f^{-1}\Ical\cdot\Ocal_Y,F^*)}\right)\to \bR f_* \left(\iota{(f^{-1}\Ical\cdot\Ocal_Y,F^*)}\right),
			\]
			where $\iota$ denotes both inclusions $\Filt^n\hookrightarrow\grMod^n$, see Proposition \ref{prop: Rees gives abelian hull}.
			As this looks even more unsightly we refrain from doing this unless confusion can arise.
		\end{remark}
		\begin{remark}
			Justification for the name is as follows.
			Suppose $X$ is a variety over field of characteristic zero, and $f:Y\to X$ is a resolution of singularities.
			One says that $X$ has \emph{rational singularities} when the natural morphism $\Ocal_X\to\bR f_*\Ocal_Y$ is an isomorphism.
			If $S:=\VV_X(\Ical)$ is a nonrational locus for $f$, then $X\backslash S$ has rational singularities (as $\Ical|_{X\backslash S}=\Ocal_{X\backslash S}$, etc.).
		\end{remark}
		
		Suppose $\Ical$ is the ideal of a filtered nonrational locus as in the definition.
		Then, due to the following diagram:
		\[
		\begin{tikzcd}
			{\bR f_*\tOcal_Y} & {\bR f_*\tOcal_{\VV_Y(f^{-1}\Ical \cdot \Ocal_Y)}} & {{\bR f_* (f^{-1}\Ical\cdot\Ocal_Y,F^*)}}[1] & {} \\
			{\tOcal_X} & {\tOcal_{\VV_X(\Ical)}} & (\Ical,F^*)[1] & {}\rlap{ ,}
			\arrow[from=2-1, to=1-1]
			\arrow[from=1-1, to=1-2]
			\arrow[from=2-1, to=2-2]
			\arrow[from=2-2, to=1-2]
			\arrow[from=1-2, to=1-3]
			\arrow[from=1-3, to=1-4]
			\arrow[from=2-2, to=2-3]
			\arrow[from=2-3, to=2-4]
			\arrow[from=2-3, to=1-3, "\sim", sloped]
		\end{tikzcd}
		\]
		one can already guess, taking a glance at Lemma \ref{lem: t(A)=t(t(alpha))}, how this notion will lead to an acyclic square on the level of enhancements.
		We show this in Proposition \ref{prop: nr locus give acyclic square} below.
		
		It is not clear how to construct nonrational loci in general.
		Luckily \cite[Lemma 6.3]{KuznetsovLunts}, showing the existence of nonrational loci for blow-ups, extends to the filtered setting.
		
		\begin{proposition}\label{prop: existence nonrat locus}
			Let $(X,F^*)$ be a Noetherian filtered scheme of finite length.
			Suppose $f : \Bl_\Ical(X,F^*) \to (X,F^*)$ is a filtered blow-up along a coherent\:\!\footnote{This is automatic if one requires $\Ical$ to come from a filtered closed subscheme, as it will then be quasi-coherent and hence automatically coherent as we are in a Noetherian setting.} sheaf of ideals $\Ical$.
			Then, for $k\gg 0$ the ideal $\Ical^k$ is the ideal of a filtered nonrational locus of $(X,F^*)$ with respect to $f$.
		\end{proposition}
		The proof requires some preparatory work which we do now. 
		Essentially, in order to `copy-paste' the proof of \cite[Lemma 6.3]{KuznetsovLunts} we have to extend some results of Serre \cite{SerreFAC} concerning the cohomology of coherent modules over projective schemes to the filtered setting.
		
		\subsubsection{Cohomology of sheaves of graded rings}
		
		\paragraph{\textit{Non-filtered version}}
		Let $X$ be a Noetherian scheme and $\Acal=\oplus_{i}\Acal_i$ a quasi-coherent graded $\Ocal_X$-algebra\footnote{Again, we write the grading with subscripts here, see also footnote \ref{foot: lower upper index}.} such that $\Acal_0$ is a coherent $\Ocal_X$-module. 			
		Assume furthermore, for convenience, that $\Acal$ is locally generated by $\Acal_1$ as an $\Acal_0$-algebra (equivalently as $\Ocal_X$-algebra); by this we mean that the morphism $\Sym_{\Acal_0}\Acal_1\to\Acal$ induced by multiplication is surjective.
		It follows that $\Acal$ is coherent in every degree, and moreover is Noetherian\footnote{It satisfies the ascending chain condition on quasi-coherent ideal sheaves, i.e.\ the categorical definition of being a Noetherian object of $\QCoh(X,\Acal)$, or equivalently it is given by a Noetherian algebra on every affine open subset.}.
		
		The following can be made to work, with slight modifications, without the local degree one generation (essentially because some Veronese will be locally generated in degree one), see also \cite[\href{https://stacks.math.columbia.edu/tag/0BXE}{Section 0BXE}]{stacks-project}.
		However, as we are only really interested in the blow-up situation, we restrict to the locally generated in degree one setting for ease.
		
		Recall that to any quasi-coherent $\ZZ$-graded $\Acal$-module $\Mcal$ we can associate a quasi-coherent sheaf $\widetilde{\Mcal}$ over $Y:=\rProj_X \Acal$. 
		By the degree one generation assumption this commutes with tensor products\footnote{
			This uses the local generation in degree one assumption.
			If $S$ is a graded ring, $M$ and $N$ are graded $S$-modules and $f\in S$, then the natural morphism (notation of Hartshorne for example)
			\[
			M_{(f)} \otimes_{S_{(f)}} N_{(f)} \to (M \otimes_{S} N)_{(f)}
			\]
			is in general only an isomorphism if $\deg(f)=1$, see e.g.\ \cite[\href{https://stacks.math.columbia.edu/tag/01ML}{Remark 01ML}]{stacks-project} for a counterexample.
		}.
		In particular, $(\Mcal(k))^\sim\cong\widetilde{\Mcal}(k):=\widetilde{\Mcal}\otimes_{\Ocal_Y} \Ocal_Y(1)^{\otimes k}$ (where $\Ocal_Y(1):=(\Acal(1))^\sim$). 
		
		\begin{lemma}\label{lem: non-filt cohomology}
			Let $Y=\rProj_X \Acal$ be as above and denote by $f:Y\to X$ the corresponding structure morphism. 
			Let $\Mcal$ be a coherent $\ZZ$-graded $\Acal$-module.
			The canonical morphisms
			\begin{equation}\label{eq: M_k=Rf_*M(k)}
				\Mcal_k \to f_* ((\Mcal(k))^\sim)\to \bR f_* ((\Mcal(k))^\sim)
			\end{equation}
			are isomorphisms for $k\gg 0$.
		\end{lemma}
		\begin{proof}
			This is well-known, see e.g.\ \cite[\href{https://stacks.math.columbia.edu/tag/0AG6}{Lemma 0AG6}]{stacks-project} and \cite[\href{https://stacks.math.columbia.edu/tag/0AG7}{Lemma 0AG7}]{stacks-project} for proofs (in the case that $X$ is affine, from which the relative version follows).
		\end{proof}
		\begin{remark}
			We chose to twist on $X$ instead of on $Y$, which is what is commonly done, for reasons that will become clear when we look at the filtered version, see Remark \ref{rem: twisting on (X,F)}.
			Since $(\Mcal(k))^\sim\cong\widetilde{\Mcal}(k)$, as $\Acal$ is locally generated in degree one, this does not matter.
			When $\Acal$ is not locally generated in degree one, twisting on $X$ is the `correct' thing to do (the functor $\Mcal\mapsto \oplus_k f_*(\widetilde{\Mcal}(k))$ is not compatible with shifts whilst $\Mcal\mapsto \oplus_k f_*((\Mcal(k))^\sim)$ is).
		\end{remark}
		
		\paragraph{\textit{Filtered version}}
		Now, let $(X,F^*)$ be a Noetherian $n$-filtered scheme and $(\Acal,F^*)$, as in \S\ref{subsec: filt blow}, 
		an $\NN$-graded algebra object of $\Filt^n(X,F^*)$ whose filtration pieces are quasi-coherent.
		Furthermore, assume that the graded $\Ocal_X$-algebra $F^0\Acal(=\Acal)$ satisfies the same assumptions as in the non-filtered case, i.e.\  $F^0\Acal_0$ is a coherent $F^0\Ocal_X$-module and $F^0\Acal$ is locally generated by $F^0\Acal_1$ as an $F^0\Acal_0$-algebra.
		It follows that every $F^j\Acal_i$ is coherent.
		
		Let $\Rees(\Acal,F^*)=\oplus_{i,j}F^j\Acal_it^j$ be the Rees algebra of $(\Acal,F^*)$. 
		This is a sheaf of $\ZZ^2$-graded $\Ocal_X$-algebras. 
		By convention the lower index indicates the grading obtained from $F^0\Acal=\oplus_iF^0\Acal_i$ and the upper index indicates the grading obtained from the filtration $F^*$.
		We denote by $\grMod^n(\Acal,F^*)$ the full subcategory of ${\ZZ^2}$-$\grMod(\Rees(\Acal,F^*))$ consisting of $\ZZ^2$-graded modules $\Mcal=\oplus_{i,j}\Mcal^j_i$ (the same naming convention concerning gradings applies here too) over $\Rees(\Acal,F^*)$ for which $\Mcal^j_*=0$ for $j\leq -n$ and multiplication by $t$ induces an isomorphism $\Mcal^j_*\isoto\Mcal^{j+1}_*$ for $j\geq 0$.
		The shift $\Mcal(k)$ of an object in $\grMod^n(\Acal,F^*)$ shifts the lower grading, $(\Mcal(k))^j_i=\Mcal^j_{i+k}$.
		We say that an object $\Mcal$ in $\grMod^n(\Acal,F^\ast)$ is coherent if every $\Mcal^i_*$ is coherent as a module over $\Rees(\Acal,F^*)^0_*=F^0\Acal$.
		
		Let $(Y,F^*):=\rProj_X(\Acal,F^*)$ be defined as in \S\ref{subsec: filt blow}.
		Applying $^\sim$ to a coherent object $\Mcal$ in $\grMod^n(\Mcal,F^*)$, component-wise `using' the lower grading, gives an object $\widetilde{\Mcal}=(\widetilde{\Mcal^j_*})_{j\in\ZZ}$ in $\Coh^n(Y,F^*)$.
		
		\begin{lemma}\label{lem: filt cohomology}
			Let $(Y,F^*)=\rProj_X (\Acal,F^*)$ be as above and denote by $f:(Y,F^*)\to (X,F^*)$ the corresponding structure morphism. 
			Let $\Mcal$ be a coherent object in $\grMod^n(\Acal,F^*)$. 
			The canonical morphisms
			\begin{equation}\label{eq: filt M_k=Rf_*M(k)}
				\Mcal_k \to f_* ((\Mcal(k))^\sim) \to \bR f_* ((\Mcal(k))^\sim)
			\end{equation}
			are isomorphisms for $k\gg 0$.
		\end{lemma}
		\begin{proof}
			To show that the morphisms in \eqref{eq: filt M_k=Rf_*M(k)} are isomorphisms, it suffices to check this for each graded component separately, i.e.\
			\[
			\gr_{j}\Mcal_k \to \gr_{j}f_* ((\Mcal(k))^\sim) \to \gr_{j}\bR f_* ((\Mcal(k))^\sim)
			\]
			is an isomorphism for each $0\geq j>-n$.
			As there are only finitely many $j$'s for which we have to check this, the result follows, using Lemma \ref{lem: compat filtered der push with underlying der push}, from Lemma \ref{lem: non-filt cohomology} applied to the $\Mcal^j_*$'s since $\gr_{j}\Mcal_k=(\Mcal^j_*)_k$ and $\gr_{j}((\Mcal(k))^\sim)= (\Mcal^j_*(k))^\sim$.
		\end{proof}
		\begin{remark}\label{rem: twisting on (X,F)}
			Twisting on $(X,F^*)$ is necessary here as opposed to naively twisting on $(Y,F^*)$, i.e.\ applying $-\otimes_{Y|n}(\Acal(1))^\sim$.
			The latter is not compatible with taking graded components as tensoring on $(Y,F^*)$ mixes these; twisting on $(X,F^*)$ twists a different grading giving $\gr_{j}((\Mcal(k))^\sim)= (\Mcal^j_*(k))^\sim$.		
		\end{remark}
		
		\subsubsection{Proof of the proposition}
		\begin{proof}[Proof of Proposition \ref{prop: existence nonrat locus}]
			Recall that $\Bl_\Ical(X,F^*):=\rProj_X(\Acal,F^*)$ where $\Acal:=\oplus_i \Ical^i$ is the Rees algebra of $\Ical$ viewed as an $\NN$-graded algebra object of $\Filt^n(X,F^*)$ by equipping every $\Acal_i=\Ical^i$ with the filtration induced from $(\Ocal_X,F^*)$, i.e.\ $F^j\Acal_i:=\Ical^i\cap F^j\Ocal_X$.
			
			Put $\Jcal:= f^{-1}\Ical\cdot \Ocal_Y$ the inverse image ideal sheaf and note that $\Jcal= (\Ical\Acal)^\sim$ (e.g.\ this follows from the exactness of $^\sim$ and $f^*\Ical=(\Ical\otimes_{\Ocal_X} \Acal)^\sim$). 
			Moreover, we have $\Jcal^k=(f^{-1}\Ical \cdot\Ocal_Y)^k= f^{-1}\Ical^k \cdot\Ocal_Y= (\Ical^k\Acal)^\sim$.
			Observe that $\Acal(k)_{\geq 0}\subseteq \Acal$ and the filtration on $\Acal(k)_{\geq 0}$ obtained by twisting is the one induced from $\Acal$ by this inclusion (since $\Acal_{l+k}\subseteq \Acal_l\subseteq \Ocal_X$ and the filtrations on both $\Acal_{l+k}$ and $\Acal_l$ are induced from $\Ocal_X$).
			Therefore, under the identification
			\[
			(\Acal(k))^\sim = (\Acal(k)_{\geq0})^\sim = (\Ical^k\Acal)^\sim = \Jcal^k,
			\]
			the filtration on $\Acal(k)$ obtained by twisting $(\Acal,F^*)$ gives exactly the filtration on $\Jcal^k$ induced from $(\Ocal_Y,F^*)$.
			Schematically,
			\[
			\begin{tikzcd}[column sep= 0.3em]
				{(\Acal(k)_{\geq 0},F^*)} & {(\Acal,F^*)} \\
				{(\Jcal^k,F^*)} & {(\Ocal_Y,F^*)}
				\arrow["\subseteq"{description}, draw=none, from=1-1, to=1-2]
				\arrow["\subseteq"{description}, draw=none, from=2-1, to=2-2]
				\arrow[maps to, "{}^\sim", from=1-1, to=2-1]
				\arrow[maps to, "{}^\sim", from=1-2, to=2-2]
			\end{tikzcd}
			\]
			where both inclusions are strict monomorphisms (as $^\sim$ is exact it preserves strict monomorphisms of filtered objects).
			
			Applying Lemma \ref{lem: filt cohomology} with $\Mcal=\Rees(\Acal,F^*)$ gives isomorphisms 
			\[
			\Rees(\Acal,F^*)_k\to f_* ((\Rees(\Acal,F^*)(k))^\sim) \to \bR f_* ((\Rees(\Acal,F^*)(k))^\sim)
			\]
			for $k\gg 0$.
			Using that $\Rees(\Acal,F^*)_k=\iota(\Ical^k,F^*)$, that by the previous paragraph 
			\begin{align*}
				(\Rees(\Acal,F^*)(k))^\sim &=  ( ( F^j\Acal(k) )^\sim )_{j\in\ZZ} \\
				&= ( F^j\Jcal^k )_{j\in\ZZ} \\
				&= \iota(f^{-1}\Ical^k\cdot\Ocal_Y,F^*)
			\end{align*}
			and noting that the morphisms of Equation \eqref{eq: filt M_k=Rf_*M(k)} are exactly those of Equation \eqref{eq: nrl} gives the result.
		\end{proof}
		
	\subsection{Two acyclic squares}\label{subsec: two acyclic squares}
		The importance of nonrational loci and functorial refinements is that they induce acyclic squares of the enhancements.
		This will be crucial in our construction of an acyclic hypercube, and thus of a categorical resolution in the next section.
		Below we use the functorial enhancement $\Dscr:\underline{\fSch}^{\op}\to\dgcat$ (constructed in \S\ref{subsec: the enhancement}), so we restrict to separated finite length filtered schemes of finite type over $\kk$.
		In particular, the diagrams we obtain do actually strictly commute and we can apply the $t$-construction from \S\ref{subsec: Tot}.
		
		\subsubsection{Refinements}
		
		\begin{proposition}
			Let $(\id,d):(X,G^*)\rightsquigarrow(X,F^*)$ be a $d$-refinement.
			Then, the derived pullback functor
			\[
			\bL(\id,d)^*:\bD(X,F^*)\to\bD(X,G^*)
			\]
			is fully faithful.
		\end{proposition}
		\begin{proof}
			The derived pullback functor being fully faithful is equivalent to the unit of the adjunction $\bL(\id,d)^*\dashv(\id,d)_*$ being an isomorphism (in the derived category).
			The latter question is local, so we may reduce to the affine setting.
			As $\bL(\id,d)^*$ commutes with direct sums and $\{ l^n(\tOcal_{(X,F^*)}(i)) \}_{0\leq i<n}$ is then a set of compact generators that gets mapped to compact objects $\bL(\id,d)^*  (l^n(\tOcal_{(X,F^*)}(i))) =  l^{dn}(\tOcal_{(X,G^*)}(di))$, it is enough, by Lemma \ref{lem: nice functor preserving or reflecting compactness}, to show that
			\begin{multline*}
				\bD(X,F^*)( l^n(\tOcal_{(X,F^*)}(i)) ,l^n(\tOcal_{(X,F^*)}(j)) [k])\\ \to \bD(X,G^*)( \bL(\id,d)^*  (l^n(\tOcal_{(X,F^*)}(i))), \bL(\id,d)^*  (l^n(\tOcal_{(X,F^*)}(j)))[k] )
			\end{multline*}
			is bijective for $0\leq i,j<n$ and $k\in\ZZ$.
			If $k\neq 0$ both sides are zero as the objects are in addition projective, whilst for $k=0$ we have
			\begin{align*}
				\text{LHS}&=l^n(\tOcal_{(X,F^*)}(j))_{-i} = F^{j-i}\Ocal_X/F^{j-n}\Ocal_X \\ &= G^{dj-di}\Ocal_X/G^{dj-dn}\Ocal_X = l^{dn}(\tOcal_{(X,G^*)}(dj))_{-di}=\text{RHS}. \qedhere
			\end{align*}
		\end{proof}
		\begin{corollary}\label{cor: refinement give acyclic square}
			In the setting of Proposition \ref{prop: existence refinements}, i.e.\ $(\id,d):(X,G^*)\rightsquigarrow(X,F^*)$ and $(\id,d):(Y,G^*)\rightsquigarrow(Y,F^*)$ are compatible $d$-refinements, the induced square of enhancements
			\[
			\begin{tikzcd}[row sep= 1.5em, column sep= 4em]
				\Dscr(X,F^*) & \Dscr(X,G^*) \\
				\Dscr(Y,F^*) & \Dscr(Y,G^*)
				\arrow[from=2-1, to=1-1]
				\arrow[from=2-2, to=1-2]
				\arrow["{\bL(\id, d)^*}", from=1-1, to=1-2]
				\arrow["{\bL(\id, d)^*}", from=2-1, to=2-2]
			\end{tikzcd}
			\]
			is acyclic.
		\end{corollary}
		\begin{proof}
			By Lemma \ref{lem: seeing acyclic on faces} it suffices to show that any refinement 
			\[
			(\id,d):(X,G^*)\rightsquigarrow(X,F^*)
			\]
			induces a quasi fully faithful dg functor
			\[
			(\id,d)^*:\Dscr(X,F^*)\to\Dscr(X,G^*)
			\]	
			of the enhancements.
			This holds by the (2-)commutative diagram
			\[
			\begin{tikzcd}[column sep= 4em, row sep= 1.5em]
				{\bD(X,F^*)} & {\bD(X,G^*)} \\
				{\bD(\Dscr(X,F^*))} & {\bD(\Dscr(X,G^*))} \\
				{[\Dscr(X,F^*)]} & {[\Dscr(X,G^*)]}
				\arrow["{\bL(\id,d)^*}", from=1-1, to=1-2]
				\arrow["\sim", sloped, from=2-1, to=1-1]
				\arrow["\sim", sloped, from=2-2, to=1-2]
				\arrow["{\bL\!\Ind_{(\id,d)^*}}", from=2-1, to=2-2]
				\arrow[hook, from=3-1, to=2-1]
				\arrow[hook, from=3-2, to=2-2]
				\arrow["{[(\id,d)^*]}", from=3-1, to=3-2]
			\end{tikzcd}
			\]
			and the previous proposition.
		\end{proof}
		
		\subsubsection{Filtered nonrational locus}
		The following, together with Proposition \ref{prop: acyclic hypercube iff qff}, is essentially a reinterpretation of \cite[Proposition 6.5]{KuznetsovLunts}.
		\begin{proposition}\label{prop: nr locus give acyclic square}
			Let $\Ical$ be the ideal sheaf of a filtered nonrational locus for a morphism $f : (Y,F^*) \to (X,F^*)$ of separated finite length filtered schemes of finite type over a field.
			Then, the induced square of enhancements
			\[
			\begin{tikzcd}[sep= 1.5em]
				\Dscr(Y,F^*)\arrow[r] & \Dscr(\VV_Y(f^{-1}\Ical \cdot \Ocal_Y),F^*) \\
				\Dscr(X,F^*)\arrow[r]\arrow[u] & \Dscr(\VV_X(\Ical),F^*)\arrow[u]
			\end{tikzcd}
			\]
			is acyclic.
		\end{proposition}
		\begin{proof}
			For notational ease, let us denote 
			\[
			S:=\VV_X(\Ical)\quad\text{and}\quad T:=\VV_Y(f^{-1}\Ical \cdot \Ocal_Y)
			\]
			and label the morphisms as follows:
			\[
			\begin{tikzcd}[sep=1.5em]
				(Y,F^*) &  (T,F^*) \\
				(X,F^*) &  (S,F^*) \rlap{ .}
				\arrow["f"', from=1-1, to=2-1]
				\arrow["j"', from=1-2, to=1-1]
				\arrow["p", from=1-2, to=2-2]
				\arrow["i"', from=2-2, to=2-1]
			\end{tikzcd}
			\]
			
			We have to show that for any $M$, $N$ in $\Dscr(X,F^*)$ the diagram 
			\[
			\begin{tikzcd}[sep=2em]
				\Dscr(Y,F^*)(f^*M, f^*N) & \Dscr(T,F^*)((f\circ j)^*M, (f\circ j)^*N) \\
				\Dscr(X,F^*)(M, N) & \Dscr(S,F^*)(i^*M, i^*N)
				\arrow[from=1-1, to=1-2, "j^*"]
				\arrow[from=2-1, to=2-2, "i^*"]
				\arrow[from=2-1, to=1-1, "f^*"]
				\arrow[from=2-2, to=1-2, "p^*"']
			\end{tikzcd}		
			\]
			is acyclic after applying $t$. This is equivalent, by Lemma \ref{lem: t(A)=t(t(alpha))}, to 
			\begin{multline}\label{eq: t(t()->t())}
				t\left(\Dscr(X,F^*)(M, N) \xrightarrow{i^*} \Dscr(S,F^*)(i^*M, i^*N)\right)  \\ \to t\left(\Dscr(Y,F^*)(f^*M, f^*N) \xrightarrow{j^*} \Dscr(T,F^*)((f\circ j)^*M, (f\circ j)^*N)\right)
			\end{multline}
			being a quasi-isomorphism.
			We may assume that $M=\Mcal^{\bullet}$ and $N=\Ncal^{\bullet}$ are honest perfect h-flat complexes\footnote{
				By our construction of $\Dscr$ we have an \emph{on the nose} fully faithful quasi-equivalence $\Dscr(X,F^*)\to\hflatperf(X,F^*)/\hflat^{\circ}(X,F^*)$ (in fact it gives actual equality on the hom-complexes). 
				Moreover, the hom-complex in the latter identifies, through quasi-isomorphism, with $\RHom$ and this is all compatible with the pullback morphisms.
			}.
			Using adjunction and the projection formula (Lemma \ref{lem: proj formula} below) the morphism \eqref{eq: t(t()->t())} identifies with\footnote{
				This is probably most easily seen using $\RHom(-,-)=H^\bullet\RHom(-,-)=\Ext^\bullet(-,-)$ since this is a complex of vector spaces, see e.g.\ \cite[Section III.2 Proposition 4]{GelfandManin}, thereby reducing to showing the compatibility on the level of hom's in the derived category.
				
				As an aside, note that the derived pullback on $\RHom$ level can be defined via the lifted adjunction Corollary \ref{cor: adjunction pullpush on RHom} and precomposing with the counit.
			}
			\begin{multline*}
				\cone(\RHom_X(\Mcal^{\bullet}, \Ncal^{\bullet}) \to \RHom_X(\Mcal^{\bullet}, \Ncal^{\bullet}\otimes_{X\mid n}i_{*}\tOcal_S))  \\ \to\cone(\RHom_X(\Mcal^{\bullet}, \Ncal^{\bullet}\otimes_{X\mid n}\bR f_{*}\tOcal_Y) \to \RHom_X(\Mcal^{\bullet}, \Ncal^{\bullet}\otimes_{X\mid n} \bR f_{*} j_{*}\tOcal_T) ).
			\end{multline*}
			This is exactly $\RHom_X(\Mcal^{\bullet}, \Ncal^{\bullet}\otimes_{X\mid n} -)$ applied to 
			\[
			(\Ical,F^*)\to  \bR f_* (f^{-1}\Ical\cdot\Ocal_Y,F^*).
			\]
			As the latter morphism is a quasi-isomorphism by definition of a nonrational locus and the functor we applied is triangulated, the claim follows.
		\end{proof}
		
		In the proof of the following lemma we make use of `flasque filtered resolutions'.
		We briefly explain what we mean by this.
		Recall that a \emph{flasque sheaf}\index{flasque/flabby sheaf} $\Fcal$, also called a flabby sheaf, is a sheaf for which the restriction maps $\Fcal(V)\to\Fcal(U)$ along open subsets $U\subseteq V$ are surjective. 
		They have the pleasant property of having no higher sheaf cohomology.
		Thus, an exact sequence of flasque sheaves is exact as sequence of presheaves.
		Let $(\Mcal,F^*)$ be a filtered module over a filtered scheme $(X,F^*)$.
		The Godement resolution $\Mcal^{gdm}$ of $\Mcal$, defined by $\Mcal^{gdm}(U):=\prod_{p\in U} \Mcal_{p}$ over an open $U$, is naturally filtered. 
		Indeed, simply put $F^{i}(\Mcal^{gdm}):=(F^{i}\Mcal)^{gdm}$.
		As $F^{i}(\Mcal^{gdm})\cap\Mcal = F^{i}\Mcal$, the natural morphism $(\Mcal,F^*)\to(\Mcal^{gdm},F^*)$ is a strict monomorphism and thus its cokernel in $\grMod$ remains filtered.
		As a consequence, any filtered module admits a resolution by filtered flasque modules.
		
		\begin{lemma}\label{lem: proj formula}
			Let $f:(X,F^*)\to (Y,F^*)$ be a morphism of quasi-compact separated $n$-filtered schemes.
			For $\Mcal^\bullet\in \bD(X,F^*)$ and $\Ncal^{\bullet}\in\bD(Y,F^*)$ there exists a functorial isomorphism (called the \emph{projection formula}\index{projection formula})
			\[
			\bR f_*\Mcal^{\bullet} \otimes^{\bL}_{Y|n} \Ncal^{\bullet} \isoto \bR f_*( \Mcal^{\bullet} \otimes^{\bL}_{X|n} \bL f^*\Ncal^{\bullet} )
			\] 
			in $\bD(Y,F^*)$.
		\end{lemma}
		\begin{proof}
			The construction of the functorial morphism is standard abstract nonsense.
			To show it is an isomorphism we readily reduce to the case $\Ncal^{\bullet}=l^n(\tOcal_Y(i))$ for $i\geq 0$, see e.g.\ \cite[\href{https://stacks.math.columbia.edu/tag/08EU}{Lemma 08EU}]{stacks-project}. 
			In this case the required isomorphism boils down to showing that
			\begin{equation}\label{eq: [Ll,Rf_{*}]}
				\bL l^n \bR f_{*} (\Mcal^{\bullet}) = \bR f_{*} \bL l^n (\Mcal^{\bullet}).
			\end{equation}
			
			In our set-up, i.e.\ $(X,F^*)$ and $(Y,F^*)$ quasi-compact and separated, $f_{*}$ has finite cohomological dimension on $\QCoh^n(X,F^*)$ by \cite[Corollary 1.4.12]{EGAIII1}, which can be lifted to the filtered setting using Lemma \ref{lem: compat filtered der push with underlying der push}.
			Moreover, $l^n$ always has finite cohomological dimension as any quasi-coherent module has a resolution by a two term quasi-coherent filtered complex (it suffices to find a surjection from a quasi-coherent filtered module, e.g.\ take a flat quasi-coherent module as these are filtered, the kernel is then automatically filtered).
			Hence, using that both functors are consequently way-out in both directions, and by first replacing $\Mcal^\bullet$ by a complex of quasi-coherent filtered modules, we can reduce to the case where $\Mcal^{\bullet}=\Mcal$ is a quasi-coherent filtered module viewed as complex concentrated in degree zero by \cite[Chapter I Proposition 7.1]{Hartshorne}.
			
			Lastly, we may replace $\Mcal$ by a flasque filtered resolution	(subtly making use of Remark \ref{rem: D_{Qc}=D(QCoh)} as flasque modules are not quasi-coherent in general).
			As these complexes are both $f_{*}$- and $l^n$-acyclic, i.e.\ $f_*$ and $l^n$ map acyclic flasque complexes to acyclic complexes, and are preserved by $f_{*}$ and $l^n$ the result follows from the fact that $l^nf_{*}=f_{*}l^n$ on flasque sheaves (as an exact sequence of flasque sheaves is exact as sequence of presheaves).
		\end{proof}
		\begin{remark}
			`Unfortunately', the above proof requires the pushforward to have finite cohomological dimension, and hence some separated and quasi-compact assumptions.
			As is clear from the proof, the crux is finding enough complexes that are simultaneously $f_{*}$- and $l^{n}$-acyclic.
			It seems unclear how to find these in general.
		\end{remark}

\section{Categorical resolutions}\label{sec: cat res}
	This chapter is dedicated to reproving the main theorem of \cite[Theorem 1.4]{KuznetsovLunts} in the finite length filtered setting without using the strong version of Hironaka.
	\begin{theorem*}[Theorem \ref{thm: main}]
		Any separated finite length filtered scheme of finite type over a field of characteristic zero has a categorical resolution by a strongly geometric triangulated category.
		Moreover, if the filtered scheme is proper, so is the resolving category.
	\end{theorem*}
	For this we start by recalling the definition of a (strongly geometric) categorical resolution and its dg version in \S\ref{subsec: generalities}.
	Followed by this, in \S\ref{subsec: constructing hyper}, we construct the hypercube of filtered schemes which will give rise to the sought after categorical resolution.
	Finally, the present gets unwrapped in \S\ref{subsec: proof of main thm}, where we give a proof of the above theorem.
	
	\subsection{Generalities}\label{subsec: generalities}
		There are a few different definitions of categorical resolutions of singularities \cite{BondalOrlov, KuznetsovLefschetz, KuznetsovLunts, Lunts}, all of which `somewhat suitably' extend the usual geometric notion to a categorical setting.
		We use the definition of \cite{KuznetsovLunts}, extended to filtered schemes.
		Their definition is tailored towards the `big' triangulated category associated to a scheme, i.e.\ the unbounded derived category.
		
		The distinction between big and small is quite imprecise.
		Roughly speaking a triangulated category would be considered big when it has direct sums, whilst it would be considered small if it is essentially small as a category.
		One way of going from big to small is by looking at the triangulated subcategory of compact objects (or other objects subject to some suitable finiteness condition), whilst going in the converse direction could, for example, be done in an enhanced setting by taking the derived category of the enhancement.
		Of course, to make this correspond nicely extra hypotheses are necessary. 
		
		Before stating the definition of a categorical resolution, we introduce a few notions.
		Recall that every compactly generated triangulated category $\Tsf$ has direct sums by definition and that the triangulated subcategory of compact objects is denoted $\Tsf^c$.
		We say $\Tsf$ is \emph{smooth} if there exists a smooth dg category $\Ascr$ such that $\bD(\Ascr)$ is equivalent to $\Tsf$.
		Since $\Tsf^c$ is then equivalent to the homotopy category of the dg category of perfect complexes over $\Ascr$, which is also smooth, this is at least morally akin to $\Tsf^c$ having a smooth dg enhancement (and it is equivalent when $\Tsf^c$ has a unique enhancement).
		Moreover, we say $\Tsf$ is \emph{proper} if $\Tsf^c$ is Ext-finite, that is $\oplus_i\Hom_\Tsf(A,B[i])$ is finite dimensional over $\kk$ for all $A,B\in\Tsf^c$.
		When $\Tsf^c$ is enhanceable this is equivalent to the enhancement being proper.
		Of course we should mention that for a `nice enough' scheme $X$, smoothness, respectively, properness of $\bD(X)$ is equivalent to smoothness, respectively, properness of $X$, see e.g.\ \cite[Proposition 3.30 and 3.31]{Orlov}.
		Lastly, we follow Kuznetsov and Lunts and say that a smooth triangulated category is \emph{strongly geometric} if it admits a semi-orthogonal decomposition consisting of derived categories of smooth varieties.
		Our goal it then to construct strongly geometric categorical resolutions for filtered schemes in $\underline{\fSch}$.
		Let us start by defining categorical resolutions.
		
		\begin{definition}\label{def: cat res}
			A \emph{categorical resolution of a finite length filtered scheme $(X,F^*)$} is a smooth compactly generated triangulated category $\Tsf$ (in particular it has direct sums) together with an adjoint pair of triangulated functors
			\[
			\begin{tikzcd}[sep=2.5em]
				{\bD(X,F^*)} \\ {\Tsf}
				\arrow[""{name=0, anchor=center, inner sep=0}, "{\pi_*}"', bend right=45, from=2-1, to=1-1]
				\arrow[""{name=1, anchor=center, inner sep=0}, "{\pi^*}"', bend right=45, from=1-1, to=2-1]
				\arrow["\dashv"{anchor=center}, draw=none, from=1, to=0]
			\end{tikzcd}
			\]
			such that
			\begin{enumerate}
				\item\label{item: in def cat res1} $\pi_*\circ\pi^*=\id$, i.e.\ $\pi^*$ is fully faithful,
				\item\label{item: in def cat res2} $\pi_*$ commutes with arbitrary direct sums ($\pi^*$ automatically does so as it is a left adjoint),
				\item\label{item: in def cat res3} $\pi_*(\Tsf^c)\subseteq \bD^b_{\Coh}(X,F^*)$.
			\end{enumerate}
		\end{definition}
		\begin{remark}
			Condition \ref{item: in def cat res2} implies that $\pi^*(\Perf(X,F^*))\subseteq\Tsf^c$ (in fact \ref{item: in def cat res2} is equivalent to this by Lemma \ref{lem: nice functor preserving or reflecting compactness} as $\bD(X,F^*)$ is compactly generated).
			Together with condition \ref{item: in def cat res3} this implies that there is an induced categorical resolution on the `small' derived category, as defined in \cite[Definition 3.2]{KuznetsovLefschetz}, i.e.\ the functors restrict to $\pi^*:\Perf(X,F^*)\to \Tsf^c$ and $\pi_*:\Tsf^c\to\bD^b_{\Coh}(X,F^*)$. 
			
			In fact, condition \ref{item: in def cat res3} can be interpreted as `properness' of the categorical resolution.
			Note, however, that there are no birationality requirements in the definition, e.g.\ $\bD(\mathbb{P}^n_\kk)$ yields a categorical resolution of $\bD(\Spec(\kk))$.
		\end{remark}
		
		The categorical resolutions we construct are made at the enhanced level, by gluing together a hypercube of dg categories. 
		Therefore, in fact, we first construct a dg version of a categorical resolution.
		We recall its definition from \cite{KuznetsovLunts} (although we leave out the pretriangulated requirement).
		
		\begin{definition}
			A \emph{categorical dg resolution} of a dg category $\Cscr$ is a smooth dg category $\Dscr$ together with a quasi fully faithful dg functor $\pi:\Cscr\to\Dscr$.
		\end{definition}
		
		A bridge between the two definitions is given by \cite[Proposition 3.13]{KuznetsovLunts}; rephrased in the filtered setting it reads.
		\begin{proposition}\label{prop: dg catres to catres}		
			Let $(X,F^*)$ be a finite length filtered scheme and let $\pi : \Dscr(X,F^*) \to \Dscr$ be a dg resolution. 
			Put $$\pi^* := \bL\!\Ind_\pi : \bD(X,F^*) = \bD(\Dscr(X,F^*)) \to  \bD(\Dscr),$$ the derived induction functor, and 	$\pi_* := \Res_\pi : \bD(\Dscr) \to \bD(X,F^*),$ the restriction functor. 
			If $\pi_*([\Dscr]) \subseteq \bD^b_{\Coh}(X,F^*)$, then $(\bD(\Dscr),\pi^*,\pi_*)$ is a categorical resolution of $(X,F^*)$.
		\end{proposition}
		\begin{proof}
			Conditions \ref{item: in def cat res1} and \ref{item: in def cat res2} in the definition of a categorical resolution are fulfilled by Lemma \ref{lem: prop F imlies prop IndF}.
			Whilst condition \ref{item: in def cat res3} is exactly what we assumed since $\Thick([\Dscr])=\bD(\Dscr)^c$.
		\end{proof}
		\begin{remark}
			In \cite[Remark 3.14]{KuznetsovLunts} it is remarked that, in the non-filtered setting, $\pi_*([\Dscr]) \subseteq \bD^b_{\Coh}(X)$ holds automatically when $X$ is projective and $\Dscr$ is proper.
			It would be interesting to know whether this remains true in the filtered setting.
			Using results from \cite{Rouquier} it seems plausible that this remains true in the filtered setting, but we did not check the details.
		\end{remark}
		
		\subsection{Constructing the resolution}\label{subsec: constructing hyper}
			We show how to construct, starting from a filtered scheme in $\underline{\fSch}$, a hypercube of filtered schemes such that the edges adjacent to the initial filtered scheme are dg smooth.
			Applying the functorial enhancement and then gluing the punctured hypercube will give a categorical dg resolution of the enhancement of the initial filtered scheme.
			
			\subsubsection{Finitely functorial squares}
			The key observation is that we can construct `finitely functorial acyclic squares'.
			
			Let $\Isf$ be the category associated to a \emph{finite} poset with a greatest element, and consider a functor $\Isf\to n\text{-}\fSch$.
			We identify $\Isf$ with its image.
			Suppose $(X,F^*)$ is the terminal object of $\Isf$ and we are given a filtered blow-up $f:\Bl_\Ical(X,F^*)\to(X,F^*)$.
			(What we really have in mind here is, unsurprisingly, a subcategory shaped as a hypercube with $(X,F^*)$ as its highest vertex.)
			Then, we can construct a compatible system of squares in the same shape as $\Isf$, i.e.\ every object is replaced by a square.
			By combining this with Corollary \ref{cor: refinements} we have the flexibility to, in addition, pick a (specific type of) refinement of the nonrational locus of $(X,F^*)$ with respect to $f$.
			Thus, we get a finite functorial system of squares:
			\begin{equation}\label{eq: the constructed acyclic squares}
				\begin{tikzcd}
					\Bl_\Ical(X,F^*) & (\VV_{\Bl}(f^{-1}\Ical\cdot\Ocal_{\Bl}),F^*) & (\VV_{\Bl}(f^{-1}\Ical\cdot\Ocal_{\Bl}),G^*) \\
					(X,F^*)  & (\VV_X(\Ical),F^*) & (\VV_X(\Ical),G^*)
					\arrow[from=1-1, to=2-1]
					\arrow[from=1-2, to=2-2]
					\arrow[from=1-3, to=2-3]
					\arrow[from=1-2, to=1-1]
					\arrow[from=2-2, to=2-1]
					\arrow[squiggly, from=2-3, to=2-2]
					\arrow[squiggly, from=1-3, to=1-2]
				\end{tikzcd}
			\end{equation}
			where the left square is a nonrational locus square and the right square is a refinement square.
			(By Lemma \ref{lem: stacking acyclic is acyclic}, Corollary \ref{cor: refinement give acyclic square} and Proposition \ref{prop: nr locus give acyclic square} this gives an acyclic square on enhancements.)
			
			Unfortunately, guaranteeing that all the squares are acyclic is only possible after potentially replacing $\Ical$ by some power $\Ical^k$, i.e.\ by considering some finite order thickening of the nonrational locus.
			As this $k$ needs to be chosen large enough for every object in $\Isf$ separately, we need $\Isf$ to be finite.
			
			\begin{lemma}\label{lem: 'functorial' squares}
				Let $\Isf$ and $f:(Y,F^*):=\Bl_\Ical(X,F^*)\to(X,F^*)$ be as above.
				For every morphism $a:(X',F^*)\to (X,F^*)$ in $\Isf$ we have a cube
				\begin{equation}\label{eq: functorial acyclic square}
					\begin{tikzcd}[sep=.75 em]
						& (Y,F^*) && (T,G^*) \\
						(Y',F^*) && (T',G^*) \\
						& (X,F^*) && (S,G^*) \\
						(X',F^*) && (S',G^*)
						\arrow[from=1-4, to=3-4]
						\arrow[squiggly, from=1-4, to=1-2]
						\arrow[from=1-2, to=3-2, "f"{pos=0.2}]
						\arrow[squiggly, from=3-4, to=3-2]
						\arrow[from=2-3, to=4-3, crossing over]
						\arrow[squiggly, from=2-3, to=2-1, crossing over]
						\arrow[from=2-1, to=4-1, "f'"]
						\arrow[squiggly, from=4-3, to=4-1]
						\arrow[from=4-1, to=3-2, "a"]
						\arrow[from=4-3, to=3-4]
						\arrow[from=2-3, to=1-4]
						\arrow[from=2-1, to=1-2]
					\end{tikzcd}
				\end{equation}
				such that the back and front face give rise to acyclic squares, where the back face is fixed and we can pick any refinement as in Proposition \ref{prop: existence refinements} of the nonrational locus $(S,F^*)$ of $(X,F^*)$ with respect to $f$.
				
				Moreover, this construction can be made compatible with compositions.
				If we are given another morphism $b:(X'',F^*)\to (X',F^*)$ in $\Isf$, then the cube obtained from $b\circ a$ is the stacking of \eqref{eq: functorial acyclic square} and the cube obtained from $b$ using $f'$.
			\end{lemma}
			\begin{proof}
				By Corollary \ref{cor: refinements} it suffices to show this for nonrational loci squares. \
				Put $\Jcal:=a^{-1}\Ical\cdot\Ocal_{X'}$ and consider the cube
				\[
				\begin{tikzcd}[row sep=.75 em, column sep=0em]
					& \Bl_{\Ical^k}(X,F^*) && (\VV_{\Bl}(f^{-1}\Ical^k\cdot\Ocal_{\Bl}),F^*) \\
					\Bl_{\Jcal^k}(X',F^*) && (\VV_{\Bl'}((f')^{-1}\Jcal^k\cdot\Ocal_{\Bl'}),F^*) \\
					& (X,F^*) && (\VV_X(\Ical^k),F^*)\rlap{ .} \\
					(X',F^*) && (\VV_{X'}(\Jcal^k),F^*)
					\arrow[from=1-4, to=3-4]
					\arrow[hookleftarrow, from=1-2, to=1-4]
					\arrow[from=1-2, to=3-2, "f"{pos=0.2}]
					\arrow[hookleftarrow, from=3-2, to=3-4]
					\arrow[from=2-3, to=4-3, crossing over]
					\arrow[hookleftarrow, from=2-1, to=2-3, crossing over]
					\arrow[from=2-1, to=4-1, "f'"]
					\arrow[hookleftarrow, from=4-1, to=4-3]
					\arrow[from=4-1, to=3-2, "a"]
					\arrow[from=4-3, to=3-4]
					\arrow[from=2-3, to=1-4]
					\arrow[from=2-1, to=1-2]
				\end{tikzcd}
				\]
				For $k\gg0$ we may assume, by Proposition \ref{prop: existence nonrat locus}, that $\Ical^k$ and $\Jcal^k$ are the ideal sheaves of filtered nonrational loci, i.e.\ the back and front face give rise to acyclic squares.
				
				The second claim follows as $b^{-1}\Jcal^k\cdot\Ocal_{X''}=(b\circ a)^{-1}\Ical^k\cdot\Ocal_{X'}$ by picking $k$ big enough to work for every morphism in $\Isf$.
			\end{proof}
			\begin{remark}
				Observe that, when $a$ is proper, all the morphisms in the cube \eqref{eq: functorial acyclic square} are proper.
				One can, for example, see this by noting that, when forgetting the filtrations, all but the left and right faces are pullback squares and $\Bl_{\Jcal^k}(X')$ is a closed subscheme of the pullback $Bl_{\Ical^k}(X,F^*)\times_X X'$.
				This will be important later.
			\end{remark}
			
		\subsubsection{A sketch, a warm-up, a cube}\label{subsubsec: sketch}
		The acyclic hypercube will be constructed inductively.
		Here, we illustrate the idea behind the construction in the first few low-dimensional steps, as we can easily illustrate these with pictures.
		We make these pictures in the filtered scheme setting, although in reality all statements concerning acyclicity are on the dg level after considering the dg enhancements.
		This procedure will be formalised in Proposition \ref{prop: constructing hypercube}.
		However, we urge the reader to simply keep the following sketch in mind when reading the proof.
		(Essentially all that changes in the general case is that the arrows `going down' in the 3-cubes are replaced by hypercubes.)
		
		So, let $(X,F^*)$ be a filtered scheme of finite length which is separated and of finite type over a field of characteristic zero. 
		We start by considering a refinement 
		\begin{equation}\label{eq: initial refinement}
			\begin{tikzcd}[sep=1.5em]
				{(X,F^*)} & {(S_0,F_0^*)}
				\arrow[squiggly, from=1-2, to=1-1]
			\end{tikzcd}
		\end{equation}
		such that $\VV_{S_0}(F_0^{-1}\Ocal_{S_0})$ is reduced.
		It can then be resolved via a blow-up that can be lifted, i.e.\ there exists a morphism $(X_0,F_0^*) \to (S_0,F_0^*)$ with $\VV_{X_0}(F_0^{-1}\Ocal_{X_0})$ smooth.
		Picking a filtered nonrational locus $(S_{1},F_{0}^*)$ we obtain
		\[
		\begin{tikzcd}[sep=1.5em]
			(X_0,F_0^*)\arrow[d]\arrow[r, hookleftarrow]  & (T_{1},F_{0}^*)\arrow[d]\\
			(S_0,F_0^*)\arrow[r, hookleftarrow] & (S_{1},F_{0}^*)\rlap{ ,}
		\end{tikzcd}
		\]
		which we can further refine to get an acyclic 2-cube
		\begin{equation}\label{eq: initial square}
			\begin{tikzcd}[sep=1.5em]
				(X_0,F_0^*)\arrow[d] & (T_{1},F_{1}^*)\arrow[d]\arrow[l, squiggly] \\
				(S_0,F_0^*) & (S_{1},F_{1}^*)\arrow[l, squiggly]
			\end{tikzcd}
		\end{equation}
		such that moreover $\VV_{S_1}(F_1^{-1}\Ocal_{S_1})$ is reduced.
		By extending the 1-cube \eqref{eq: initial refinement} and stacking with the 2-cube \eqref{eq: initial square} we obtain an acyclic 2-cube
		\[
		\begin{tikzcd}[sep=1.5em]
			{(X,F^*)} & {(S_0,F_0^*)} & {(X_0,F_0^*)} \\
			{(S_1,F_1^*)} & {(S_1,F_1^*)} & {(T_1,F_1^*)}\rlap{ .}
			\arrow[squiggly, from=1-2, to=1-1]
			\arrow[from=1-3, to=1-2]
			\arrow[squiggly, from=2-2, to=1-2]
			\arrow[squiggly, from=2-3, to=1-3]
			\arrow[from=2-3, to=2-2]
			\arrow[squiggly, from=2-1, to=1-1]
			\arrow[equal, from=2-2, to=2-1]
		\end{tikzcd}
		\]
		Now, in the (rotated) resulting acyclic 2-cube 
		\begin{equation}\label{eq: final initial square}
			\begin{tikzcd}[sep=1.5em]
				(X_0,F_0^*)\arrow[d, squiggly] & (T_{1},F_{1}^*)\arrow[d]\arrow[l, squiggly] \\
				(X,F^*) & (S_{1},F_{1}^*)\arrow[l, squiggly]
			\end{tikzcd}
		\end{equation}
		we have that $\Dscr(X_0,F_0^*)$ is smooth as $\VV(F_0^{-1}\Ocal_{X_0})$ is smooth.
		However, as we did not use strong Hironaka, there is no reason for $\VV(F_1^{-1}\Ocal_{S_1})$ to be smooth, hence we need to replace the $(S_1,F_1^*)$ vertex. 
		As $\VV(F_1^{-1}\Ocal_{S_1})$ is reduced, we can again find a morphism $(X_1,F_1^*) \to (S_1,F_1^*)$ with $\VV_{X_1}(F_1^{-1}\Ocal_{X_1})$ smooth.
		By using the functorial squares we construct an acyclic $3$-cube which shares an edge with the $2$-cube \eqref{eq: final initial square}
		\[
		\begin{tikzcd}[sep=0.75em]
			{(X_0,F_0^*)} && {(T_1,F_1^*)} && \bullet \\
			& \bullet && \bullet \\
			{(X,F^*)} && {(S_1,F_1^*)} && {(X_1,F_1^*)}\rlap{ .} \\
			& {(S_2,F_2^*)} && {(T_2,F_2^*)}
			\arrow[from=1-3, to=3-3]
			\arrow[squiggly, from=4-2, to=3-3]
			\arrow[squiggly, from=4-4, to=3-5]
			\arrow[from=3-5, to=3-3]
			\arrow[from=4-4, to=4-2]
			\arrow[squiggly, from=2-4, to=1-5]
			\arrow[squiggly, from=2-2, to=1-3]
			\arrow[squiggly, from=3-3, to=3-1]
			\arrow[squiggly, from=1-3, to=1-1]
			\arrow[from=2-2, to=4-2, crossing over]
			\arrow[from=2-4, to=4-4, crossing over]
			\arrow[from=1-5, to=3-5]
			\arrow[from=1-5, to=1-3]
			\arrow[from=2-4, to=2-2, crossing over]
			\arrow[squiggly, from=1-1, to=3-1]
		\end{tikzcd}
		\]
		Extending the $2$-cube and stacking with the $3$-cube gives us an acyclic $3$-cube
		\begin{equation}\label{eq: cube stacking diagram}
			\begin{tikzcd}[sep=0.75em]
				& {(X_0,F_0^*)} && {(T_1,F_1^*)} && \bullet \\
				\bullet && \bullet && \bullet \\
				& {(X,F^*)} && {(S_1,F_1^*)} && {(X_1,F_1^*)}\rlap{ .} \\
				{(S_2,F_2^*)} && {(S_2,F_2^*)} && {(T_2,F_2^*)}
				\arrow[from=1-4, to=3-4]
				\arrow[squiggly, from=4-3, to=3-4]
				\arrow[squiggly, from=4-5, to=3-6]
				\arrow[from=3-6, to=3-4]
				\arrow[from=4-5, to=4-3]
				\arrow[squiggly, from=2-5, to=1-6]
				\arrow[squiggly, from=2-3, to=1-4]
				\arrow[squiggly, from=3-4, to=3-2]
				\arrow[squiggly, from=1-4, to=1-2]
				\arrow[from=2-3, to=4-3, crossing over]
				\arrow[from=2-5, to=4-5, crossing over]
				\arrow[from=1-6, to=3-6]
				\arrow[from=1-6, to=1-4]
				\arrow[from=2-5, to=2-3, crossing over]
				\arrow[squiggly, from=1-2, to=3-2]
				\arrow[from= 2-3, to=2-1, crossing over, equal]
				\arrow[from= 4-3, to=4-1, equal]
				\arrow[from= 2-1, to=1-2, squiggly]
				\arrow[from= 4-1, to=3-2, squiggly]
				\arrow[from= 2-1, to=4-1]
			\end{tikzcd}
		\end{equation}
		Now, two of the vertices $(X_0,F_0^*)$ and $(X_1,F_1^*)$ connected to $(X,F^*)$ are `good', but at the cost of making a third one $(S_2,F_2^*)$ that is `bad'
		\[
		\eqref{eq: cube stacking diagram}=\begin{tikzcd}[sep=0.75em]
			& {(X_0,F_0^*)} && \bullet \\
			\bullet && \bullet \\
			& {(X,F^*)} && {(X_1,F_1^*)}\rlap{ .} \\
			{(S_2,F_2^*)} && {(T_2,F_2^*)}
			\arrow[squiggly, from=4-3, to=3-4]
			\arrow[from=4-3, to=4-1]
			\arrow[from=2-1, to=4-1]
			\arrow[from=1-4, to=3-4]
			\arrow[squiggly, from=2-3, to=1-4]
			\arrow[squiggly, from=1-2, to=3-2]
			\arrow[squiggly, from=1-4, to=1-2]
			\arrow[squiggly, from=3-4, to=3-2]
			\arrow[squiggly, from=2-1, to=1-2]
			\arrow[squiggly, from=4-1, to=3-2]
			\arrow[from=2-3, to=4-3, crossing over]
			\arrow[from=2-3, to=2-1, crossing over]
		\end{tikzcd}
		\]
		Note that the front face only contains filtered schemes of the same length. 
		We can thus, similarly, replace every vertex in this front face by an acyclic square, thereby obtaining an acyclic 4-cube.
		Extending the 3-cube \eqref{eq: cube stacking diagram} and stacking with this 4-cube we then obtain an acyclic 4-cube.
		This time three of the four vertices connected to $(X,F^*)$ will be `good' whilst we have created a fourth `bad' vertex $(S_3,F_3^*)$.
		However, in every step the dimension of $S_i$ strictly decreases, as its complement in $S_{i-1}$ is dense.
		Hence, after a finite number of steps, $S_r$ will be a disjoint union of points.
		Thus, $(S_r,F_r^*)$ is already `good', and we obtain in this manner an acyclic hypercube as desired.
		
		\subsubsection{Constructing the acyclic hypercube}	
		We say that a hypercube of filtered schemes in $\underline{\fSch}$ is \emph{acyclic} if applying the functor $\Dscr:\underline{\fSch}^{\op}\to\dgcat$ of Proposition \ref{prop: funct enhancements} yields an acyclic hypercube of dg categories.
		Because the functor $\Dscr$ is contravariant we use the opposite order to label our hypercubes of filtered schemes, so that after applying $\Dscr$ the hypercube of dg categories has the correct labelling.
		Concretely this means that the arrows in the hypercube of filtered schemes point in the opposite direction of how they would for dg categories:
		\begin{itemize}
			\item for $\underline{\fSch}$ all arrows point towards $\varnothing$,
			\item for $\dgcat$ all arrows point away from $\varnothing$.
		\end{itemize}
		
		\begin{proposition}\label{prop: constructing hypercube}
			Let $(X,F^*)$ be a separated finite length filtered scheme of finite type over a field of characteristic zero.
			There exists an acyclic hypercube $\underline{A}_{r+1}$ of filtered schemes in $\underline{\fSch}$, with $r\leq\dim X$, such that
			\begin{enumerate}
				\item $A_\varnothing= (X,F^*)$,
				\item for all $i\in\{0,\dots,r\}$ the filtered scheme $A_i=:(X_i,F_i^*)$ has $\VV_{X_i}(F_i^{-1}\Ocal_{X_i})$ smooth,
				\item the edges of the hypercube are proper morphisms of filtered schemes.
			\end{enumerate}
		\end{proposition}
		\begin{proof}
			To begin we note that $X$, being of finite type over a field, is finite dimensional.
			For the proof we inductively construct sequences of filtered schemes  $(S_0,F_0^*),\dots,\allowbreak (S_{i-1},F_{i-1}^*)$ and $(X_0,F_0^*),\dots, (X_{i-2},F_{i-2}^*)$ together with an acyclic $i$-cube $\underline{A}(i)$ such that
			\begin{enumerate}
				\item $X=S_0\varsupsetneq S_1 \varsupsetneq \dots \varsupsetneq S_{i-1}$ are closed subschemes of strictly decreasing dimension,
				\item the $\VV_{S_j}(F_j^{-1}\Ocal_{S_j})$'s are reduced and the $\VV_{X_j}(F_j^{-1}\Ocal_{X_j})$'s are smooth,
				\item $A(i)_\varnothing=(X,F^*)$,
				\item $A(i)_{j}=(X_j,F_j^*)$ for $0\leq j<i-1$,
				\item $A(i)_{i-1}=(S_{i-1},F_{i-1}^*)$,
				\item the face determined by the subset $\{I\in[i]\mid i-1\in I \}$ consists of filtered schemes of the same length.
			\end{enumerate}
			Since by construction $\dim(S_{j})<\dim(S_{j-1})$ after a finite number of steps either
			\begin{itemize}
				\item $\VV_{S_r}(F_r^{-1}\Ocal_{S_r})$ is already smooth, so we put $(X_r,F_r^*) := (S_r,F_r^*)$,
				\item $\dim (S_r) = 0$ in which case $\VV_{S_r}(F_r^{-1}\Ocal_{S_r})$ is a disjoint union of reduced points, and hence certainly smooth.
				Again we put $(X_r,F_r^*) := (S_r,F_r^*)$.
			\end{itemize}
			Moreover, as every morphism of filtered schemes throughout the construction is proper, taking $i=r+1$ gives the desired hypercube.
			
			Let $F^{-1}\Ocal_X\subseteq\Ical\subseteq\Ocal_X$ be the ideal with $\Ical/F^{-1}\Ocal_X=\rad(\Ocal_X/F^{-1}\Ocal_X)$
			Then, as $X$ is Noetherian, there exists an integer $d$ s.t.\ $\Ical^d\subseteq F^{-1}\Ocal_X$.
			Therefore, by Proposition \ref{prop: existence refinements}, there exists a refinement $F_0^*$ with $F_0^{-1}\Ocal_X=\Ical$.
			We put $(S_0,F_0^*):=(X,F_0^*)$.
			If per chance $\VV_{S_0}(F_0^{-1}\Ocal_{S_0})$ is smooth, we put $(X_0,F_0^*):=(S_0,F_0^*)$ and are done. 
			Otherwise, as
			\[
			\Ocal_X/\Ical = (\Ocal_X/F^{-1}\Ocal_X)/(\Ical/F^{-1}\Ocal_X) = (\Ocal_X/F^{-1}\Ocal_X)/\rad(\Ocal_X/F^{-1}\Ocal_X),
			\]
			$\VV_{S_0}(F_0^{-1}\Ocal_{S_0})$ is a reduced separated scheme of finite type over a field of characteristic zero.
			Thus, it can be resolved via a blow-up, which by Corollary \ref{cor: existence filt resolution} can be lifted to a filtered morphism $f_0:(X_0,F_0^*) \to (S_0,F_0^*)$ with $\VV_{X_0}(F_0^{-1}\Ocal_{X_0})$ smooth.
			Using Proposition \ref{prop: existence nonrat locus} we find a filtered nonrational locus $(S_{1},F_{0}^*)$ and obtain a (pullback) square ($T_{1}=f_0^{-1}(S_1)$ is the scheme theoretic inverse image)
			\begin{equation}\label{eq: in constructing initial data 1}
				\begin{tikzcd}[sep=1.5em]
					(X_0,F_0^*)\arrow[d, "f_0"']\arrow[r, hookleftarrow]  & (T_{1},F_{0}^*)\arrow[d, "f_0|_{T_1}"]\\
					(S_0,F_0^*)\arrow[r, hookleftarrow] & (S_{1},F_{0}^*)\rlap{ .}
				\end{tikzcd}
			\end{equation}
			By applying Proposition \ref{prop: existence refinements} again, this time with $F_0^{-1}\Ocal_{S_1}\subseteq\Ical_{S_1}\subseteq\Ocal_{S_1}$, the ideal such that $\Ical_{S_1}/F_0^{-1}\Ocal_{S_1}=\rad(\Ocal_{S_1}/F_0^{-1}\Ocal_{S_1})$, and $\Ical_{T_1}=(f_0|_{T_1})^{-1}\Ical_{S_1}\cdot\Ocal_{T_1}$, we find refinements
			\begin{equation}\label{eq: in constructing initial data 2}
				\begin{tikzcd}[sep=1.5em]
					(T_1,F_0^*)\arrow[d] & (T_{1},F_{1}^*)\arrow[d]\arrow[l, squiggly] \\
					(S_1,F_0^*) & (S_{1},F_{1}^*)\arrow[l, squiggly]
				\end{tikzcd}
			\end{equation}
			such that $\VV_{S_1}(F_1^{-1}\Ocal_{S_1})$ is reduced.
			Stacking the squares \eqref{eq: in constructing initial data 1} and \eqref{eq: in constructing initial data 2} and using Lemma \ref{lem: stacking acyclic is acyclic}, Corollary \ref{cor: refinement give acyclic square} and Proposition \ref{prop: nr locus give acyclic square} we obtain an acyclic square
			\begin{equation}\label{eq: in constructing initial data 3}
				\begin{tikzcd}[sep=1.5em]
					(X_{0},F_{0}^*)\arrow[d] & (T_{1},F_{1}^*)\arrow[d]\arrow[l, squiggly] \\
					(S_{0},F_{0}^*) & (S_{1},F_{1}^*)\rlap{ .}\arrow[l, squiggly]
				\end{tikzcd}
			\end{equation}
			Extending the refinement 
			\[
			\begin{tikzcd}[sep=1.5em]
				{(S_0,F_0^*)} \\ {(X,F^*)}
				\arrow[squiggly, from=1-1, to=2-1]
			\end{tikzcd}
			\]
			and stacking with the previous square \eqref{eq: in constructing initial data 3} we thereby obtain, this time using Lemmas \ref{lem: stacking acyclic is acyclic} and \ref{lem: extension of acyclic is acyclic}, the required acyclic 2-cube 
			\[
			\underline{A}(2)=\begin{tikzcd}[sep=1.5em]
				(X_0,F_0^*)\arrow[d, squiggly] & (T_{1},F_{1}^*)\arrow[d]\arrow[l, squiggly] \\
				(X,F^*) & (S_{1},F_{1}^*)\rlap{ ,}\arrow[l, squiggly]
			\end{tikzcd}
			\]
			i.e.\ $A(2)_\varnothing=(X,F^*)$, $A(2)_{0}=(X_0,F_0^*)$, $A(2)_{1}=(S_{1},F_{1}^*)$ and $A(2)_{01}=(T_{1},F_{1}^*)$.
			That $\dim (S_1)< \dim(S_0)$ follows as the resolution of $S_0$ is an isomorphism over the regular locus, in particular $S_0\backslash S_1$ is dense in $S_0$.
			
			Now, suppose $\underline{A}(i)$ has been constructed. 
			If $\VV_{S_{i-1}}(F_{i-1}^{-1}\Ocal_{S_{i-1}})$ is smooth we can put $(X_{i-1},F_{i-1}^*):= (S_{i-1},F_{i-1}^*)$ and are done.
			Otherwise, consider the face ${F}$ of $\underline{A}(i)$ determined by the subset $\{I\in[i]\mid i-1\in I \}$.
			When viewing $\underline{A}(i)$ as a morphism, via Lemma \ref{lem: n-cube can be viewed as morphism of (n-1)-cubes}, $F$ is the source and the target is the `good' face containing $(X,F^*)$ and all the $(X_j,F_j^*)$'s.
			As $\VV_{S_{i-1}}(F_{i-1}^{-1}\Ocal_{S_{i-1}})$ is reduced we have, as above, a morphism $(X_{i-1},F_{i-1}^*)\to (S_{i-1},F_{i-1}^*) $ with $\VV_{X_{i-1}}(F_{i-1}^{-1}\Ocal_{X_{i-1}})$ smooth.
			Since all the filtered schemes in $F$ have the same length, we obtain, using Lemma \ref{lem: 'functorial' squares}, an $(i+1)$-cube ${B}$ by `inserting' an acyclic square at every vertex of ${F}$.
			More precisely, the vertex $(S_{i-1},F_{i-1}^*)$ is replaced by an acyclic square
			\begin{equation}\label{eq: in constructing initial data 4}
				\begin{tikzcd}[sep=1.5em]
					{(S_{i-1},F_{i-1}^*)} & {(X_{i-1},F_{i-1}^*)} \\
					{(S_{i},F_{i}^*)} & {(T_{i},F_{i}^*)}
					\arrow[from=1-2, to=1-1]
					\arrow[from=2-2, to=2-1]
					\arrow[squiggly, from=2-1, to=1-1]
					\arrow[squiggly, from=2-2, to=1-2]
				\end{tikzcd}
			\end{equation}
			where  $\dim S_i<\dim S_{i-1}$ and, by possibly refining, $\VV_{S_i}(F_i^{-1}\Ocal_{S_i})$ is reduced and every other vertex is replaced by a square of this form.
			We have ${B}$ acyclic by Lemma \ref{lem: seeing acyclic on faces}.
			Let\footnote{
				To follow this part it is easiest to think of $B$ as portrayed in the right square of the diagram \eqref{eq: A(i+1)} below, which one should think of as the square \eqref{eq: in constructing initial data 4} with every vertex replaced by an $(i-1)$-cube (all of whose filtered schemes have the same length as the filtered scheme from the respective vertex of \eqref{eq: in constructing initial data 4}).
			} 
			$G$ be the face of $B$ which, when viewed as a morphism via Lemma \ref{lem: n-cube can be viewed as morphism of (n-1)-cubes}, has $F$ as target and has $(S_{i-1},F_{i-1}^*)$ as image of $(S_{i},F_{i}^*)$, which sits in the opposing face $F'$ of $G$.
			Finally, denote by ${C}$ the extension of ${G}$ by $\underline{A}(i)$ (${G}$ and $\underline{A}(i)$ share the face ${F}$), which is acyclic by Lemma \ref{lem: extension of acyclic is acyclic}, and define $\underline{A}(i+1)$ to be the stacking of ${B}$ and ${C}$ (they share the face ${G}$), which is acyclic by Lemma \ref{lem: stacking acyclic is acyclic},
			\begin{equation}\label{eq: A(i+1)}
				\underline{A}(i+1)=
				\begin{tikzcd}[sep={5em, between origins}]
					{\text{`good'}} & F & \bullet \\
					{F'} & {F'} & \bullet\rlap{ .}
					\arrow["{\underline{A}(i)}"', to=1-1, from=1-2]
					\arrow["G"{description}, to=1-2, from=2-2]
					\arrow[to=1-1, from=2-1]
					\arrow[to=2-1, from=2-2, equal]
					\arrow["A"{description}, draw=none, to=1-1, from=2-2]
					\arrow[to=2-2, from=2-3]
					\arrow[to=1-2, from=1-3]
					\arrow[to=1-3, from=2-3]
					\arrow["B"{description}, draw=none, to=1-2, from=2-3]
				\end{tikzcd}
			\end{equation}
			In the diagram\footnote{Compare this diagram with the top view of diagram \eqref{eq: cube stacking diagram}.} \eqref{eq: A(i+1)}, the upper left vertex is the `good' face of $\underline{A}(i)$, the one containing $(X,F^*)$ and all the $(X_j,F_j^*)$ for $j<i-1$.
			The lower left vertex contains $(S_{i},F_{i}^*)$ whilst the upper right vertex contains $(X_{i-1},F_{i-1}^*)$ (the stacking exchanged $(S_{i-1},F_{i-1}^*)$ which was contained in $F$ by $(X_{i-1},F_{i-1}^*)$).
			Moreover, the face determined by the subset $\{I\in[i+1]\mid i\in I \}$ consists of filtered schemes all having the length of $F_i^*$.
			The resulting hypercube $\underline{A}(i+1)$ thus has all the required properties.
		\end{proof}
		
	\subsection{Proof of main theorem}\label{subsec: proof of main thm}
		We are finally ready to prove the main theorem.
		First, we show that we obtain a dg resolution.
		
		\begin{theorem}
			Let $(X,{}_nF^*)$ be a separated finite length filtered scheme of finite type over a field of characteristic zero.
			Then, there exists a (pretriangulated) dg category $\Dscr$, glued from several copies of filtered schemes, and a dg functor $\pi:\Dscr(X,F^*)\to\Dscr$ such that
			\begin{enumerate}
				\item\label{item: mainthm1} $(\Dscr,\pi)$ is a categorical dg resolution of $\Dscr(X,F^*)$,
				\item\label{item: mainthm2} the restriction functor $$\Res_\pi:\bD(\Dscr)\to \bD(\Dscr(X,F^*))=\bD(X,F^*)$$ maps $[\Dscr]$ into $\bD^b(\Coh^n(X,F^*))$.
			\end{enumerate}
			Moreover, if $(X,F^*)$ is proper, then so is $\Dscr$.
		\end{theorem}
		\begin{proof}
			Let $\underline{A}_{r+1}$ be the acyclic hypercube constructed in Proposition \ref{prop: constructing hypercube} and put $\underline{\Dscr}_{r+1}:=\Dscr(\underline{A}_{r+1})$.
			Next, define
			\[
			\Dscr:=\Glue(\underline{\Dscr}_{r+1}^\circ)
			\]
			and let $\pi$ denote the quasi fully faithful dg functor 
			\[
			\Dscr(X,F^*)=\Dscr_\varnothing\to \Dscr
			\]
			from Proposition \ref{prop: acyclic hypercube iff qff}.
			
			To show $\ref{item: mainthm1}$ it suffices, by Corollary \ref{cor: smoothness gluing}, to verify that restriction along 
			\[
			\Dscr_i\to\Dscr_I\quad
			\]
			for $i\in I\subseteq [r+1]$, induced from composing the edges of the hypercube, preserves compactness.
			To see this, observe that such a $V:\Dscr_i\to\Dscr_I$ is of the form $f^*: \Dscr(X_i,{}_{n_i}F_i^*)\to \Dscr(T,{}_mG^*)$ for some proper (generalised) morphism of filtered schemes with in particular $\VV_{X_i}(F_i^{-1}\Ocal_{X_i})$ smooth.
			Consequently, by adjointness, $\Res_V$ identifies with $\bR f_*$ (as $\bL\!\Ind_V$ identifies with $\bL f^*$).
			Thus, we need to show that $\bR f_*$ preserves perfect complexes, as these are exactly the compact objects by Proposition \ref{prop: compact iff perfect on nice fSch}.
			As $f$ is proper, we know by Lemma \ref{lem: compat filtered der push with underlying der push} that $\bR f_*$ preserves coherence.
			Hence, since $\VV_{X_i}(F_i^{-1}\Ocal_{X_i})$ is smooth, we conclude by Proposition \ref{prop: Perf(X,F)=Db(Coh(X,F))} that $\bR f_*$ maps $\Perf(T,G^*)\subset \bD^b(\Coh^m(T,G^*))$ into $\bD^b(\Coh^{n_i}(X_i,F_i^*))=\Perf(X_i,F_i^*)$.
			
			Showing $\ref{item: mainthm2}$ is similar.
			By Lemma \ref{lem: restriction functor restricted to component of sod} together with the first semi-orthogonal decomposition from Corollary \ref{cor: SODs Glue} (and with the  Lemma \ref{lem: bimodules right perfect} and the previous paragraph) it suffices to show that restriction along
			\[
			\quad\Dscr_\varnothing\to\Dscr_i
			\]
			preserves coherence.
			This is shown in exactly the same manner as the previous paragraph.
		\end{proof}
		
		By the above theorem, together with Proposition \ref{prop: dg catres to catres}, we immediately obtain our desired categorical resolution.
		The fact that it is strongly geometric follows from Corollary \ref{cor: SODs Glue} and the semi-orthogonal decomposition of the derived category of a finite length filtered scheme from Proposition \ref{prop: filtered SODs}.
		Indeed, let $(X_i,{}_{n_i}F_i^*)$ be as in the theorem and put $\overline{X}_i:=\VV_{X_i}(F^{-1}\Ocal_{X_i})$.
		Then, we have a semi-orthogonal decomposition
		\begin{align*}
			\bD(\Dscr)&=\langle \bD(X_0,F_0^*),\bD(X_1,F_1^*),\dots, \bD(X_r,F_r^*)  \rangle \\
			&= \langle \underbrace{\bD(\overline{X}_0),\dots,\bD(\overline{X}_0)}_{n_0\text{ components}},\underbrace{\bD(\overline{X}_1),\dots,\bD(\overline{X}_1)}_{n_1\text{ components}}, \dots, \underbrace{\bD(\overline{X}_r),\dots,\bD(\overline{X}_r)}_{n_r\text{ components}} \rangle,
		\end{align*}
		where $r\leq \dim X$ and the components are derived categories of smooth (potentially nonconnected) varieties.
		Altogether we have obtained the desired result.
		
		\begin{theorem}\label{thm: main}
			Any separated finite length filtered scheme of finite type over a field of characteristic zero has a categorical resolution by a strongly geometric triangulated category.
			Moreover, if the filtered scheme is proper, so is the resolving category.
		\end{theorem}
		
		\begin{remark}
			Similarly as in \cite[\S6.5]{KuznetsovLunts} one can argue that this resolution is birational in some sense; essentially because the entire construction is relative over the initial filtered scheme $(X,F^*)$.
			
			For any open $U\subseteq X$ we can restrict the acyclic hypercube $\underline{A}_{r+1}$, constructed in Proposition \ref{prop: constructing hypercube}, to obtain a hypercube with $\varnothing$-vertex $(U,F^*)$ (replace every other vertex by its restriction to the inverse image of $U$).
			Applying $\Dscr$ and then gluing gives a categorical resolution $\Dscr_U$ of $\Dscr(U,F^*)$.
			The mapping $U\mapsto \Dscr_U$ defines a presheaf of dg categories on $X$.
			
			Now, suppose that moreover $\VV_X(F^{-1}X_0)$ is reduced, so that we can take $F_0^*=F^*$ in the construction.
			Then for $U$ small enough, contained in the complement of $S_1$, $\Dscr_U\cong\Dscr(U,F^*)$, giving `birationality'. 
			(Let $f_0:X_0\to X$ denote the underlying morphism in the hypercube, then all the vertices in the hypercube are zero except for $\Dscr(f_0^{-1}(U),F_0^*)\cong \Dscr(U,F_0^*) = \Dscr(U,F_0^*)$.)
		\end{remark}

\appendix

\section{More on the smoothness of directed dg categories}\label{Asec: more on smoothness}
	We briefly mention necessary conditions for the smoothness of directed dg categories, at least for small $n$.
	
	In Proposition \ref{prop: smoothness directed dg cat} the bimodules are required to be \emph{right} perfect, instead of merely perfect as one would expect by naively trying to generalise Theorem \ref{thm: LS}.
	Undoubtedly our assumption is too strong and one could get by with less.
	However, simply requiring perfectness of the bimodules will not be enough, as can be expected from Proposition \ref{prop: perf iff}, see also Remark \ref{rem: necessity perf bimod}.
	In general there will be conditions involving the dg bimodule morphisms $\phi_{ik}\otimes_{\Ascr_k} \phi_{kj}\to \phi_{ij}$ of the directed dg category, and thus simply requiring the bimodules to be perfect will not suffice.
	In effect the right perfectness assumption allows us to avoid taking these morphisms into consideration.
	Moreover, as Example \ref{ex: smooth n=3 non-perfect bimodule} below shows, requiring the dg bimodules to be perfect is in general not necessary.
	A possible explanation for the dichotomy between the cases $n=2$ and $n>2$ is the fact that in the former case the directed dg category is a tensor dg category, which does not hold for $n>2$ in general.
	So, it is not unreasonable to expect a difference in behaviour for both cases.
	
	As an illustration, for the $n=3$ case we obtain the following sufficient and necessary conditions (the cases $n>3$ become slightly more complicated for increasing $n$, see Remark \ref{rem: n>3}).
	
	\begin{proposition}
		Let 
		\[
		\Cscr=
		\begin{pmatrix}
			\Ascr_0 & 0 & 0\\
			\phi_{10} & \Ascr_1 & 0\\
			\phi_{20} & \phi_{21} & \Ascr_2
		\end{pmatrix}.
		\] 
		Then, $\Cscr$ is smooth if and only if the $\Ascr_i$'s are smooth and $\phi_{10}$, $\phi_{21}$ and $\cone(\phi_{21}\otimes^{\bL}_{\Ascr_1}\phi_{10}\to\phi_{20})$ are perfect.
	\end{proposition}
	\begin{proof}
		It suffices to determine when $\left(\begin{smallmatrix}	\phi_{20} & \phi_{21} \end{smallmatrix}\right)$ is perfect as $\left(\begin{smallmatrix}
			\Ascr_0\otimes_\kk \Ascr_2^{\op} & 0\\
			\phi_{10}\otimes_\kk \Ascr_2^{\op} & \Ascr_1\otimes_\kk \Ascr_2^{\op}
		\end{smallmatrix} \right)$-module (with the obvious dg module structure).
		This is exactly the content of Proposition \ref{prop: perf iff}, using 
		\[
		\phi_{21}\otimes^{\bL}_{\Ascr_1\otimes_\kk \Ascr_2^{\op}}(\phi_{10}\otimes_\kk\Ascr_2^{\op}) = \phi_{21}\otimes^{\bL}_{\Ascr_1}\phi_{10}.
		\]
		(Which can be seen by taking an h-projective resolution of $\phi_{21}$ and noting that h-projective bimodules are right h-projective as we work over a field.)
	\end{proof}
	\begin{remark}\label{rem: necessity perf bimod}
		There is no reason for $\phi_{21}\otimes^{\bL}_{\Ascr_1}\phi_{10}$ to be perfect when $\phi_{21}$ and $\phi_{10}$ are merely perfect.
		This would hold if $\phi_{21}$ is left perfect or $\phi_{10}$ is right perfect.
		However, in general for a dg $(\Ascr,\Bscr)$-bimodule perfectness implies left (respectively, right) perfectness only when $\Bscr$ (respectively, $\Ascr$) is proper (as can be seen by letting $\Ascr$ (respectively, $\Bscr$) equal $\kk$).
	\end{remark}
	
	\begin{example}\label{ex: smooth n=3 non-perfect bimodule}
		Let $A$, $B$ and $C$ be $\kk$-algebras which we view as dg categories concentrated in degree zero with one object.
		Then, it immediately follows from the previous proposition that
		\[
		\begin{pmatrix}
			A & 0 & 0\\
			B\otimes_\kk A & B & 0\\
			C\otimes_\kk B\otimes_\kk A & C \otimes_\kk B & C
		\end{pmatrix}
		\] 
		is smooth if $A$, $B$ and $C$ are smooth. 
		However, when $B$ is infinite dimensional as a $\kk$-vector space, $C\otimes_\kk B\otimes_\kk A$ is not a perfect $(A,C)$-bimodule.
	\end{example}
	
	To finish we say something about the $n>3$ case.		
	\begin{remark}\label{rem: n>3}
		Writing down the necessary and sufficient conditions for smoothness explicitly for general $n$ becomes a bit `tricky', but it will be clear from the $n=4$ case below what one should expect.
		The reason is that an inductive proof leads to conditions involving taking cones of cones, of cones, etc. 
		Luckily these iterated cones can be packaged nicely, using the $t$-construction from \S\ref{subsec: Tot}.
		
		As an illustration we write out the $n=4$ case.		
		Let 
		\[
		\Cscr=
		\begin{pmatrix}
			\Ascr_0 & 0 & 0 & 0\\
			\phi_{10} & \Ascr_1 & 0 & 0\\
			\phi_{20} & \phi_{21} & \Ascr_2 & 0\\
			\phi_{30} & \phi_{31} & \phi_{32} & \Ascr_3
		\end{pmatrix}.
		\] 
		Then, $\Cscr$ is smooth if and only if 
		\begin{itemize}
			\item $\Ascr_0$, \dots, $\Ascr_3$ are smooth,
			\item $\phi_{10}$, $\phi_{21}$ and $\phi_{32}$ are perfect,
			\item $\cone( \phi_{21}\otimes^{\bL}_{\Ascr_1} \phi_{10} \to \phi_{20} )$ and $\cone( \phi_{32}\otimes^{\bL}_{\Ascr_2} \phi_{21} \to \phi_{31} )$ are perfect,
			\item $\cone(  \cone(\phi_{32}\otimes^{\bL}_{\Ascr_2} \phi_{21}\to \phi_{31})\otimes^{\bL}_{\Ascr_1} \phi_{10}  \to \cone(\phi_{32}\otimes^{\bL}_{\Ascr_2} \phi_{20} \to \phi_{30}) )$ is perfect.
		\end{itemize}
		Rewriting these conditions a bit, getting rid of all the cones, one finds a more symmetric description.
		Namely, $\Cscr$ is smooth if and only if
		\begin{itemize}
			\item $\Ascr_0$, \dots, $\Ascr_3$ are smooth,
			\item $\phi_{10}$, $\phi_{21}$ and $\phi_{32}$ are perfect,
			\item $t( \phi_{21}\otimes^{\bL}_{\Ascr_1} \phi_{10} \to \phi_{20} )$ and $t( \phi_{32}\otimes^{\bL}_{\Ascr_2} \phi_{21} \to \phi_{31} )$ are perfect,
			\item 
			\[
			t\left(\begin{tikzcd}[sep=small]
				\phi_{32}\otimes^{\bL}_{\Ascr_2} \phi_{21}\otimes^{\bL}_{\Ascr_1} \phi_{10} \ar[d]\ar[r] & \phi_{32}\otimes^{\bL}_{\Ascr_2} \phi_{20}\ar[d] \\
				\phi_{31}\otimes^{\bL}_{\Ascr_1} \phi_{10} \ar[r] & \phi_{30}
			\end{tikzcd}\right)
			\]
			is perfect.
		\end{itemize}
		(There are some subtleties concerning the fact that we have derived tensor products, which we have swept under the rug.
		It suffices to pick resolutions for $\phi_{32}$ and $\phi_{10}$, then everything can be expressed with honest morphisms of complexes.
		In general choosing resolutions, to compute the derived tensor products, for all the $\phi$'s in a compatible way seems tricky.)
	\end{remark}

\section{An alternative definition of the gluing}\label{Asec: alternative gluing}
	Here, we give an alternative definition of the gluing, more in line with how the gluing of two dg categories is defined in \cite{KuznetsovLunts}.
	As it makes sense for arbitrary directed dg categories, we do it in this setting.
	To obtain the gluing of a hypercube, we then apply the construction below to the generalised arrow dg category.
	This construction is only well-suited when the vertices of the hypercube are pretriangulated.
	
	Let $\Cscr$ be a directed dg category as in Equation \eqref{eq: lower matrix dg cat}.
	We define $\Glue'(\Cscr)$ to be the full subcategory of the twisted complexes over $\Cscr$ consisting of objects of the form
	\begin{equation}\label{eq: ob in glue'}
		\left( \oplus_{i=0}^{n-1}M_{n-1-i}[i], \mu \right),
	\end{equation}
	where $M_j \in \Ascr_j$, and $\mu$ is a strictly upper triangular matrix satisfying the usual condition for a twisted complex.
	For notational ease, we will not denote these objects by \eqref{eq: ob in glue'}, rather we opt for notation that leaves out the $\oplus$ and $[i]$'s.
	We write the objects as
	\[
	\left( (M_i), (\mu_{ij}) \right)
	\]
	with $M_j \in \Ascr_j$ and\footnote{We altered the naming convention for morphisms, compared to \S\ref{subsec: twist}, writing $\mu_{ij}:M_i\to M_j$ instead of $\mu_{ji}:M_i\to M_j$.} $\mu_{ij}\in\Cscr^{i-j+1}(M_i,M_j)$ where the hom-complex of $\Cscr$ carries an extra $(-1)^{n-1-j}$ in the differential due to the shifts in the twisted complex.
	Note that this minus sign reflects the fact that $\Ascr_j$ sits in $\Glue'(\Cscr)$ with a twist $[n-1-j]$.
	
	For convenience we make some aspects of this category more explicit.
	\begin{itemize}
		\item The $(\mu_{ij})$ satisfy the following relations
		\begin{align}
			\mu_{ij}&=0\quad\text{for } i\geq j,\notag \\
			(-1)^{n-1-j}d\mu_{ij}+\sum_k\mu_{kj}\mu_{ik}&=0\quad\text{for all } i\leq j,\label{eq: mu condition}
		\end{align}	
		where the differential and composition are those of $\Cscr$ (i.e.\ we already took into account the minus signs from the shifts in the twisted complex).
		\item  The morphism complexes are given by
		\[
		\Hom( \left( (M_i), (\mu_{ij})\right), \left( (N_i), (\nu_{ij}) \right) )=\oplus_{i,j} \Cscr(M_i, N_j)[i-j]
		\]
		with differential
		\[
		(d(f_{ij})) = \left( (-1)^{n-1-j}df_{ij} + \sum_k \left(\nu_{kj}f_{ik} - (-1)^{|f|} f_{kj}\mu_{ik}\right) \right)
		\] 
		for $f=(f_{ij}):\left( (M_i), (\mu_{ij})	\right) \to \left( (N_i), (\nu_{ij})	\right)$ of degree $|f| (=|f_{ii}|$ for all $i$)
		(here the differential and composition are again those of $\Cscr$).
		\item 	Composition of morphisms is matrix multiplication
		\[
		(f_{ij})(g_{ij}) = \left(\sum_k f_{kj}g_{ik}\right).
		\]
	\end{itemize} 
	\begin{remark}
		The ordering in our objects \eqref{eq: ob in glue'} comes from the ordering of the morphisms in $\Cscr$ (which is chosen so that the indices in the resulting semi-orthogonal decomposition increase from left to right).
		If we considered the subcategory with objects $(\oplus_{i=0}^{n-1}M_{i}[i], \mu)$, $\mu$ would always equal zero as there are no arrows between the $M_i$ when $i$ increases.
		Then, the gluing category would just be $\oplus_i\Ascr_i$, which is clearly not what we want.
	\end{remark}
	
	The following lemma mimics \cite[Lemma 4.3]{KuznetsovLunts} and shows that $\Glue'(\Cscr)$ is pretriangulated whenever the $\Ascr_k$'s are pretriangulated, and thus $\Glue'(\Cscr)$ gives a much smaller (and more hands-on) pretriangulated hull compared to $\tw(\Cscr)$.
	\begin{lemma}
		If the $\Ascr_i$'s are pretriangulated, respectively, strongly pretriangulated, then so is $\Glue'(\Cscr)$.
	\end{lemma}
	\begin{proof}
		The same proof as \cite[Lemma 4.3]{KuznetsovLunts} works in this generality, although it is more tedious to check all the details.
		We give a sketch.
		
		By a suitable generalisation\footnote{
			One needs a neat trick for this. 
			Namely, suppose $(f_{ij})$ is a closed degree zero morphism with every $f_{ii}$ a homotopy equivalence.
			Then, $(f_{ij})$ is itself a homotopy equivalence. 
			To see this, note that, using the natural filtration on twisted complexes, $(f_{ij})$ can be written as an iterated extension of the $f_{ii}$'s in the homotopy category. 
		} of \cite[Proposition 4.14]{KuznetsovLunts} we can replace $\Ascr_k$ by $\tw(\Ascr_k)$, thereby reducing to the strongly pretriangulated case.
		Thus, we henceforth assume that the $\Ascr_k$ are strongly pretriangulated.
		
		It is clear that $\Glue'(\Cscr)$ is closed under shifts; the shift of $\left( (M_i), (\mu_{ij})	\right)$ is represented by $\left( (M_i[1]), (-\mu_{ij})	\right)$.
		So, it suffices to show that it is closed under cones. 
		Thus, let $f=(f_{ij}):\left( (M_i), (\mu_{ij})	\right) \to \left( (N_i), (\nu_{ij})	\right)$ be a closed degree zero morphism. 
		As $(-1)^{n-1-k}f_{kk}:M_k\to N_k$ is a closed degree zero morphism of $\Ascr_k$, which is strongly pretriangulated, we can take its cone $C_k$.
		This comes with degree zero morphisms
		\[
		M_k[1]\xrightarrow{i_k}C_k\xrightarrow{p_k}M_k[1],\quad N_k\xrightarrow{j_k}C_k\xrightarrow{s_k}N_k,
		\]  
		satisfying
		\begin{gather*}
			p_ki_k=\id_{M_k[1]},\quad s_kj_k=\id_{N_k},\quad p_kj_k=0,\quad s_ki_k=0,\quad i_kp_k+j_ks_k=\id_{C_k},\\	d(j_k)=d(p_k)=0,\quad d(i_k)=(-1)^{n-1-k}j_kf_{kk}\epsilon,\quad d(s_k)=-(-1)^{n-1-k}f_{kk}\epsilon p_k,
		\end{gather*}
		where $\epsilon:M[1]\to M$ is a closed degree one isomorphism.
		Put
		\[
		\gamma_{ij}= - i_j\epsilon^{-1}\mu_{ij}\epsilon p_i + j_j\nu_{ij}s_i + j_jf_{ij}\epsilon p_i\in\Cscr(C_i,C_j)\quad\text{for }i<j.
		\]
		Then, $(C_i, \gamma_{ij})$ is the cone of $f$.
		For this one needs to check two things
		\begin{enumerate}
			\item the $\gamma$'s satisfy the relation \eqref{eq: mu condition},
			\item collecting the $i_k$'s, $p_k$'s, $j_k$'s and $s_k$'s together (as diagonal matrices) into morphisms $i$, $p$, $j$ and $s$, these satisfy the required relations defining a cone intrinsically in the dg category $\Glue'(\Cscr)$.
		\end{enumerate}
		We leave this to the motivated reader. 
	\end{proof}
		
\bibliographystyle{amsalpha}
\bibliography{bibliography}
\end{document}